\documentclass{lms}

\usepackage{amssymb}
\usepackage{amsmath}
\usepackage{nicefrac}  
\usepackage{graphicx, psfrag}
\usepackage[backref=page]{hyperref} 
\usepackage[all]{hypcap} 

	\newtheorem{ithe}{Theorem}            
	\newtheorem{theo}{Theorem}[section]	  
	\newtheorem{lemm}[theo]{Lemma}
	\newtheorem{coro}[theo]{Corollary}
	\newtheorem{prop}[theo]{Proposition}
	
	\newtheorem{prde}[theo]{Proposition/Definition}
\newunnumbered{clai}{Claim} 
\newnumbered{ques}[theo]{Question}
\newnumbered{rema}[theo]{Remark}

\numberwithin{equation}{section}

\renewcommand*{\backref}[1]{}
\renewcommand*{\backrefalt}[4]{\tiny 
    \ifcase #1 (Not cited.)%
    \or        (Cited on page~#2.)%
    \else      (Cited on pages~#2.)%
    \fi}

   \def\DD{{\mathbb D}}

 \def\NN{{\mathbb N}}  
 \def\RR{{\mathbb R}}  
   
 \def\ZZ{{\mathbb Z}}

\def\La{\Lambda}
\def\la{\lambda}

\def\La{\Lambda}

\newcommand{\cA}{\mathcal{A}}\newcommand{\cB}{\mathcal{B}}

\newcommand{\cG}{\mathcal{G}}

\newcommand{\cK}{\mathcal{K}}

\newcommand{\cO}{\mathcal{O}}\newcommand{\cP}{\mathcal{P}}
\newcommand{\cR}{\mathcal{R}}
\newcommand{\cS}{\mathcal{S}}\newcommand{\cT}{\mathcal{T}}
\newcommand{\cU}{\mathcal{U}}\newcommand{\cV}{\mathcal{V}}
\newcommand{\cW}{\mathcal{W}}\newcommand{\cX}{\mathcal{X}}
\newcommand{\cZ}{\mathcal{Z}}

\def\Diff{\operatorname{Diff}}
\def\Int{\operatorname{Int}}
\def\dim{\operatorname{dim}}

\def\orb{\operatorname{orb}}
\def\supp{\operatorname{supp}}

\newcommand{\R}{\mathbb{R}}
\newcommand{\C}{\mathbb{C}}
\newcommand{\eps}{\varepsilon}
\renewcommand{\setminus}{\smallsetminus}
\renewcommand{\emptyset}{\varnothing}
\renewcommand{\angle}{\measuredangle}
\newcommand{\SL}{\mathrm{SL}}
\newcommand{\GL}{\mathrm{GL}}

\newcommand{\Id}{\mathrm{Id}}

\DeclareMathOperator{\trace}{tr}
\DeclareMathOperator{\vol}{vol}
\newcommand{\wed}{{\mathord{\wedge}}}

\newcommand{\m}{\mathfrak{m}}
\newcommand{\restr}{\mathop{\upharpoonright}} 
\newcommand{\gtr}{\ge} 
\newcommand{\les}{\le} 
\newcommand{\bsigma}{\boldsymbol{\sigma}}  
\newcommand{\rad}{\mathbf{r}}    
\newcommand{\sing}{\mathbf{s}}   
\DeclareMathOperator{\jac}{jac}
\newcommand{\PT}{2\mathord{\wedge}\mathbb{N}} 
\newcommand{\dplus}{\mathbin{\vcenter{\hbox{$\mathord{\oplus}$}\nointerlineskip\hbox{\hspace{.25ex}{$\scriptscriptstyle\mathord{<}$}}}}}  
\newcommand{\DDiff}[1]{\Diff^1_{\#}(\DD^{#1})}

\newcommand{\arxiv}[1]{\href{http://arxiv.org/abs/#1}{\tt arXiv:{#1}}}
\newcommand{\mr}[1]{\href{http://www.ams.org/mathscinet-getitem?mr=#1}{\tt MR{#1}}}

\title[Lyapunov spectra of periodic orbits]{Perturbation of the Lyapunov spectra of periodic orbits}

\author{J.~Bochi and C.~Bonatti}

\classno{37D30 (primary), 15A18 (secondary)}

\date{April 28, 2010; first revision: June, 2010; this revison: May, 2011.}

\extraline{During the preparation of this paper, the first author was partially supported by CNPq--Brazil,
Faperj and Universit\'e de Bourgogne, and the second author was partially supported by IMPA and MathAmSud.}

\begin{document}

\maketitle

\begin{abstract}
We describe all Lyapunov spectra that
can be obtained by perturbing the derivatives along periodic orbits of a diffeomorphism.
The description is expressed in terms of the finest dominated splitting and Lyapunov exponents that appear in
the limit of a sequence of periodic orbits, and involves the majorization partial order.
Among the applications, we give a simple criterion for the occurrence of universal dynamics.
\end{abstract}


\section{Introduction}

\subsection{Perturbing the Derivatives along Periodic Orbits}

An important consequence of perturbation results as Pugh's closing lemma, Hayashi's connecting lemma and its generalizations is that the global dynamics of a $C^1$-generic diffeomorphism is very well approached by its periodic orbits. Understanding the periodic orbits and their behavior under perturbations is therefore a way for describing the global dynamics of a generic diffeomorphism. 

Let $f: M \to M$ be a diffeomorphism of a compact manifold of dimension $d$.
The main ingredient for describing the local dynamics in a neighborhood of a periodic point $p$ is the derivative $Df^{\pi(p)}$ at the period $\pi(p)$, and more specifically the Lyapunov exponents
$\lambda_1(p)\leq\cdots\leq \lambda_d(p)$, which are 
obtained by applying the function $\frac{1}{\pi(p)}\log |\cdot|$ to the eigenvalues.
However the derivative at the period is not sufficient for the understanding of how the orbit of $p$ reacts under small perturbations. For that purpose, one needs to know the derivatives $Df$ along the whole orbit. 
Let us give a simple example: 

Consider a set of periodic saddles $p_n$ in a compact invariant set $\Lambda$ of a surface diffeomorphism $f$.
If $\Lambda$ is a hyperbolic set then under perturbations of $f$ of $C^1$-size $\varepsilon$ the Lyapunov exponents of the periodic orbits will vary at most by a quantity proportional to $\varepsilon$. 
Consider now the case where $\Lambda$ is not hyperbolic.
This can happen even if the derivatives at the period ``look'' uniformly hyperbolic: 
for example, there may exist arbitrarily large segments of orbit on which the derivative of $f$ is almost an isometry.
Then, as Ma\~n\'e noticed in \cite{Mane ECL},
arbitrarily small perturbations of $f$ allow us to 
mix two Lyapunov exponents 
and create a new periodic orbit 
which is a sink or source.

More generally, Ma\~n\'e proved that, in any dimension, if the stable/unstable splitting over a set of periodic orbits is not dominated, then arbitrarily small perturbations of $f$ may create a non-hyperbolic periodic orbit, and then change its index. This was an important step in his proof of the $C^1$-stability conjecture (structural stability implies the Axiom A plus strong transversality condition). Ma\~n\'e's simple argument leads to a natural question:

\begin{ques}\label{q.perturbation} 
Let $f$ be a diffeomorphism.
Consider a sequence of periodic orbits $\gamma_n$ converging, in the Hausdorff topology, to a compact set $\Lambda$. 
What are the derivatives of periodic orbits one may obtain by small perturbations of the derivative of $f$ along the $\gamma_n$'s? 
\end{ques}

Recall that, by Franks Lemma \cite{Franks}, every perturbation of the derivatives
along a periodic orbit $\gamma_n$ can be realized by a $C^1$-perturbation of $f$
that keeps $\gamma_n$ invariant. So an answer to the question gives information
about the $C^1$-nearby diffeomorphisms.

It is clear that the existence of a dominated splitting on $\Lambda$ imposes obstructions.
Thus the answer of the question clearly depends on the finest dominated splitting.

Partial answers of Question~\ref{q.perturbation} have already been obtained: 
It was shown in \cite{BDP} that if the set 
of periodic orbits homoclinically related to a saddle $p$ has no dominated splitting at all, then arbitrarily small perturbations of $f$ may turn  the derivative at the period of one of these orbits to be an homothety. 
Removing the hypothesis that the orbits are all homoclinically related, \cite{BGV} obtains a slightly weaker result: a perturbation gives a periodic orbit having all the Lyapunov exponents equal. 
Other results along this direction are given in~\cite{LL}.
 
These results lead to the feeling that one can obtain any barycentric combinations of the Lyapunov exponents in a subbundle without dominated splitting. The present paper gives a precise meaning to this intuition, turning it into a theorem.  
 
\subsubsection{The Lyapunov Graph}
If $p$ is a periodic point with  Lyapunov exponents $\lambda_1\leq \cdots \leq \lambda_d$,
we associate to $p$ the \emph{Lyapunov graph} 
$\bsigma(p)=(\bsigma_0,\bsigma_1,\dots,\bsigma_d$) where $\bsigma_0=0$ and 
$\bsigma_i = \sum_{j=1}^i\lambda_i$ for $i>0$.  
(See Figure~\ref{f.graphs}.)
The fact that the $\lambda_i$ are increasing in $i$ is equivalent to the fact that the Lyapunov graph $\bsigma(p)\colon\{0,\dots,d\}\to \RR$ is \emph{convex}: 
$$
\bsigma_j \leq \frac{k-j}{k-i} \cdot \bsigma_i + \frac{j-i}{k-i} \cdot \bsigma_k 
\quad \text{for every $i<j<k$.}
$$  
We denote by $\cS_d \subset \{0\}\times \RR^d$ the set of convex graphs; its elements 
are seen as the graphs of convex maps $\sigma \colon\{0,\dots,d\} \to \RR$ with $\sigma_0=0$. 

\psfrag{1}[c][c]{{\tiny $\lambda_1$}}
\psfrag{2}[c][c]{{\tiny $\lambda_2$}}
\psfrag{3}[c][c]{{\tiny $\lambda_3$}}
\psfrag{4}[c][c]{{\tiny $\lambda_4$}}
\psfrag{5}[c][c]{{\tiny $\lambda_5$}}
\psfrag{a}[c][c]{{\tiny $\lambda_1'$}}
\psfrag{b}[c][c]{{\tiny $\lambda_2'$}}
\psfrag{c}[c][c]{{\tiny $\lambda_3'$}}
\psfrag{d}[c][c]{{\tiny $\lambda_4'$}}
\psfrag{e}[c][c]{{\tiny $\lambda_5'$}}
\psfrag{f}[l][l]{{\tiny $\lambda_1'' = \cdots = \lambda''_5$}}
\psfrag{z}[c][c]{{\tiny $0$}}
\begin{figure}[hbt]
\begin{center}
\includegraphics[scale=.3]{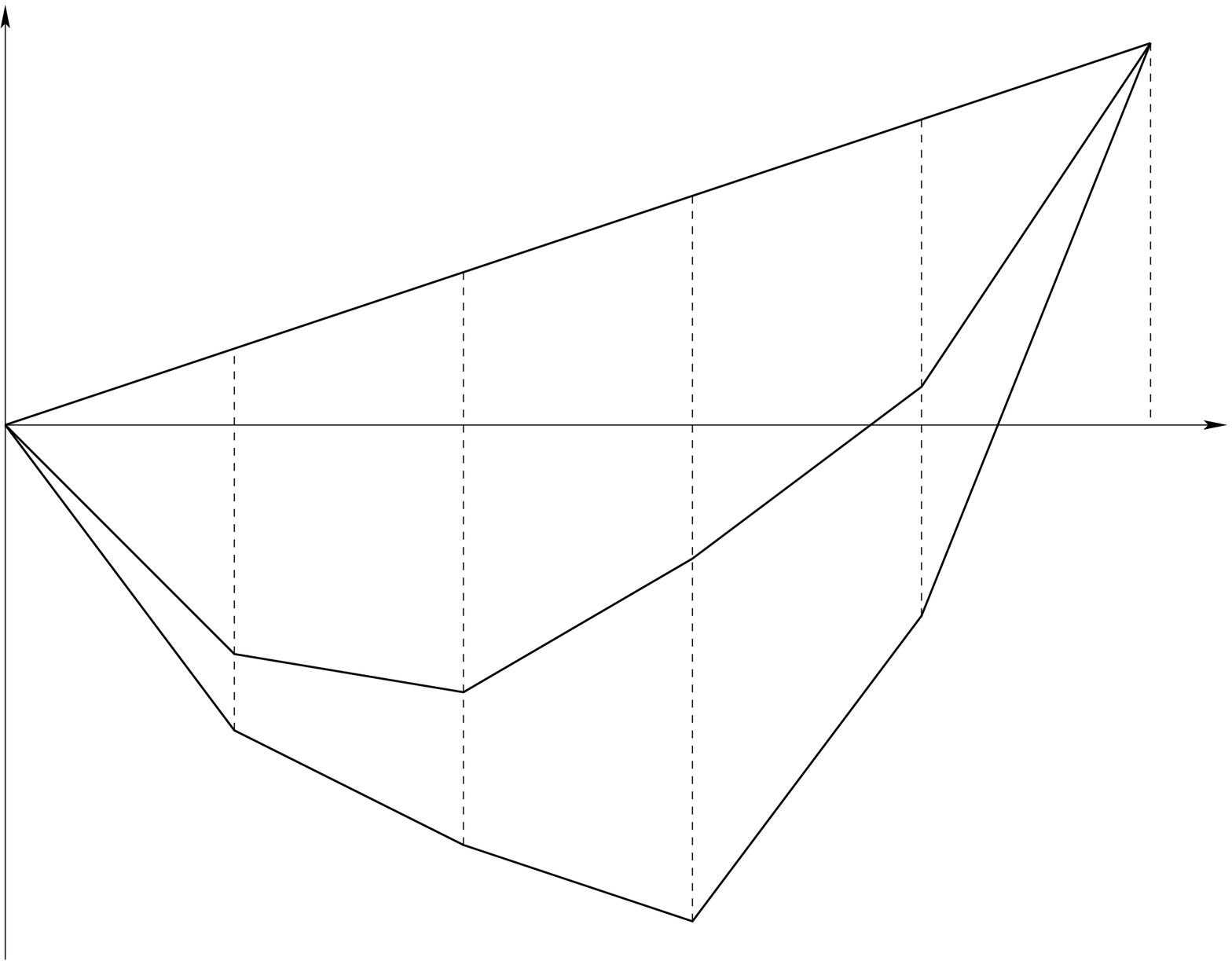}    \\
\vspace{13pt}
\includegraphics[scale=.25]{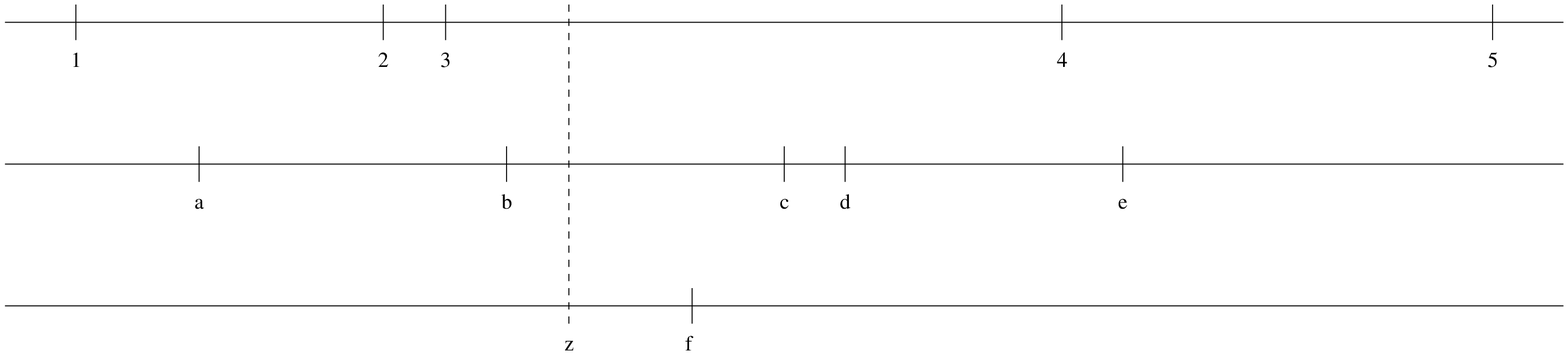} 
\end{center}
\caption{Three graphs $\sigma \les \sigma' \les \sigma''$ in $\cS_5$ with $\sigma_5 = \sigma'_5 = \sigma''_5$,
and the corresponding Lyapunov spectra.} \label{f.graphs}
\end{figure}

\subsubsection{Mixing Lyapunov Exponents, or Raising the Lyapunov Graph}
Our first result is this:

\begin{ithe}\label{t.periodic} 
Let $f$ be a diffeomorphism of a $d$-dimensional compact manifold, and 
$\gamma_n = \orb(p_n)$ be a sequence of periodic orbits whose periods tend to infinity. Assume that the sequence $\gamma_n$ converges in the Hausdorff topology to a compact set $\La$ that has no dominated splitting. 
Then given $\varepsilon>0$ there is $N$ such that
\begin{itemize}
\item for every $n\geq N$ and 
\item for any convex graph $\sigma\in \cS_d$ with
\begin{itemize}
\item $\sigma_d = \bsigma_d(\gamma_n)$  and 
\item $\sigma_i \geq \bsigma_i(\gamma_n)$ for every $i\in\{1,\dots, d-1\}$,
\end{itemize}
\end{itemize} 
there exists a $\varepsilon$-$C^1$-perturbation $g$ of $f$ with support in an arbitrary neighborhood of $\gamma_n$, preserving the orbit $\gamma_n$, and such that $\bsigma(\gamma_n,g)=\sigma$.
\end{ithe}

\begin{rema}\label{r.majorization}
The partial order on Lyapunov spectra that appears in the statement of the theorem
is called 
\emph{majorization}.\footnote{One says that the 
Lyapunov spectrum $\lambda_1\leq \cdots \leq \lambda_d$
\emph{majorizes} the spectrum $\lambda_1'\leq \cdots \leq \lambda_d'$
if the associated Lyapunov graphs $\sigma$, $\sigma'\in \cS_d$ satisfy
$\sigma_d = \sigma_d'$  and 
$\sigma_i \leq \sigma_i'$ for every $i\in\{1,\dots, d-1\}$.
(The disagreement here between $\sigma \leq \sigma'$ and the word ``majorize''
is due to the fact that we ordered the spectrum in increasing order,
while in the literature on majorization the decreasing order is preferred.)}
This terminology was introduced by Hardy, Littlewood and P\'{o}lya \cite{HLP}.
Majorization is a widely studied subject and has application in many different contexts:
see \cite{MOA}.
\end{rema}


We present now a more complete version of Theorem~\ref{t.periodic} including the case where $\La$ admits a dominating splitting:

\begin{ithe}\label{t.periodic2} 
Let $f$ be a diffeomorphism of a $d$-dimensional compact manifold, 
and $\gamma_n=\orb(p_n)$ be a sequence of periodic orbits whose periods tend to infinity, 
and that converges in the Hausdorff topology to a compact set $\La$. 
Let 
$$
E_1\dplus E_2 \dplus \cdots \dplus E_m
$$ 
be the finest dominated splitting over $\La$, and  denote 
$$
i_j = \dim (E_1\oplus\cdots\oplus E_j)
\quad \text{for $j\in\{1,\dots,m\}$.}
$$ 
Then given $\varepsilon>0$ there is $N$ such that:
 \begin{itemize}
 \item for every $n\geq N$ and 
 \item for any convex graph $\sigma\in \cS_d$ with 
 \begin{itemize}
 \item $\sigma_{i_j} = \bsigma_{i_j}(\gamma_n)$ for every $j\in\{1,\dots, m\}$ and
 \item $\sigma_i \geq \bsigma_i(\gamma_n)$ for every $i\in\{1,\dots, d\}$,
\end{itemize}
\end{itemize}
there exists a $\varepsilon$-$C^1$-perturbation $g$ of $f$ with support in an arbitrary neighborhood of $\gamma_n$, preserving the orbit $\gamma_n$ and such that $\bsigma(\gamma_n,g)=\sigma$.
\end{ithe}

The case where the finest dominated splitting is \emph{trivial} (that is, $m=1$)
corresponds to Theorem~\ref{t.periodic}.

In this statement, the requirement that the graph $\sigma$ touches $\bsigma(\gamma_n)$
at the dimensions $i_j$ corresponding to the finest dominated splitting cannot be 
significantly weakened: we could at most move these points by a quantity proportional to $\varepsilon$. 
Indeed, 
for large $n$ the orbit $\gamma_n$ (with respect to $f$ or the perturbation $g$)
has a dominated splitting close to that on $\Lambda$.
So $\bsigma_{i_j}(\gamma_n)$ is the average of the logarithm of the determinant of the derivative restricted
to the sum of the first $i$ bundles of this splitting,
and cannot vary much.

On the other hand, we can improve the conclusions of Theorem~\ref{t.periodic2} along other directions
(see Section~\ref{s.up} for precise statements):

\begin{itemize}
\item $g$ is isotopic to $f$ by an isotopy $g_t$ and all the $g_t$ are close to~$f$.

\item The $\bsigma_i(\gamma_n,g_t)$ vary monotonically along the isotopy. 

\item If $\gamma_n$ is a hyperbolic periodic orbit,
then we may require that along the isotopy
$\gamma_n$ remains hyperbolic, provided of course that 
the index of $\gamma_n$ (ie, the dimension of the contracting bundle)
is the same as the index of $\sigma$.

\item Theorems~\ref{t.periodic} and \ref{t.periodic2} are expressed in terms of dominated splittings on the Hausdorff limit $\La$ of the periodic orbits $\gamma_n$ and not on the orbits themselves.
That is because every periodic orbit automatically has a dominated splitting separating different Lyapunov exponents.
However such splittings may be very weak. We give individual (hence stronger) versions of Theorems~\ref{t.periodic} and \ref{t.periodic2}, expressed in terms of the weakness of the dominated splitting and of the period of a periodic orbit: given $\varepsilon>0$ there is $N$ and $\ell$ such that 
for every periodic orbit $\gamma$ with period larger than $N$,
we can make any perturbation of the Lyapunov spectrum of $\gamma$
that is compatible with its finest $\ell$-dominated splitting.

\item 
In particular,
the conclusions of Theorems~\ref{t.periodic}, \ref{t.periodic2} still hold if $\gamma_n$ are periodic orbits of diffeomorphisms $f_n$ such that $\gamma_n$ converges to $\La$ in the Hausdorff topology, and $f_n$ converges to $f$ in the $C^1$-topology. 
\end{itemize}

\subsubsection{Separating Lyapunov Exponents, or Lowering the Lyapunov Graph} \label{sss.separating}

Theorem~\ref{t.periodic2} explains what are the Lyapunov graphs $\sigma$ \emph{above} $\bsigma(\gamma_n)$ 
(i.e.\ $\sigma_i\geq \bsigma_i(\gamma_n)$) 
that can be obtained by small $C^1$-perturbations of the derivative $Df$ along $\gamma_n$. 
To get a complete answer to Question~\ref{q.perturbation} 
we need to remove the 
hypothesis that $\sigma$ is above $\bsigma(\gamma_n)$. 
At first, let us remark that there is a natural lower bound for the possible perturbations of $\bsigma(\gamma_n)$
in terms of limit measures:

Consider a sequence of periodic orbits $\gamma_n$ such that the 
the invariant probability $\mu_n$ supported on $\gamma_n$
converges weakly to a (non necessarily ergodic) measure $\mu$.
Let $\la_1(\mu)\leq\cdots\leq\la_d(\mu)$ be the integrated Lyapunov exponents of $\mu$ and $\bsigma(\mu)=(\bsigma_0(\mu),\dots,\bsigma_d(\mu))$ the associated Lyapunov graph, 
where $\bsigma_i(\mu) = \sum_{j=1}^i \lambda_j(\mu)$.
Recall that the map $(f,\mu)\mapsto \bsigma(f, \mu)$ 
is lower semicontinuous 
(see \S~\ref{ss.semicontinuity}).   
As a consequence, for any $\delta>0$, there is $\varepsilon>0$ and $N\in \NN$ such that, for any $n\geq N$ and any $\varepsilon$-perturbation $g_n$ of the derivative along the orbits $\gamma_n$, one has:
$$
\bsigma_i(g_n,\gamma_n) \geq \bsigma_i(f,\mu)-\delta \quad \text{for any $i=1,\dots ,d$.}
$$
Hence the Lyapunov graph of the limit measure $\mu$ appears as being a lower bound of the Lyapunov graph of perturbations of the derivative along $\gamma_n$. 
The result below asserts that this bound can be attained:

\begin{ithe}\label{t.measure intro} 
Let $f$ be a diffeomorphism of a $d$-dimensional compact manifold, 
and let $f_n$ be a sequence of diffeomorphisms converging to $f$ in the $C^1$-topology.  
Let $\gamma_n=\orb(p_n,f_n)$ be a sequence of periodic orbits of  $f_n$ whose periods tend to infinity.
Suppose that the $f_n$-invariant probabilities $\mu_n$ associated to $\gamma_n$ 
converge in the weak-star topology to an $f$-invariant measure $\mu$. 
Then there is a sequence of diffeomorphisms $g_n$ such that:
\begin{itemize}
\item the $C^1$-distance between $g_n$ and $f_n$ tends to $0$;
\item the diffeomorphism $g_n$ preserves $\gamma_n$ and coincides with $f_n$ out of 
an arbitrarily small neighborhood of $\gamma_n$;
\item 
$\bsigma(g_n,\gamma_n) = \bsigma(f, \mu)$ for all $n$.
\end{itemize}
\end{ithe}

\begin{rema}
Section~6 from \cite{ABC} contains some related results.
(See also \S~\ref{sss.intro generic} below.)
Theorem~\ref{t.measure intro} is finer than the results in \cite{ABC}, because:
\begin{itemize}
\item it allows $\mu$ to be non-ergodic;
\item the periodic orbits with the desired Lyapunov exponents
produced by Theorem~\ref{t.measure intro} 
are also periodic orbits of the unperturbed map $f$,
while on \cite{ABC} these periodic orbits arise from the Ergodic Closing Lemma.
\end{itemize} 
\end{rema}

\subsubsection{Answer to Question~\ref{q.perturbation}} \label{sss.answer}

Next we combine Theorem~\ref{t.periodic2} with Theorem~\ref{t.measure intro}, 
in order to obtain a complete description of what is possible to get as a Lyapunov graph 
by perturbing the derivative along periodic orbits.
Let us first introduce some notation.

Given  a compact invariant set $\Lambda$ be for a diffeomorphism $f$,
let $E_1\dplus E_2 \dplus \cdots \dplus E_m$ be the finest dominated splitting over $\La$, and denote 
$i_j = \dim(E_1\oplus\cdots\oplus E_j)$ for $j \in \{1,\dots,m\}$. 
Given an $f$-invariant probability measure $\mu$ whose support is contained in $\Lambda$,
let $\cG(\mu,\La)$ indicate the set of convex graphs $\sigma\in \cS_d$ with 
 \begin{itemize}
 \item $\sigma_{i_j} = \bsigma_{i_j}(f,\mu)$ for every $j\in\{1,\dots, m\}$, and
 \item $\sigma_i \geq \bsigma_i(f,\mu)$ for every $i\in\{1,\dots, d\}$.
\end{itemize}

\begin{ithe}\label{t.measure intro 2} 
Let $f$ be a diffeomorphism of a $d$-dimensional compact manifold, and $f_n$ a sequence of diffeomorphisms converging to $f$ in the $C^1$-topology. Let $\gamma_n =\orb(p_n,f_n)$ be a sequence of periodic orbits of diffeomorphisms $f_n$ whose periods tend to infinity. 
Suppose that the $f_n$-invariant probabilities $\mu_n$ associated to $\gamma_n$ 
converge in the weak-star topology to an $f$-invariant measure $\mu$, 
and that the sets $\gamma_n$ converge in the Hausdorff topology to an $f$-invariant compact set $\La$. 

Then $\cG(\mu,\La)$ is precisely the set of the limits of Lyapunov graphs 
$\bsigma(g_n,\gamma_n)$ where $\{g_n\}$ runs over the set of all sequences of diffeomorphisms 
$g_n$ preserving $\gamma_n$
whose $C^1$-distance to $f_n$ tends to $0$ as $n \to \infty$.
\end{ithe}

Notice that $\Lambda \supset \supp \mu$ in the statement above, and this inclusion can be strict.
(See Lemma~\ref{l.class measure}, for example.)

\subsection{Consequences}\label{ss.intro conseq}

We now explore some consequences of our perturbation theorems.

\subsubsection{Index Changes}

The first corollary describes explicitly 
what are the indices\footnote{Recall that the \emph{index} (or \emph{stable index})
of a hyperbolic periodic orbit is the dimension of its contracting bundle.} 
we can create by perturbing a set of periodic orbits. 

\begin{coro}\label{c.index} 
Let $f$ be a diffeomorphism of a $d$-dimensional compact manifold, and 
let $\gamma_n =\orb(p_n)$ be a sequence of periodic orbits whose periods tend to infinity. 
Suppose that the invariant probabilities $\mu_n$ associated to $\gamma_n$ 
converge in the weak-star topology to a measure $\mu$, 
and that the sets $\gamma_n$ converge in the Hausdorff topology to an $f$-invariant compact set $\La$. 
Let $E_1\dplus E_2 \dplus \cdots \dplus E_m$ be the finest dominated splitting over $\La$, and denote 
$i_j = \dim(E_1\oplus\cdots\oplus E_j)$ for $j \in \{1,\dots,m\}$, and $i_0=0$. 

Assume that $k \in \{0,\ldots,d\}$ satisfies 
\begin{equation}\label{e.cutoff}
\bsigma_k(\mu) \leq \min_{j \in \{0,\dots,m\}}  \bsigma_{i_j}(\mu) \, .
\end{equation}
Then for every sufficiently large $n$ there exists a perturbation $g$ of $f$ 
such that $\gamma_n$ is preserved by $g$ 
and is hyperbolic with index $k$.
\end{coro}

An example is shown in Figure~\ref{f.interval}.
Notice that the indices $k$ that satisfy relation~\eqref{e.cutoff} form an interval in $\ZZ$,
which is contained in an interval of the form $[i_{j-1}, i_j]$.

\begin{figure}[hbt]
\begin{center}
\psfrag{1}[c][c]{{\footnotesize $i_{j_0-1}$}}
\psfrag{2}[c][c]{{\footnotesize $i_{j_0}$}}
\psfrag{M}[l][l]{{\footnotesize $\bsigma_{i_{j_0}} = \min_j \bsigma_{i_{j}}$}}
\includegraphics[scale=.5]{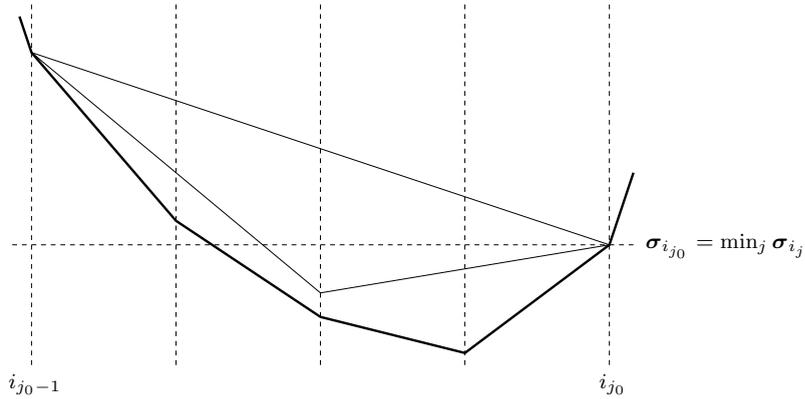}
\end{center}
\caption{\small An example in the situation of Corollary~\ref{c.index}: 
the lower graph is $\bsigma(\mu)$.
In this case, the minimum of $\bsigma_{i_{j}}(\mu)$ is attained at a unique $j = j_0$.
The numbers $k$ that satisfy condition \eqref{e.cutoff} are
$i_{j_0}-2$, $i_{j_0}-1$ and $i_{j_0}$; any of these is the index of a periodic point of a perturbation of $f$.
Some possibilities for the corresponding Lyapunov graphs are pictured.}
\label{f.interval}
\end{figure}

It was shown in \cite{ABCDW} that for any homoclinic class $H$
of a $C^1$-generic diffeomorphism, 
the indices of the periodic points in $H$ form an interval in $\ZZ$. This result is a motivation for the so called \emph{index completeness problems}:
\begin{ques} Fix an homoclinic class $H$ of a $C^1$ generic diffeomorphism $f$.
\begin{itemize}
	\item \emph{Inner completeness problem:} Does the interval $I_\text{in}$ of the indices of periodic orbits contained in $H$ coincide with the set $I_\text{erg}$ of indices of ergodic measures supported in $H$?
	\item \emph{Outer completeness problem:} Consider the set $I_\text{out}$ of indices of sequences of periodic orbits accumulating in $H$ (which contains $I_\text{erg}$, by \cite[Theorem~3.8]{ABC}). Does it form an interval? Does it coincide with the interval $I_\text{in}$?
\end{itemize}
\end{ques}

As a step towards the solution of the first problem above, we pose the following question:

\begin{ques}
In the situation of Corollary~\ref{c.index},
suppose additionally that all the periodic orbits $\gamma_n$ belong to the same homoclinic class $H$.
Let $k$ satisfy condition \eqref{e.cutoff}. 
Can we find a perturbation of $f$
such that the continuation of the class $H$
contains a periodic orbit of index~$k$?
\end{ques}

We expect the answer to be positive; Gourmelon's result from \cite{Gou} (a Franks' lemma that controls the position of the invariant manifolds of a periodic point) should be useful here. 

\subsubsection{Lyapunov Spectra of Periodic Orbits for Generic Diffeomorphisms}\label{sss.intro generic}

A general principle is that generic diffeomorphisms already display any robust property that we can get by small perturbations. 
Theorem~\ref{t.measure intro 2} explains what Lyapunov exponents are possible to get by perturbing the derivative of a diffeomorphism along a sequence of periodic orbits. 
Applying the general principle to Theorem~\ref{t.measure intro 2},
we obtain Corollary~\ref{c.generic} below.
It states that generically 
the closure of the set of Lyapunov graphs associated to a sequence of periodic orbits is 
exactly the set of graphs that are greater than the Lyapunov graph of the limit measure and that respect the constraints due to the dominated splittings on the limit support. 

\medskip

In order to be more precise, let us introduce some notation.
If $X$ is a compact metric space, let $\cP(X)$ be the set of Borel probability measures on $K$,
endowed with the weak-star topology.
Also, let $\cK(X)$ be the set of compact subsets of $X$ endowed with the Hausdorff distance.
Both $\cP(X)$ and $\cK(X)$ are compact sets.

If $f$ is a diffeomorphism of a compact manifold, and $\gamma \in \cK(M)$ is a periodic orbit,
let $\mu_\gamma$ indicate the only $f$-invariant probability supported on $\gamma$.
Let $\cX(f)$ be the closure in $\cP(M) \times \cK(M)$ of the set of pairs $(\mu_\gamma,\gamma)$ 
where $\gamma$ runs over the set of hyperbolic periodic orbits of $f$.
Since hyperbolic periodic orbits persist under perturbations,
the map $f \mapsto \cX(f)$ is lower semicontinuous\footnote{Recall that 
if $Y$ is a compact metric space and $Z$ is a topological space then a map 
${\Phi: Z \to \cK(Y)}$ is called \emph{lower} (or \emph{inner}) \emph{semicontinuous}
if for every $z \in Z$ and every open $V \subset Y$ with $V \cap \Phi(z) \neq \emptyset$ 
there is a neighborhood $U$ of $z$ in $Z$ such that $\Phi(z') \cap V \neq \emptyset$ for all $z' \in U$.
Also, $\Phi$ is called  \emph{upper} (or \emph{outer}) \emph{semicontinuous}
if for every $z \in Z$ and every open $V \subset Y$ with $V \supset \Phi(z)$ 
there is a neighborhood $U$ of $z$ in $Z$ such that $\Phi(z') \subset V$ for all $z' \in U$.}.

According to Ma\~n\'e ergodic closing lemma~\cite{Mane ECL} (see also \cite[Theorem~4.2]{ABC}), 
for $C^1$-generic diffeomorphisms $f$ the set $\cX(f)$ contains the pair $(\mu,\supp \mu)$ for every ergodic measure~$\mu$.
For another way of finding elements of $\cX(f)$, see Lemma~\ref{l.class measure}.

Recall that if $\mu \in \cP(M)$ is an $f$-invariant measure and $\Lambda \in \cK(M)$ is an $f$-invariant set 
containing $\supp \mu$ then $\cG(\mu,\La)$ indicates the set of all Lyapunov graphs that are compatible with 
$\bsigma(\mu)$ and with the finest dominated splitting on $\La$.

\medskip

Now we can state the following consequence of Theorem~\ref{t.measure intro 2}, 
which simultaneously improves Theorem~3.8 and Corollary~3.9 from \cite{ABC}:

\begin{coro}\label{c.generic} 
For $C^1$-generic diffeomorphisms $f$, for every $(\mu,\La)\in\cX(f)$, for every $\sigma\in \cG(\mu,\La)$ 
there is a sequence $\gamma_n$ of periodic orbits 
converging to $\La$ for the Hausdorff topology,
with $\mu_{\gamma_n}$ converging to $\mu$ in the weak-star topology, 
and $\bsigma(f,\gamma_n)$ converging to $\sigma$.
\end{coro}

\subsubsection{Universal Dynamical Systems}   \label{sss.intro universal}

Among other applications of Theorem~\ref{t.measure intro 2},
we will obtain (in Theorem~\ref{t.criterion} below) a criterion  for a diffeomorphism 
to be approximated by wild diffeomorphisms. 
Let us begin with the relevant definitions.

Let $f$ be a diffeomorphism of a compact manifold $M$.
The \emph{chain recurrent set} $R(f)$ has a natural partition into \emph{chain recurrence classes}: 
two points $x$, $y$ are equivalent if for any $\varepsilon>0$
there is an $\varepsilon$-pseudo orbit starting at $x$, passing by $y$, and coming back to $x$. 

We say that the diffeomorphism $f$ is \emph{tame}
if each chain recurrence class $C$ is robustly isolated: 
for every $g$ in a $C^1$-neighborhood of $f$, there is a unique chain recurrent class contained in a small neighborhood of $C$. In an equivalent way, a diffeomorphism $f$ is tame if the number of chain recurrence classes is finite and constant in a $C^1$-neighborhood of $f$. 

The set of tame diffeomorphisms is indicated by $\cT(M)$,
and the set of \emph{wild diffeomorphisms} is defined as 
$$
\cW(M)=\Diff^1(M)\setminus\overline{\cT(M)} \, .
$$ 

A generic diffeomorphism is tame (resp.\ wild) if and only if 
it has finitely (resp.\ infinitely) many chain recurrence classes.\footnote{Proof: 
If a diffeomorphism $f$ satisfies the generic properties of 
Remark~1.12 and Corollary~1.13 from~\cite{BC},
and it has finitely many chain recurrence classes,
then this number is locally constant.} 

It is shown in \cite{BD ENS} that $\cW(M)$ is nonempty for every compact manifold $M$ with $\dim M\geq 3$.  

\medskip

A stronger notion of wildness called \emph{universal dynamics} 
was introduced in the paper \cite{BD IHES}.
Let us define this notion.

Let $\DD^k$ be the closed $k$-dimensional disk.
Let $\DDiff{k}$ be the set of diffeomorphisms from $\DD^k$ to a subset of $\Int \DD^k$
that are diffeotopic to the identity map.

For any $k\in \{1,\dots,d\}$, we say that a diffeomorphism $f$ 
is \emph{$k$-universal}
(or has \emph{$k$-universal dynamics})
if there is a collection $\{D_n\}$ 
of embedded $k$-dimensional closed discs, 
with embeddings $\phi_n \colon \DD^k \to D_n$,
such that the following properties hold:
\begin{itemize}
\item For each $n$ there is $\pi_n$ 
such that $D_n$, $f(D_n)$, \ldots, $f^{\pi_n-1}(D_n)$ are pairwise disjoint
and $f^{\pi_n}(D_n)$ is contained in the (relative) interior of~$D_n$.
\item The orbits of the discs are pairwise disjoint. 
\item The discs are normally hyperbolic.
\item Let $F_n$ indicate the restriction of $f^{\pi_n}$ to $D_n$.
Then the maps $\phi_n^{-1} \circ F_n \circ \phi_n\colon \DD^k \to \Int \DD^k$ 
form a dense family in the set $\DDiff{k}$.
\end{itemize}
The $d$-universal diffeomorphisms are simply called \emph{universal}.

Some easy observations about this definition follow:

\begin{itemize}

\item
If $f$ is $k$-universal, then so is $f^{-1}$. 
(Take a suitable family of smaller discs.)

\item
$k+1$-universal dynamics implies $k$-universal dynamics. 

\item
The property of $1$-universal dynamics is very weak: 
it is generically satisfied in $\Diff^1(M)\setminus \overline{\mathrm{Hyp}(M)}$, 
where $\mathrm{Hyp}(M)$ is the set of Axiom~A diffeomorphisms without cycles. 

\item
If $\cR$ is a locally residual set (i.e., a set that is residual on an open set)
of $d$-universal diffeomorphisms 
then every $f \in \cR$ is wild.

\end{itemize}

It was shown in \cite{BD IHES} that (nonempty) locally residual sets of diffeomorphisms
with $d$-universal (and hence wild) dynamics indeed exist. 
These diffeomorphisms generically
have any robust or locally generic dynamical property 
that appears in $\DDiff{d}$
(e.g.\ existence of sink, source, aperiodic maximal transitive Cantor sets, etc.),
and this property is displayed in infinitely many periodic discs.

\medskip

However, $k$-universal dynamics does not imply wildness, for $k\leq d-2$: 
It is not too hard to modify Shub example 
of a non-hyperbolic robustly transitive diffeomorphism
(see e.g.\ \cite[\S~7.1.1]{BDV})
in order to find open sets of robustly transitive diffeomorphisms that are generically $d-2$-universal.

For this reason, we will introduce a stronger notion.
We will say that $f$ is \emph{freely $k$-universal} 
if one may choose the discs $D_n$ in the definition of $k$-universal dynamics 
so that they are pairwise separated by a filtration: 
for any $n\neq m$ there is an attracting region $U$ of $f$
(that is, $f(\bar U)\subset \Int U$) containing $D_n$ and disjoint from $D_m$, or vice versa.

Some easy observations are:  

\begin{itemize}

\item If $f$ has free $k$-universal dynamics, then so does $f^{-1}$. 
(Replace $U$ by {$M \setminus \bar{U}$}.)

\item $k$-universal dynamics is always freely $k$-universal if $k=d$ or $d-1$.
For $k=d$ this is immediate. 
If $k=d-1$ then
by definition each disc $D_n$ is either normally contracting or normally expanding,
and so there is a small neighborhood $V$ of $D_n$ (which can be chosen disjoint from the orbit of any other disc $D_m$) such that either $V$ or $M \setminus V$ is an attracting region.
\end{itemize}

Let us say that $f$ has \emph{normally contracting} (resp.\ \emph{normally expanding}) $k$-universal dynamics
if all discs in the definition of $k$-universality can be taken normally contracting 
(resp.\ normally expanding). \footnote{More generally, it could be interesting, for a global study of a wild dynamics, to distinguish other types of universal dynamics according to the kind of normal hyperbolicity.}
In any case, $f$ is freely $k$-universal. 

\medskip
 
A direct consequence of results from \cite{BD IHES} (see \S~\ref{ss.BD} below for details)
is the following criterion for free $k$-universality: 
\begin{theo}\label{t.BD}
Let $f$ be a diffeomorphism having a periodic point $p$ such that $Df^{\pi(p)}(p)$ satisfies:
\begin{itemize}
\item There is an invariant subspace $E \subset T_p M$ restricted to which $Df^{\pi(p)}(p)$ is the identity map;
\item $\dim E = k \geq 3$;
\item The other $d-k$ eigenvalues all have modulus bigger than $1$. 
\end{itemize}
Then there are arbitrarily small $C^1$ perturbations of $f$, 
supported in arbitrarily small neighborhoods of $p$,
that belong to a locally generic set formed by normally expanding $k$-universal diffeomorphisms.
\end{theo}

Obviously, there is a similar criterion for normally contracting universal dynamics.

\begin{rema}
The criterion given by the theorem is certainly wrong if $k=1$. 
If $k=2$, it is unknown and is related with Smale conjecture on the denseness of Axiom~A diffeomorphisms on surfaces. 
However, if $\cU$ is a $C^1$-open set such that diffeomorphisms $f$ in a dense subset of $\cU$ have a periodic point satisfying the hypotheses of the theorem, with $k=2$, then generic diffeomorphisms in $\cU$ have free $2$-universal dynamics.
\end{rema}

Theorem~\ref{t.BD} gives a hint that the control of Lyapunov exponents
can be useful to get free $k$-universality (at least if $k \ge 3$):
one needs $k$ vanishing exponents, and all the others having the same sign. 
In fact, using Theorem~\ref{t.measure intro 2},
we can show the following simple criterion for a $C^1$-generic diffeomorphism to generate 
free $k$-universal dynamics, for any $k<d$:

\begin{ithe}\label{t.criterion} 
Let $f$ be a $C^1$-generic diffeomorphism 
having a periodic point $p$ of index $k \in \{ 1, \ldots, d-1 \}$.
Let $E_1\dplus\cdots \dplus E_m$ be the finest dominated splitting 
on the homoclinic
class\footnote{The \emph{homoclinic class} of a hyperbolic periodic point $p$ 
is the closure of the transverse intersections of stable and unstable manifolds of points along the \emph{orbit} 
of $p$.} 
$H(p)$.
Suppose that $\big| \det Df^{\pi(p)} \restr E_1(p) \big| > 1$.
Then generic diffeomorphisms in a neighborhood of $f$ have the normally expanding $k$-universal dynamics.
\end{ithe} 


Here the interest is not to provide an example of a locally residual set with free $k$-universal dynamics: 
this could be done without difficulty using the arguments of \cite{BD IHES}. 
We can in fact strengthen Theorem~\ref{t.criterion} and obtain the following result,
which in particular characterizes the diffeomorphisms which are far from 
normally expanding $k$-universal dynamics:

\begin{ithe} \label{t.dichotomy}  
For any $k \in \{1, \ldots, d-1\}$,
if $f$ is a generic diffeomorphism  then $f$ has (at least) one of the following properties:
\begin{enumerate}
\item 
$f$ is normally expanding $k$-universal; or:
\item 
Let $p$ be any periodic saddle of index $k$ and 
let $E_1\dplus\cdots \dplus E_m$ be the finest dominated splitting on the homoclinic class $H(p)$. 
Then $f$ contracts uniformly at the period the volume 
in $E_1$, on the periodic orbits homoclinically related with $p$.
More precisely, there is $\alpha = \alpha(p) > 0$ such that, 
for any $q$ homoclinically related with $p$, 
$$
\frac{1}{\pi(q)} \log \big| \det Df^{\pi(q)} \restr E_1(q) \big| < -\alpha .
$$
\end{enumerate}
\end{ithe}

Coming back to $k$-universal dynamics, we also obtain a criterion for it 
similar to Theorem~\ref{t.criterion}: see Theorem~\ref{t.non-free} in Section~\ref{s.universal}.

\medskip

Let us mention that \cite{BLY} 
provides, in  dimension $3$, an example of a $C^1$-open set $\cO$ of diffeomorphisms $f$ having  a robust quasi-attractor $\La_f$, satisfying the hypotheses of Theorem~\ref{t.criterion}: $\La_f$ admits a dominated splitting $E^{cs}\dplus E^u$ with $\dim E^{cs}=2$, and a periodic point $p_f\in\La_f$ of index $1$ with $\big| \det Df^{\pi(p_f)} \restr E^{cs}(p_f) \big| > 1$.  Hence, for generic $f$ in $\cO$ the local dynamics in a neighborhood of the quasi-attractor $\La_f$ is freely (normally expanding) $2$-universal.

\subsubsection{Other Consequences?}

We expect our results to be useful for other applications.
Bearing this in mind, we proved results that are actually stronger than 
those stated in this introduction.
For example, a strengthened version of Theorem~\ref{t.periodic2}
gives a whole \emph{path} of perturbations along which we have fine control of the Lyapunov graph
(see Section~\ref{s.up}).
This information can be useful if one wants to apply 
the Gourmelon--Franks Lemma \cite{Gou}, for instance.

\subsection{Other Comments and Organization of the Paper}

Actually most of our results are expressed in terms of \emph{linear cocycles}.
In fact, since this paper concerns \emph{periodic} orbits, we are mainly interested in 
cocycles over \emph{cyclic} dynamical systems.
The results for diffeomorphisms explained above follow by Franks Lemma.

Thus some of our results fit into the perturbation theory of matrix eigenvalues.
However, the literature in this area usually considers 
a single matrix or operator, while here we consider a finite product of them.
Of course, the key concept of domination is uninteresting for a single matrix.

\medskip

The paper is organized as follows. 
Section~\ref{s.prelim} introduces cocycles, 
and also contains other definitions, notations and basic facts to be used throughout the paper.
In Section~\ref{s.prop} we establish a central proposition that permits to mix two Lyapunov exponents while
keeping the others fixed.
In Section~\ref{s.up} 
we obtain cocycle versions of Theorems~\ref{t.periodic} and \ref{t.periodic2} 
that also incorporate the improvements mentioned above. 
In Section~\ref{s.measure} 
we obtain stronger versions of Theorems~\ref{t.measure intro} and \ref{t.measure intro 2}. 
The short Section~\ref{s.corol} contains the proofs of Corollaries~\ref{c.index} and \ref{c.generic}.
In Section~\ref{s.universal} we give the applications to universal dynamics.

\section{Definitions and Notations}  \label{s.prelim} 

\subsection{Linear Cocycles} 

A \emph{linear cocycle} is a vector bundle automorphism.
Let us be more precise and fix some notations.
Let $X$ be a compact metric space, 
and let $E$ be a vector bundle over $X$ of dimension $d$,
endowed with a euclidian metric $\|\mathord{\cdot}\|$.
The fiber over a point $x \in X$ is denoted by $E_x$ or $E(x)$.
Then a linear cocycle $A$ 
on $E$ is completely determined by a homeomorphism $T:X \to X$ and a 
continuous map that associates to each $x \in X$ an 
invertible linear map $A(x) : E(x) \to E(Tx)$.
We then say that \emph{$A$ is a cocycle on $E$ over $T$},
or more precisely that \emph{$(X,T,E,A)$ is a cocycle}.

The $n$-iterate of a cocycle is the cocycle over $T^n$ whose 
fiber maps are $A^n(x) = A(T^{n-1} x) \cdots A(x)$ if $n>0$,
$A^n(x) = A(T^{-1}x) \cdots A(T^n x)$ if $n<0$.

If $K>1$, we say that a cocycle as above is \emph{bounded by $K$}
if $K^{-1} \le \m(A(x)) \le \|A(x)\| \le K$ for every $x\in X$.
Here $\m(B)$ indicates the minimum expansion factor of the linear map $B$, that is
$\m(B) = \inf_{\|v\|=1}\|B v\|$, or $\m(B) = \|B^{-1}\|^{-1}$ when $B$ is invertible.

A cocycle $\tilde A$ is called an \emph{$\eps$-perturbation} of a cocycle $A$ if 
$\|\tilde A - A\| < \eps$.

A \emph{path of cocycles} is a family of cocycles $(X,T,E,A_t)$,
where $t$ runs on an interval $[t_0,t_1] \subset \RR$, such that $A_t(x)$ depends continuously on $(t,x)$.
We say that a path of cocycles $A_t$, $t\in [t_0,t_1]$ is \emph{$\eps$-short}
if each $A_t$ is an $\eps$-perturbation of $A_{t_0}$.

\begin{rema}\label{r.epsilon}
\begin{itemize}
\item If $A$ is bounded by $K$ and $\tilde A$ is an $\eps$-perturbation of $A$ then
$\tilde{A}^{-1}$ is a $K^2\eps$-perturbation of $A^{-1}$.
\item For any $K>1$, there is $\eps>0$ such that any $\eps$-perturbation of a cocycle bounded by $K$
is bounded by $2K$.   
\end{itemize} 
\end{rema}

\subsection{Restricted and Quotient Cocycles}\label{ss.procedures}

We say that a subbundle $F$ of $E$ (whose fibers by definition have constant dimension)
is \emph{invariant} if $A(x) \cdot F(x) = F(Tx)$ for each $x \in X$.
In that case, we define two new cocycles:
\begin{itemize}
\item 
The \emph{restricted cocycle} $A \restr F$ on the bundle $F$; 

\item 
the \emph{quotient cocycle} $A/F$ on the quotient bundle $E/F$ 
(where the norm of an element of $E(x)/F(x)$ is defined as the norm of its unique representative 
that is orthogonal to $F(x)$). 
\end{itemize}
Notice that if $A$ is bounded by $K$ then $A \restr F$ and $A/F$ are also bounded by $K$.

Let us recall Lemma~4.1 from \cite{BDP}, which gives some procedures for extension of cocycles
that will be used several times. 
Let $A$ be a cocycle on a bundle $E$ with an invariant subbundle $F$.
With respect to the splitting $E = F \oplus F^\perp$, we can write
$$ 
A = \begin{pmatrix} A\restr F & D \\ 0 & A/F \end{pmatrix}.
$$
Given any cocycle $B$ on $F$ we can define
a cocycle $\hat B$ on $E$ that preserves $F$ and satisfies $\hat B \restr F = B$ and $\hat B/F = A/F$, namely
$$ 
\hat B = \begin{pmatrix} B & D \\ 0 & A/F \end{pmatrix}.
$$
Moreover, $\hat B$ depends continuously on $A$ and $B$, and is bounded by $K$ if so are $A$ and $B$.
Similarly, given any cocycle $C$ on $E/F$ we can define a cocycle $\bar C$ on $E$ that preserves $F$,
and satisfies $\bar C \restr F = A$ and $\bar C/F = C$, namely
$$ 
\bar C = \begin{pmatrix} A \restr F & D \\ 0 & C \end{pmatrix}.
$$
Moreover, $\bar C$ depends continuously on $A$ and $C$, and is bounded by $K$ if so are $A$ and $C$.

\subsection{Domination}

Assume given a coycle $(X,T,E,A)$
and two invariant non-zero subbundles $F$, $G$ of constant dimensions.
Take $\ell$ in the set $\PT = \{ 2^0, 2^1, 2^2,  \ldots\}$.
We say that \emph{$F$ is $\ell$-dominated by $G$}
if
$$
\frac{\| A^\ell \restr F(x) \|}{\m(A^\ell \restr G(x))} < \frac{1}{2} \quad \text{for every $x\in X$.}
$$
This is denoted by $F <_\ell G$.
If in addition $E = F \oplus G$ then we say that
\emph{$E = F \dplus G$ is an $\ell$-dominated splitting}. 
The symbol $<$ under $\oplus$ is necessary because the order matters.
Notice that with respect to the inverse cocycle we have the reverse domination, that is, $G <_\ell F$.

A splitting is \emph{dominated} if it is $\ell$-dominated for some~$\ell \in \PT$.
The \emph{index}\footnote{Not to be confused with the (stable) index of a hyperbolic periodic orbit.}
of the dominated splitting is the number $\dim F$.

\begin{rema}
Most references do not require that the domination parameter $\ell$
must be a power of $2$. 
As it is trivial to see, this gives the same concept of dominated splitting.
An advantage of our powers-of-$2$ convention is that
$\ell$-dominated splittings are $L$-dominated for $L >\ell$.
\end{rema}

Given a cocycle and $\ell \in \PT$, an (ordered) invariant splitting $E = F_1 \dplus \cdots \dplus F_m$
into an arbitrary number of subbundles is called \emph{$\ell$-dominated} if 
$F_1 \oplus \cdots \oplus F_i$ is $\ell$-dominated by $F_{i+1} \oplus \cdots \oplus F_m$,
for each $i=1$, $2$, \ldots, $m-1$.
The \emph{indices} of the splitting are the numbers
$i_j = \dim F_1 \oplus \cdots \oplus F_j$ for $1 \le j \le m-1$.

\begin{prde}
Given any cocycle and any $\ell\in\PT$, there is an unique \emph{finest $\ell$-dominated splitting}
$F_1 \dplus \cdots \dplus F_m$, that is, an $\ell$-dominated splitting such that
if  $G_1 \dplus \cdots \dplus G_k$ is an $\ell$-dominated splitting then 
each $G_j$ is the sum of some of the $F_i$ (and in particular, $k \le m$).
(Here we must allow the possibility of a \emph{trivial} splitting, that is, $m=1$.)
\end{prde}

\begin{proof}
If $E = F' \dplus F''$ and $E = G' \dplus G''$ 
are $\ell$-dominated splittings with $\dim F' \le \dim G'$ and $\dim F'' \ge \dim G''$ then
$F' \subset G'$ and $F'' \supset G''$
(see \cite[p.~291]{BDV}).
If these inequalities are strict, we define $H = F'' \cap G'$.
Then, by dimension counting,
$G' = F' \oplus H$ and $F'' = H \oplus G''$ and 
in particular $F' \dplus H \dplus G''$ is an $\ell$-dominated splitting into $3$ bundles.
Existence and uniqueness of the finest $\ell$-dominated splitting follows easily from these remarks.
\end{proof}

\begin{rema}
Beware that the bundles of the finest $\ell$-dominated splitting 
can admit nontrivial $\ell$-dominated splitting themselves.
On the other hand, for the more standard notion of \emph{finest dominated splitting} (with no fixed $\ell$),
the bundles are indecomposable.
\end{rema}

\subsection{Lyapunov Exponents}

Given a cocycle as above, Oseledets theorem 
assures the existence of a full probability set $R \subset X$
(that is, a set that has full measure with respect to any $T$-invariant Borel probability measure),
called the \emph{set of regular points},
such that for each $x \in R$ we have well-defined \emph{Lyapunov exponents} 
$$
\lambda_1 (A, x) \le \cdots \le \lambda_d(A, x)
$$
(repeated according to multiplicity).
Define
$$
\bsigma_i(A,x) = \sum_{j=1}^i \lambda_j(A,x).
$$
For each $x \in X$, the vector 
$$
\bsigma(A,x) = \big( 0, \bsigma_1(A,x) , \ldots, \bsigma_d(A,x) \big) \in \R^{d+1}
$$
is called the \emph{Lyapunov graph} of $A$ at $x$; 
the reason for the name is that we think of it as the graph of a map
$\{0,1,\ldots, d\}\to \R$.
Lyapunov graphs are always \emph{convex},
that is, they belong to the set
$$
\cS_d = \big\{ (\sigma_0, \ldots, \sigma_d) \in \R^{d+1} ; \; 
\sigma_0=0, \   \sigma_i - \sigma_{i-1} \le \sigma_{i+1}-\sigma_i  \text{ for $0<i<d$} \big\}.
$$

If $\tilde \sigma$, $\sigma \in \cS_d$, then we write
$\tilde \sigma \gtr \sigma$ to indicate that
$$
\tilde \sigma_i \ge \sigma_i \quad \text{for $i=1,2,\ldots,d$.} 
$$

We say that a continuous path of graphs 
$\sigma(t) \in \cS_d$, $t\in [t_0,t_1]$ 
is \emph{non-decreasing} if $t>t'$ implies $\sigma(t) \gtr \sigma(t')$.

\subsection{Cocycles Over Cyclic Dynamical Systems}

We will be specially concerned with cocycles  $(X,T,E,A)$
where the dynamical system $T:X\to X$ is \emph{cyclic},
that is, the set $X$ is finite, say with cardinality $n$, and $T$ is a cyclic permutation.
In that case, we will say that $T$ (or $A$) has \emph{period $n$},
and that the cocycle is \emph{cyclic}.

The \emph{eigenvalues} of the cocycle are the 
the eigenvalues of $A^n(x)$, where $x$ is any point of~$X$.
The Lyapunov exponents are the logarithms of the moduli
of the eigenvalues (repeated according to multiplicity)
divided by $n$.
The Lyapunov graph does not depend on the point $x$ and is written as $\bsigma(A)$.

\subsection{Difference Operator Notation and Convexity}\label{ss.difference operator}

(This material will be used in Sections~\ref{s.up} and \ref{s.measure}.)

Given a finite sequence of real numbers $y = (y_0, \ldots, y_k)$ we define another sequence
$\Delta y = (\Delta y_0, \ldots, \Delta y_{k-1})$ by 
\begin{equation}\label{e.difference operator}
\Delta y_i = y_{i+1} - y_i \, .
\end{equation}
A more precise notation would be $(\Delta y)_i$,
but we will follow custom and drop the parentheses.

Recursively we define another sequence $\Delta^2 y = \Delta (\Delta y)$,
that is $\Delta^2 y_i = y_{i+2} - 2y_{i+1} + y_i$,
for $i=0$, \ldots, $k-2$.
If the numbers $\Delta^2 y_i$ are always non-negative
then the graph of $y$ is convex.
The next lemma says that if in addition these numbers are always small then
the graph of $y$ is close to affine:

\begin{lemm}\label{l.nonlinearity}
If a sequence  $y_0, \ldots, y_k$ satisfies
$0 \le \Delta^2 y_i \le \gamma$ for $0 \le i \le k-2$
then
$$
0 \le \frac{k-i}{k} y_{0} + \frac{i}{k} y_{k}  -  y_i \le \frac{k^2}{4} \gamma
 \quad \text {for $0 \le i \le k$.}
$$
\end{lemm}

\begin{proof}
Since $\Delta^2 y_i \ge 0$, the graph of the sequence $y$
is convex and so the first asserted inequality holds.
Also by convexity, 
\begin{equation}\label{e.triangle}
y_i \ge \max \big( y_0  + i\Delta y_{0} , y_{k} - (k-i)\Delta y_{k-1} \big)
\end{equation}
for any $i \in (0,k)$.
Now,
\begin{multline*}
y_{k} 
=   y_{0} + \sum_{j=0}^{k-1} \Delta y_j
=   y_{0} + \sum_{j=0}^{k-1} \left( \Delta y_0 + \sum_{i=0}^{j-1} \Delta^2 y_i \right) \\
= y_0 + k \Delta y_0 + \sum_{i=0}^{k-2} (k-1-i)\Delta^2 y_i
\le y_{0} + k \Delta y_0 + \frac{k^2}{2} \gamma.
\end{multline*}
That is,  $\Delta y_{0} \ge (y_{k}-y_{0})/k - k\gamma/2$.
Symmetrically, $\Delta y_{k} \le (y_{k}-y_{0})/k + k \gamma/2$.
Using these estimates in \eqref{e.triangle} we get that for any $i \in (0, k)$, 
$$
y_i \ge \frac{k-i}{k} y_{0} + \frac{i}{k} y_{k}
- \frac{k\gamma}{2} \min (i, k-i) ,
$$
which immediately implies the lemma.
\end{proof}

\subsection{Singular Values and Exterior Powers}\label{ss.singular}

(This material will be used in Section~\ref{s.measure} only.)

Let $E$ and $E'$ be euclidean spaces (that is,
real vector spaces endowed with inner products) of the same dimension $d$.
Let $M: E \to E'$ be a linear map. 
We denote by $\jac M$  
the modulus of the determinant of the matrix of $M$ with respect to 
an arbitrary pair of orthonormal bases.
Let us indicate by
$\sing_1(M) \ge \cdots \ge \sing_d(M)$
the singular values of $M$ (that is, the eigenvalues of $\sqrt{M^* M}$, or equivalently
the semi-axes of the ellipsoid  $M (\mathbb{S}^{d-1})$)
repeated according to multiplicity.
(Notice that the singular values are non-increasingly ordered, 
opposite to our convention for the Lyapunov exponents.)
Thus $\|M\| = \sing_1(M)$, $\m(M)=\sing_d(M)$, 
and $\jac M = \prod_{i=1}^d \sing_i(M)$.

If $E' = E$ then we indicate by $\rad(M)$ the spectral radius of $M$. 

\medskip

We will need a few facts about exterior powers; see
e.g.~\cite{LArnold} for details. 
If $E$ is a vector space of dimension $d$,
let $\wed^i E$ indicate its $i$-th exterior power; 
this is a vector space of dimension
$\binom{d}{i}$ whose elements are called $i$-vectors.
Moreover, an inner product on $E$ induces an inner product on $\wed^i E$ 
with the following properties:
\begin{itemize}
\item The norm of a decomposable $i$-vector $v_1 \wedge \cdots \wedge v_i$
equals the $i$-volume of the parallelepiped with edges $v_1$, \dots, $v_i$.
\item If $\{e_1, \ldots, e_d\}$ is an orthonormal basis for $E$
then $\{e_{j_1} \wedge \cdots \wedge e_{j_i} \; ; j_1< \ldots < j_i\}$
is an orthonormal basis for $\wed^i E$.
\end{itemize}
Any linear map $M: E \to E'$ 
induces a linear map $\wed^i M: \wed^i E \to \wed^i E'$ such that 
the image of a decomposable $i$-vector $v_1 \wedge \cdots \wedge v_i$
is $M v_1 \wedge \cdots \wedge M v_i$.
Moreover, the singular values of $\wed^i M$ are obtained by taking
all possible products of $i$ singular values of $M$;
in particular,
\begin{align*}
\|\wed^i M\|      &= \sing_1 (M) \sing_2(M) \cdots \sing_i(M) , \\
\sing_2(\wed^i M) &= \sing_1 (M) \sing_2(M) \cdots \sing_{i-1}(M) \sing_{i+1}(M) , \\
\m(\wed^i M)      &= \sing_{d-i+1} (M) \sing_{d-i+2} (M)  \cdots \sing_d(M) .
\end{align*}
Analogously, if $E'=E$ then the eigenvalues of 
$\wed^i M$ are obtained by taking
all possible products of $i$ eigenvalues of $M$.

\subsection{Semicontinuity of the Lyapunov Spectrum} \label{ss.semicontinuity}

(This material will be used in Section~\ref{s.measure} only.)

Let $(X,T,E,A)$ be a cocycle of dimension $d$. 
Suppose $\mu$ is a (non necessarily ergodic) 
$T$-invariant probability measure.
We denote
$$
\bsigma_i(A,\mu) = \int \bsigma_i (A,x)\, d\mu(x) \quad \text{and} \quad
\bsigma(A,\mu) = \big(\bsigma_0(A,\mu) , \ldots, \bsigma_d(A,\mu)\big) .
$$
Then $\bsigma(A,\mu)$ is a convex graph, that is, an element of $\cS_d$.

It is sometimes more convenient to deal with the integrated sum of the $i$ \emph{biggest} Lyapunov exponents:
$$
L_i(A,\mu) = \int \big( \lambda_d(A,x)+\lambda_{d-1}(A,x) + \cdots + \lambda_{d-i+1}(A,x) \big) \, d\mu(x).
$$
That is, $L_i(A,\mu) = \bsigma_d(A,\mu) - \bsigma_{d-i}(A,\mu)$.
These numbers are also expressed by
$$
L_i(A,\mu) = 
\lim_{m \to \infty} \frac{1}{m} \log \|\wed^i A^m\| \, d\mu = 
\inf_{m} \frac{1}{m} \log \|\wed^i A^m\| \, d\mu  \, .
$$
As an immediate consequence of this formula,
the numbers $L_i(A,\mu)$ are upper-semicontinuous
with respect to $A$ and $\mu$
(where in the space of measures we use the weak-star topology).
Of course, if $i=d$ then the function is continuous, because it is given 
by $\int \log \jac A \, d\mu$.

Thus $\bsigma_i$ is lower-semicontinuous and $\bsigma_d$ is continuous.
In other words, if $A$ is the limit of a sequence of cocycles $B_k$,
and $\mu$ is the weak-star limit of a sequence of invariant probabilities $\mu_k$
then every accumulation point $\sigma$ of the sequence $\bsigma(B_k,\mu_k)$
satisfies $\sigma \gtr \bsigma(A,\mu)$ and $\sigma_d = \bsigma_d(A,\mu)$.

\section{Mixing Only Two Exponents}\label{s.prop}

The proofs of our results on raising Lyapunov graphs
rely on the central Proposition~\ref{p.two exp} below, 
which says how to perturb two ``neighbor'' Lyapunov exponents,
while keeping the others fixed.
This section is devoted to prove it.

\begin{prop}\label{p.two exp}
For any $d \ge 2$, $K>1$, $\eps>0$, there exists $\ell\in \PT$ 
such that the following holds:
Let $(X,T,E,A)$ be a $d$-dimensional cyclic 
cocycle bounded by $K$ 
and of period at least $\ell$.
Assume that $A$ has only real eigenvalues
and has no $\ell$-dominated splitting of index~$i$.

Then there exists 
an $\eps$-short path of cocycles $A_t$, $t\in [0,1]$ starting at $A$,
all of them with only real eigenvalues,
such that the path of graphs $\bsigma(A_t)$ is non-decreasing,
$\bsigma_j(A_t) = \bsigma_j(A)$ for all $j\neq i$,
and $\bsigma_i(A_1) = \big( \bsigma_{i-1}(A) + \bsigma_{i+1}(A) \big)/2$.
\end{prop}

The assertions about $\bsigma(A_t)$ can be reread as follows:
\begin{itemize}
\item the functions $\lambda_1(A_t)$, \ldots, $\lambda_{i-1}(A_t)$, 
$\lambda_i(A_t) + \lambda_{i+1}(A_t)$, $\lambda_{i+1}(A_t)$, \ldots, $\lambda_d(A_t)$
are constant;
\item $t>t'$ implies $\lambda_i(A_{t'}) \le \lambda_i(A_t) \le \lambda_{i+1}(A_t) \le \lambda_{i+1}(A_{t'})$;
\item $\lambda_i(A_1) = \lambda_{i+1}(A_1)$.
\end{itemize}

\begin{rema}\label{r.robinhood}
Such a path of graphs corresponds to what is known in majorization theory as an
``elementary Robin Hood operation'' or ``elementary $T$-transform'' (see \cite{MOA}, p.~82).
In the interpretation where $\lambda_j$ is the wealth of the individual $j$ 
(see \cite{MOA}, p.~5--8),
such operation consists of transferring part of the wealth of the individual $i+1$
to its neighbor $i$ 
in such a way that $i+1$ stays at least as rich as $i$.
(In the case of Proposition~\ref{p.two exp}, the two individuals become equally rich.)
A general ``Robin Hood operation'' 
is a transfer of wealth between two individuals that
are not necessarily neighbors.
\end{rema}

Here is a extremely brief indication of the proof:
The case $d=2$ is easy: A preliminary perturbation makes the angle between the two bundles small over some point,
and then the exponents are mixed by composing with rotations at this point.
The general case would be a trivial consequence of the $2$-dimensional case
if it were true that subbundles of a bundle without (strong) domination have no (strong) domination as well.
(It is false!)
To deal with the general case, we show that if a cocycle has no domination
then after a perturbation this non-dominance appears on a subbundle or on a quotient bundle.
This permits us to prove Proposition~\ref{p.two exp} by induction on the dimension~$d$.

\medskip

Before going into the proof of Proposition~\ref{p.two exp} itself,
we need some auxiliary results.

\subsection{Converting Non-dominance into Small Angles}

Assume that a cocycle $A$ has an invariant splitting $F \oplus H$ 
such that the eigenvalues of the restricted cocycle $A \restr F$ 
are all different from those of the restricted cocycle $A\restr H$.
Suppose that $A_t$ is a path of cocycles, all of them with the same eigenvalues.
Then $A_t$ has an invariant splitting $F_t \oplus H_t$ 
that depends continuously on $t$
and coincides with $F \oplus H$ for $t=0$.
The bundles $F_t$ and $H_t$ are called the \emph{continuations} of $F$ and $H$.

The following lemma is based on an argument by Ma\~{n}\'{e}:

\begin{lemm}\label{l.mane} 
For any $d \ge 2$, $K>1$, $\eps>0$ 
and $\alpha>0$, there exists $\ell\in \PT$ 
such that the following holds:
Let $(X,T,E,A)$ be a $d$-dimensional cyclic cocycle bounded by $K$ 
and of period $n \ge \ell$.
Assume that $E= F \oplus H$ is an invariant splitting 
such that
\begin{equation}\label{e.separation}
\frac{\|A^n \restr F(x)\|}{\m( A^n \restr H(x) )} < 1 \quad \text{for any $x\in X$.}
\end{equation}
(In particular, the Lyapunov exponents along $F$ are smaller that the Lyapunov exponents along $H$.)
Assume also that $F$ is not $\ell$-dominated by $H$,
and that $F$ or $H$ is one-dimensional.

Then there exists
an $\eps$-short path of cocycles $A_t$, $t\in [0,1]$ starting at $A$,
all of them with the same eigenvalues,
such that if $F_t$ and $H_t$ denote the continuations of $F$ and $H$,
then 
$$
\angle (F_1(x_0), H_1(x_0)) < \alpha \quad \text{for some point $x_0 \in X$.}
$$
\end{lemm}

\begin{proof}
Let $d \ge 2$, $K>1$, $\eps>0$ and $\alpha>0$ be given.
Let $\ell$ be large (how large it needs to be will become clear later).

Take the cocycle $(X,T,E,A)$ and the splitting $E = F \oplus H$ as in the statement of the lemma.
At least one of the bundles $F$ or $H$ is one-dimensional.
Let us suppose it is $H$; the other case will follow by considering the inverse cocycle.

We assume that 
\begin{equation}\label{e.alpha}
\angle(F(x), H(x)) \ge \alpha \quad \text{for every $x\in X$,}
\end{equation}
otherwise there is nothing to prove.

Take a large number $\Lambda > 1$ (how large it needs to be will become clear later).
We split the proof into two cases:

\smallskip \noindent \emph{First case:}
We suppose that a strong form of non-dominance is present:
there is $z \in X$ and $0 < k < n= \# X$ such that
\begin{equation}\label{e.first case}
\frac{\|A^k \restr F(z)\|}{\m( A^k \restr H(z) )} > \Lambda.
\end{equation}
Then let $f$ be the unit vector in $F(z)$ most expanded by $A^k(z)$,
and let $h$ be an unit vector in $H(z)$. 
For $\tau \in [0,1]$, we define a linear map $S_\tau : E(z) \to E(z)$ by
$$
S_\tau \cdot h = \tau f + h , \quad
S_\tau \restr F(z) \text{ is the identity.}
$$
It follows from \eqref{e.alpha} that
there exists $C_1>0$ depending only on $\alpha$ such that 
$\|S_\tau  - \Id \| \le C_1 \tau$.
On the other hand, by~\eqref{e.separation} the vector $f$ is less expanded than $h$ by the map $A^n(z): E(z) \to E(z)$.
Using \eqref{e.alpha} again, we see that 
if $\tau\in [0,1]$ then the angle that the vector $A^n(z) \cdot (\tau f + h)$
makes with $F(z)$ cannot be too small.
Therefore, defining another linear map $U_\tau: E(z) \to E(z)$ by
$$
U_\tau \cdot \big[ A^n(z) \cdot (\tau f + h) \big] = A^n(z) \cdot h , \qquad
U_\tau \restr F(z) \text{ is the identity,}
$$
we have $\|U_\tau  - \Id \| \le C_2 \tau$, for some $C_2>0$ that depends only on $\alpha$.

Take $0<\beta<1$  with $C_1 \beta K$, $C_2 \beta K < \eps/2$.
Define a family of cocycles $A_t$, $t\in [0,1]$ by
$$
A_t(z) = A(z) \circ S_{\beta t} , \quad
A_t(T^{-1} z) = U_{\beta t} \circ A(T^{-1} z), \qquad
A_t(x) = A(x) \text{ if } x \neq z, T^{-1} z.
$$
Then $A_0=A$ and $\|A_t - A\| \le \eps/2$ for every $t\in [0,1]$.
The cocycle $A_t$ has an invariant splitting $F_t \oplus H_t$ where
$F_t = F$ and $H_t (T^j z)$ is spanned by $A_t^j(z) \cdot h$.
Also, the linear maps $A_t^n(z)$ and $A^n(z)$ have the same eigenvalues.

If we choose $\Lambda$ large enough (depending on $\alpha$ and $\beta$)
then it follows from \eqref{e.first case} that
the angle between $A_1^k(z) \cdot h$
and $A^k(z) \cdot f$ is less than $\alpha$.
Thus $\angle (F_1(T^k z), H_1(T^k z)) < \alpha$, as desired.

\smallskip \noindent \emph{Second case:}
Assume that we are not in the previous case, that is,
$$
\frac{\|A^k \restr F(x)\|}{\m( A^k \restr H(x) )} \le \Lambda \quad \text{for every $x \in X$ and $0<k<n$.}
$$
By assumption, the splitting $F \oplus H$ is not $\ell$-dominated,
so there is $z\in X$ such that
$$
\frac{\|A^\ell \restr F(z)\|}{\m( A^\ell \restr H(z) )} \ge \frac{1}{2} \, .
$$
In particular,
$$
\frac{\|A^{\ell/2} \restr F(z)\|}{\m( A^{\ell/2} \restr H(z) )} \ge
\frac{\m( A^{\ell/2} \restr H(T^{\ell/2} z) )}{\|A^{\ell/2} \restr F( T^{\ell/2} z)\|} \cdot
\frac{\|A^\ell \restr F(z)\|}{\m( A^\ell \restr H(z) )} \ge
\frac{1}{2\Lambda} \, .
$$
Let $\gamma>0$ and define a family of cocycles $A_t$, $t\in [0,1]$ by:
\begin{itemize}
\item $A_t(x)$ equals $e^{\gamma t} A(x)$ over $F(x)$ and 
$e^{-\gamma t} A(x)$ over $H(x)$ if $x=T^i z$ with $0 \le i \le \ell/2$;
\item $A_t(x)$ equals $e^{-\gamma t} A(x)$ over $F(x)$ and 
$e^{\gamma t} A(x)$ over $H(x)$ if $x=T^i z$ with $\ell/2 \le i \le \ell$;
\item $A_t(x) = A(x)$ otherwise.
\end{itemize}
We choose $\gamma$ depending only on $\alpha$, $K$ and $\eps$ as large as possible
so that ${\|A_t - A\|} \le \eps/2$.
The cocycles $A_t$ have the same invariant subbundles and Lyapunov spectrum as $A$.
Also, 
$$
\frac{\|A_1^{\ell/2} \restr F(z)\|}{\m( A_1^{\ell/2} \restr H(z) )} =
e^{\gamma \ell} \frac{\|A^{\ell/2} \restr F(z)\|}{\m( A^{\ell/2} \restr H(z) )} \ge
\frac{e^{\gamma \ell}}{2\Lambda} \, .
$$
We assume that $\ell$ is large enough so that the right hand side is bigger than $\Lambda$.
Thus the cocycle $A_1$ (which still satisfies \eqref{e.separation})
falls in the first case of the proof, with $k=\ell/2$.
The desired path of cocycles from $A$ to a cocycle with a small angle
is thus obtained by concatenating the path just described with a path given by the first case.
\end{proof}

\subsection{The Effect of Rotations in Dimension 2}   

Recall that we indicate the spectral radius of a square matrix $B$ by $\rad(B)$.
Let $R_\theta$ indicate the rotation of angle $\theta$ in $\R^2$.

\begin{lemm}\label{l.dim 2}
Suppose $B \in \GL(2,\RR)$ has eigenvalues of different moduli.
Let $\alpha \in (0, \pi/2]$ be the angle between the eigenspaces.
Then there exists $\beta \in (0, \alpha)$ and $s \in \{+1, -1\}$
such that:
\begin{itemize}
\item the function $\theta \in [0,\beta] \mapsto \rad(R_{s\theta} B)$ is decreasing;
\item the matrix $R_{s\beta} B$ has two eigenvalues of the same moduli.
\end{itemize}
\end{lemm}

\begin{proof}
Begin noticing the following facts:
If $B \in \GL(2,\RR)$ has determinant $\pm 1$, then
$B$ has eigenvalues of different moduli if and only if
either $\det B = 1$ and $|\trace B| > 2$, or $\det B = -1$ and $\trace B \neq 0$.
Also, the spectral radius $\rad(B)$ increases with $\trace B$
on the set 
$\{\det B = 1 \text{ and } \trace B > 2\} \cup \{\det B = -1 \text{ and } \trace B > 0\}$. 

Now let $B$ be any matrix with eigenvalues of different moduli.
By multiplying it by a non-zero number, we can assume that 
it has determinant $\sigma = \pm 1$ and has $\lambda = \rad(B) > 1$ as an eigenvalue.
By conjugating with a rotation, we can further assume that
$$
B = \begin{pmatrix} \lambda & -c \\ 0 & \sigma \lambda^{-1} \end{pmatrix} .
$$
Let $s \in \{+1,-1\}$ be the sign of $c$ (if $c=0$ either choice works).
Then $\trace R_{s\theta} B = (\lambda + \sigma \lambda^{-1}) \cos \theta  - |c| \sin \theta$
is a decreasing function of $\theta$ on the interval $[0, \pi/2]$.
Let $\alpha \in (0,\pi/2]$ be the angle between the eigenvalues.
We have $\tan \alpha = |c|^{-1} (\lambda - \sigma \lambda^{-1})$.
So $\trace R_{s\alpha} B = 2 \sigma \lambda^{-1} \cos \alpha$,
which is less than $2$ if $\sigma = + 1$, and negative if $\sigma = -1$.
In either case, it follows that there exists a least $\beta \in (0,\alpha)$
such that $R_{s\beta} A$ has eigenvalues of same moduli.
\end{proof}

\subsection{Some Dominance and Angle Relations between Subbundles and Quotient Bundles}

Lemma~2.6 of \cite{BV} is an useful angle relation; it says that
for any splitting of euclidian space into three non-zero subspaces $U$, $V$, $W$, 
\begin{equation}\label{e.BV}
\sin \angle(W, U \oplus V) \ge \sin \angle(W, U) \cdot \sin \angle(U \oplus W, V) \, .
\end{equation}

The following is a generalization of Lemma~4.4 from \cite{BDP}; see Figure~\ref{f.BDP}:

\begin{lemm}\label{l.BDP}
For any $d$, $K$ and $\ell \in \PT$, there exists $L = L(d,K,\ell) \in \PT$ 
with the following properties: 
Let $(X,T,E,A)$ be a cyclic $d$-dimensional cocycle bounded by $K$,
with an invariant splitting $F \oplus F'$ into non-zero bundles.
Assume that $F'$ has an invariant subbundle $H$ such that
$F <_{\ell} H$ and $F/H <_{\ell} F'/H$. 
Then $F <_{L} F'$.
\end{lemm}

\psfrag{F}[l][l]{{\footnotesize $F$}}
\psfrag{G}[c][c]{{\footnotesize $F'$}}
\psfrag{H}[l][l]{{\footnotesize $H$}}
\psfrag{O}[l][l]{{\footnotesize $H^\perp$}}
\psfrag{X}[c][c]{{\footnotesize $\nicefrac{F'}{H}$}}
\psfrag{Y}[r][r]{{\footnotesize $\nicefrac{F}{H}$}}
\begin{figure}[hbt]
\begin{center}
\includegraphics[scale=.5]{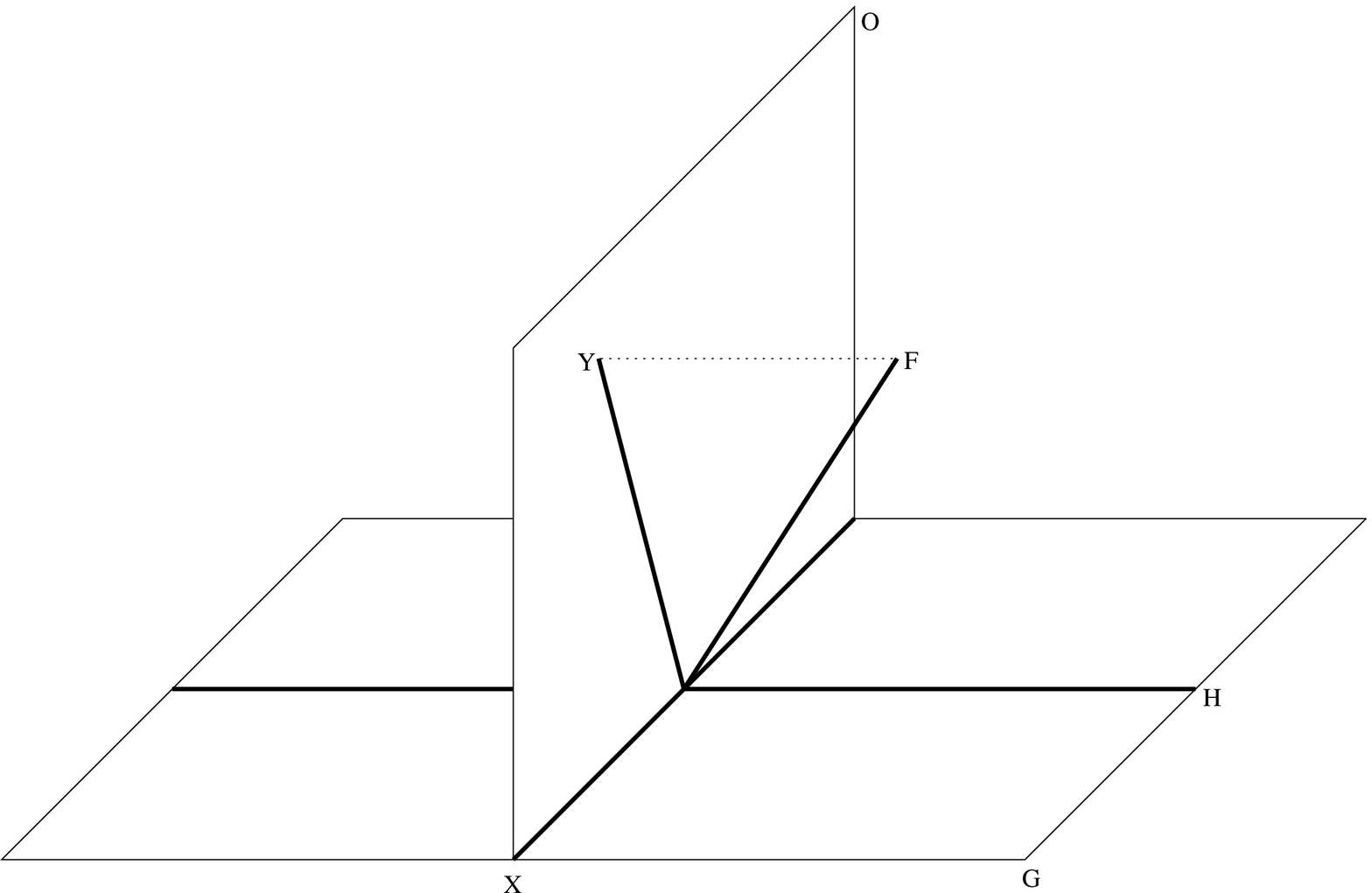} 
\end{center}
\caption{Lemma~\ref{l.BDP}} \label{f.BDP}
\end{figure}

\begin{proof}
It is simple to adapt the proof of Lemma~4.4 from \cite{BDP} to our
more general lemma, but let us spell out the details for the reader's convenience.

To simplify writing, in what follows $C$ indicates some positive number depending only on $d$, $K$ and $\ell$,
whose value may change from one line to another.

Take a cocycle $A$ preserving the subbundles $F$, $F' \supset H$ as in the statement of the lemma.
It is slightly more convenient to work with the inverse cocycle $B=A^{-1}$.
So the domination relations are reversed, that is,  $F >_{\ell} H$ and $F/H >_{\ell} F'/H$
with respect to $B$.

We claim that at each point
$$
\sin \angle(F,F') \ge \sin \angle(F,H) \cdot \sin \angle(\nicefrac{F}{H},  \nicefrac{F'}{H}) \, ;
$$
and in particular, since $\angle(F,H)$ and $\angle(F/H, F'/H)$ are not very small,
$\angle(F,F')$ is not very small either.
The proof consists on applying relation \eqref{e.BV} with 
$U=H$, $V=F' \cap H^\perp$ (which we naturally identify with $F'/H$), $W=F$:
$$
\sin \angle(F, \underbrace{\nicefrac{F'}{H} \oplus H}_{F'}) \ge \sin \angle (F, H) \cdot 
\sin \angle (F \oplus H, \nicefrac{F'}{H}) \, ,
$$
and noticing that $\angle (F \oplus H, \nicefrac{F'}{H}) = \angle(\nicefrac{F}{H} \oplus H,  \nicefrac{F'}{H})
 = \angle(\nicefrac{F}{H},  \nicefrac{F'}{H})$.

At each point $x$, take a basis $\{e_1(x), \ldots, e_d(x)\}$ of $E(x)$ with the following properties.
The first $\dim F$ vectors form an orthonormal basis of $F(x)$,
and the following $\dim F'$ vectors form an orthonormal basis of $F'(x)$,
so that the last $\dim H$ vectors belong to $H(x)$.
Since $\angle(F,F')$ is not very small, 
the changes of bases between $\{e_j(x)\}$
and any orthonormal basis of $E(x)$ are bounded by some constant $C$. 

With respect to the bases we chose, we can write
$$
B(x) = 
\begin{pmatrix} M(x) & 0 \\ 0 & N(x) \end{pmatrix}, 
\quad \text{where} \quad
N(x) = 
\begin{pmatrix} P(x) & 0 \\ Q(x) & R(x) \end{pmatrix}.
$$
As usual, we denote $M^k(x) = M(T^{k-1} x) \cdots M(x)$ etc.
By the dominance assumptions, there exists $0<\lambda<1$ depending only on $K$ and $\ell$
such that
\begin{equation}\label{e.NP}
\max \big( \|P^k(x)\|, \|R^k(x)\| \big) \le C \lambda^k \m(M^k(x)) \quad \text{for all $k>0$ and $x\in X$.}
\end{equation}
We have
$$
N^k(x) = 
\begin{pmatrix} P^k(x) & 0 \\ Q_k(x) & R^k(x) \end{pmatrix}, 
\text{ where }
Q_k(x) = \sum_{j=0}^{k-1} R^{k-j-1}(T^{j+1} x) Q(T^j x) P^j(x) \, .
$$
We estimate
\begin{alignat*}{2}
\|Q_k(x)\| &\le C \lambda^{k-1} \sum_{j=0}^{k-1} \m(M^{k-j-1}(T^{j+1} x)) \m(M^j(x)) &\quad &\text{(using \eqref{e.NP})} \\
           &\le C k \lambda^k \m(M^k(x)) &\quad &\text{(since $\m(XY)\ge \m(X)\m(Y)$)} \\
           &\le C \lambda^{k/2} \m(M^k(x)).
\end{alignat*}
The last inequality together with \eqref{e.NP} imply that
$$
\|N^k(x)\|  \le C \lambda^{k/2} \m(M^k(x)) \quad \text{for all $k>0$ and $x\in X$.}
$$
Coming back to the original norms, we have
$\|(B \restr F')^k(x)\|  \le C \lambda^{k/2} \m((B\restr F)^k(x))$.
This means that $F >_{L} F'$ with respect to $B$, or $F <_{L} F'$ with respect to $A$,
for some $L$ that depends only on $d$, $K$ and $\ell$.
\end{proof}

Another angle relation we will need is:
\begin{lemm}\label{l.BV}
There is $c_0>0$ such that
for any splitting of euclidian space into three non-zero subspaces $F$, $G$, $H$, we have
$$
\angle(H, F \oplus G) > c_0 \cdot \angle(H, F) \cdot  \angle(H, G) \cdot \angle(F/H, G/H) \, .
$$
\end{lemm}

\begin{proof}
Apply \eqref{e.BV} with $U=F$, $V=G$, $W=H$:
$$
\sin \angle(H, F \oplus G) \ge \sin \angle(H, F) \cdot \sin \angle(F \oplus H, G) \, .
$$
Now apply \eqref{e.BV} with $U=H$, $V = (F\oplus H) \cap H^\perp$ (which we naturally identify with $F/H$), $W=G$:
$$
\sin \angle(G, \underbrace{\nicefrac{F}{H} \oplus H}_{F \oplus H}) \ge \sin \angle (G, H) \cdot 
\sin \underbrace{\angle (G\oplus H, \nicefrac{F}{H})}_{\angle(\nicefrac{G}{H}, \nicefrac{F}{H})} \, .
$$
Therefore
$$
\sin \angle(H, F \oplus G) \ge 
\sin \angle(H, F) \cdot \sin \angle(H, G) \cdot \sin \angle(\nicefrac{F}{H}, \nicefrac{G}{H}) \, ,
$$
which implies the lemma.
\end{proof}

\subsection{Proof of Proposition {\rm\ref{p.two exp}}}

\begin{proof}
As already mentioned, the proof is by induction in the dimension $d$.

\medskip \noindent \emph{Starting the induction:}
Given $K$ and $\eps$, take a positive $\alpha \ll \eps/K$.
Apply Lemma~\ref{l.mane} with $d=2$ and $\eps/2$ in the place of $\eps$ 
to obtain $\ell$.
Now take a cyclic dynamical system of period at least $\ell$,
and a $2$-dimensional cocycle $A$ over it that is bounded by $K$,
has real eigenvalues, but no $\ell$-dominated splitting.
We can assume that the eigenvalues are different, otherwise there is nothing to prove.
So consider the two invariant one-dimensional bundles.
Using Lemma~\ref{l.mane}\footnote{Notice that the hypothesis \eqref{e.separation} in the lemma
is automatically satisfied here, because $d=2$.} 
we find a path of cocycles starting at $A$, 
always $\eps/2$-close to $A$, with the same spectrum as $A$, and such that
for the final cocycle $A'$, the two bundles form an angle less than $\alpha$
over at least one point $x_0 \in X$.
Now we consider another path of cocycles, starting at $A'$, ending at $A''$ 
and always $\eps/2$-close to $A'$,
such that the bundle maps are not altered except over the point $x_0$,
where the perturbations consist on composing with small rotations.
By Lemma~\ref{l.dim 2}, this can be done so that the upper Lyapunov exponent
strictly decreases along the path, and at the endpoint $A''$ the two exponents are equal.
The concatenation of the two paths described above is a path with all the required properties.
This shows that the proposition is true when $d=2$.

\medskip \noindent \emph{The induction step:}
Take $d \ge 3$, and assume the proposition is true for all dimensions $d'$ between $2$ and $d-1$.
Let $K>1$ and $\eps>0$ be given.
Reducing $\eps$ if necessary (recall Remark~\ref{r.epsilon}), we assume that 
any $\eps$-perturbation of a cocycle bounded by $K$ is bounded by $2K$,
and that we are allowed to $\eps$-perturb not only a given cocycle $A$ but also the inverse cocycle $A^{-1}$.
Let $\ell_0\in \PT$ be the maximum of $\ell(2K, \eps/2, d')$ where $2 \le d' \le d-1$.

Let $\ell \gg \ell_0$ (how large is necessary will become clear along the proof).

Take a cyclic $d$-dimensional cocycle $(X,T,E,A)$ bounded by $K$ and of period at least~$\ell$.
Also suppose that $A$ has only real eigenvalues, and has no $\ell$-dominated splitting of index $i$.
We will complete the proof assuming $i \le d-2$.
Then the remaining case $i=d-1$ will follow from the case $i=1$ applied to the inverse cocycle.
Write for simplicity $\lambda_j = \lambda_j(A)$ for the Lyapunov exponents.

We assume that $\lambda_i < \lambda_{i+1}$, because otherwise we simply take a constant path
of cocycles.
Let $F \oplus F'$ be the invariant splitting so that
the exponents along $F$ are $\lambda_1$, \ldots, $\lambda_i$,
and the exponents along $F'$ are $\lambda_{i+1}$, \ldots, $\lambda_d$.
By assumption, 
$$
F \not <_\ell F'.
$$
We consider separately two cases according to the multiplicity of the upper exponent~$\lambda_d$:

\medskip \noindent \emph{The case $\lambda_{d-1} = \lambda_d$:}
By assumption, the cocycle has only real eigenvalues.
Thus we can choose an one-dimensional invariant subbundle $H$ with an exponent $\lambda_d$.

In the case that $F/H \not<_{\ell_0} F'/H$ 
we apply the induction assumption to the cocycle $A/H$ (which satisfy the real eigenvalues hypothesis),
obtaining a path of cocycles mixing its $i$-th and $(i+1)$-th exponents,
without changing the others.
Then, using the second procedure explained in \S~\ref{ss.procedures},
we extend these cocycles to the whole bundle $E$.
The direction $H$ is kept invariant and so the upper exponent $\lambda_d$ does not change.

Thus assume that $F/H <_{\ell_0} F'/H$.
Recall that $F \not<_\ell F'$ and $\ell \gg \ell_0$;
thus by Lemma~\ref{l.BDP} we have $F \not<_{\ell_0} H$.
Let $G$ be the invariant subbundle associated to the exponents bigger than
$\lambda_i$ and smaller than $\lambda_d$
(so $G$ is zero if $\lambda_{i+1}=\lambda_d$).
Then $F \not<_{\ell_0} G \oplus H$. 
We apply the induction assumption to the cocycle $A$ restricted to 
the smaller bundle $F \oplus G \oplus H$,
obtaining a path of cocycles mixing its $i$-th and $(i+1)$-th exponents,
without changing the others.
Then, using the first procedure from \S~\ref{ss.procedures},
we extend these cocycles to the whole bundle $E$.

\medskip \noindent \emph{The case $\lambda_{d-1} < \lambda_d$:}
We split invariantly $F' = G \oplus H$, where
the Lyapunov exponents along $G$ are
$\lambda_{i+1}$, \ldots, $\lambda_{d-1}$,
and $H$ is one-dimensional with Lyapunov exponent $\lambda_d$.
Then we are in position to apply the following lemma:

\begin{lemm}\label{l.main}
Given $d$, $K$, $\ell_0$ and $\eps$, there exists $\ell$
with the following properties:
Let $(X,T,E,A)$ be a cyclic $d$-dimensional cocycle bounded by $K$ 
and of period at least~$\ell$.
Assume that $A$ has an invariant splitting $E = F \oplus G \oplus H$ such that 
$\dim H = 1$, 
the exponents on $F$ are smaller  than the exponents on $G$,
which in turn are smaller than the exponent on $H$.
Also assume that
$$
F \not<_{\ell} G \oplus H \, .
$$
Then there exists
an $\eps$-short path of cocycles $A_t$, $t\in [0,1]$ starting at $A$,
all of them with the same eigenvalues,
such that, denoting by 
$F_t$, $G_t$, $H_t$ the continuations of the bundles $F$, $G$, $H$,
we have 
\begin{equation}\label{e.alternatives}
F_1 \not <_{\ell_0} G_1 \quad \text{or} \quad F_1 / H_1  \not <_{\ell_0} G_1 / H_1.
\end{equation}
\end{lemm}

We postpone the proof of the lemma to the next subsection,
and see how it permits us to conclude:

\begin{itemize}
\item In the case $F_1 \not <_{\ell_0} G_1$,
we apply the induction assumption to the cocycle $A_1$ restricted to the bundle $F_1 \oplus G_1$
(which has only real eigenvalues),
obtaining a path of cocycles mixing the $i$-th and $(i+1)$-th exponents,
without changing the others.
Then, using the first procedure explained in \S~\ref{ss.procedures},
we extend these cocycles to the whole bundle $E$.
This does not alter the upper exponent $\lambda_d$
(although the $H_1$ direction is not preserved).
The desired path of cocycles is obtained by concatenation.

\item
The case $F_1/H_1 \not <_{\ell_0} G_1/H_1$ is similar:
we apply the induction assumption to the
quotient cocycle $A_1/H_1$, and then we use the second extension procedure from \S~\ref{ss.procedures}.
\end{itemize}
So we have proved Proposition~\ref{p.two exp} modulo Lemma~\ref{l.main}.
\end{proof}

\subsection{Proof of Lemma {\rm\ref{l.main}}}

Let us begin with an informal outline of the proof.
We assume that $F<G$ and $F/H < G/H$, otherwise there is nothing to show.
By Lemma~\ref{l.BV},
this implies that each time that $\angle(H, F \oplus G)$ is small,
then either $\angle(H,F)$ or $\angle(H,G)$ is small, but not both.
Thus the objective is to perturb the cocycle without changing the eigenvalues 
so that at some point of the orbit, the new space $H$ (i.e., the continuation of $H$) is close to
the new $F \oplus G$ but far from the new $F \cup G$;
this breaks a dominance relation and we are done.
Coming back to the unperturbed cocycle, by Lemma~\ref{l.BDP} we have $F \not < H$
(otherwise we would get $F < G \oplus H$).
Using Lemma~\ref{l.mane} we can perturb the cocycle (but not the eigenvalues) 
so to make $\angle(H,F)$ extremely small at some point $x_0$.
We can still assume that $F<G$ and $F/H < G/H$ because otherwise we are done.
Iterating negatively from $x_0$, these angles must remain small for a long time, but not forever.
So we take a point $x_1$ where $\angle(H,F)$ is small, but not extremely small, 
and remain small for many positive iterates, until we reach $x_0$.
Choose eigenvectors $h \in H(x_1)$, $g\in G(x_1)$ such that $\|g\|$
is much (but not extremely) smaller than $\|h\|$.
We follow the iterates of the vector $\tilde h = h+g$:
during the (long) time the iterates of $h$ remain close to $F$, 
the iterates of $\tilde h$ remain close to $F \oplus G$.
On the other hand, since $F<G$, 
an iterate of $\tilde h$ gets far from $F$ while it's still close to $F \oplus G$
and not yet close to $G$ -- see Figure~\ref{f.proof}.
Sometime later (and before we come back to $x_1$), 
the iterates of $\tilde h$ get again close to $H$:
that happens because the expansion rate of the eigenvector $g$ until the period 
is less than the expansion rate of $h$.
It is then simple to perturb the cocycle without changing $F$, $G$ nor the eigenvalues 
so that the new $H$ follows the route of~$\tilde h$.

\begin{figure}[hbt]
\psfrag{F}[r][r]{{\footnotesize $F$}}
\psfrag{G}[l][l]{{\footnotesize $G$}}
\psfrag{H}[c][c]{{\footnotesize $H$}}
\psfrag{S}[l][l]{{\footnotesize $F \oplus G$}}
\begin{center}
\includegraphics[scale=.4]{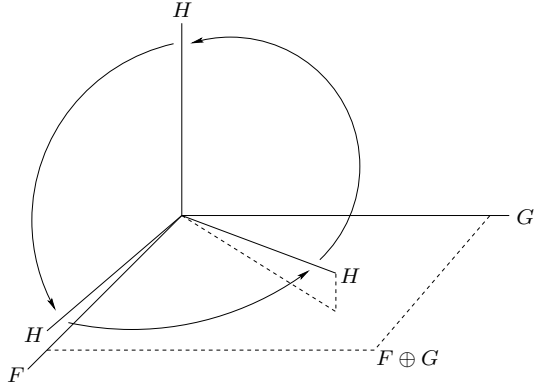}
\end{center}
\caption{\small The proof of Lemma~\ref{l.main}: For the perturbed cocycle the spaces
$F$ and $G$ are always apart. The figure shows the position of the one-dimensional
space $H$ at three different times.
At one of these, we see a small $\angle(F/H, G/H)$.}
\label{f.proof}
\end{figure}

\begin{proof}[of Lemma~{\rm\ref{l.main}}]  
For clarity, we divide the proof into parts.

\medskip 
\noindent \emph{Part 0: Fixing several constants.}	
Let $d$, $K$, $\ell_0$ and $\eps$ be given.
Reducing $\eps$ if necessary  (recall Remark~\ref{r.epsilon}), we assume that any $\eps$-perturbation of a cocycle bounded by $K$ is bounded by $2K$.

Fix numbers $\theta_1 > \theta_0 > 0$ depending only on $K$ and $\ell_0$
such that the following holds:
For any cocycle $A$ bounded by $2K$ that has invariant subbbundles $F$ and $G$
with $F <_{\ell_0} G$, we have $\angle(F,G)>\theta_0$ at each point.
Moreover, 
\begin{equation}\label{e.cone}
x \in X, \ 
u \in F^{\theta_1} (x), \ g \in G(x) 
\ \Rightarrow \  
\frac{\|A^{\ell_0}(x) \cdot u\|/\|u\|}{\|A^{\ell_0}(x) \cdot g\|/\|g\|} < 0.6,
\end{equation}
where $F^{\theta_1} (x)$ indicates the cone of size $\theta_1$ around $F(x)$, that is,
the set of $v \in E(x)$ such that $\angle (v, F(x)) < \theta_1$. 
We will keep using this cone notation in what follows.

In the rest of the proof, $c_1$, $c_2$, \ldots, $c_8$ will indicate certain positive numbers that
we haven't bothered to calculate explicitly, but depend only on $K$ and $\ell_0$.

We choose a positive $\beta$ satisfying 
$$
\beta < \min \left( \theta_1 , \theta_0/2 , \theta_0/c_7 \right) \, ,
$$
and a positive $\delta$ such that 
\begin{equation}\label{e.feiomashonesto}
c_3 \delta / \beta < \min \left( 1/2 , \eps/(8K) \right) \, .
\end{equation}
Fix an integer $m_0 \ge 1$ such that $\delta^{-1} 0.6^{m_0} < 1$.
Then choose $\gamma>0$  with the following property: 
If two vectors form an angle less than $\gamma$,
then their images by any linear map with norm less than $(2K)^{\ell_0 m_0}$
form an angle less than $\beta$.

Let $\ell_1$ be the number given by Lemma~\ref{l.mane}, 
applied with $\gamma$ in the place of $\alpha$, $\eps/2$ in the place of $\eps$, and all smaller $d$.
Take $\ell = L(d, 2K, \ell_1)$,
where the function $L$ is given by Lemma~\ref{l.BDP}.
Increasing $\ell$ if necessary, 
we assume that
\begin{equation}\label{e.ell}
\ell > \ell_0 m_0 \quad  \text{and} \quad
K^{\ell_0} 0.6^{\ell/\ell_0} < 1.
\end{equation}

\medskip 
\noindent \emph{Part 1: Preliminary perturbation.}	
Let $(X,T,E,A)$ be a cocycle as in the statement of the lemma: 
$X$ has cardinality $n \ge \ell$; $E$ has dimension $d$; the cocycle is bounded by $K$;
there is an invariant splitting $F \oplus G \oplus H$
with exponents along $F$ less than those along $G$, which are less the exponent along 
the one-dimensional bundle $H$;
and  $F \not<_{\ell} G \oplus H$.

We can assume that 
$$
F <_{\ell_0} G \quad \text{and} \quad F / H  <_{\ell_0} G / H,
$$
because otherwise there is nothing to show (just take a constant path).
So applying Lemma~\ref{l.BDP} (with $F' = G \oplus H$),
we conclude that 
$$
F \not<_{\ell_1} H \, .
$$

Notice that for any $x\in X$, choosing an eigenvector $g\in G(x)$ of $A^n(x)$
(where $n = \# X$),
we can write
\begin{alignat*}{2}
\m(A^n(x) \restr H) 
&>   \|A^n(x) \cdot g\|/\|g\| &\quad &\text{(since the exponent along $H$ is the biggest)} \\ 
&> K^{-\ell_0} 2^{\lfloor n/\ell_0\rfloor} \|A^n(x) \restr F\| &\quad &\text{(because $F <_{\ell_0} G$)} \\
&> \|A^n(x) \restr F\| &\quad &\text{(by \eqref{e.ell} and $n \ge \ell$).}
\end{alignat*} 
Hence the cocycle $A \restr F \oplus H$ satisfies
requirement~\eqref{e.separation} from Lemma~\ref{l.mane}.
Applying the lemma 
together with the first extension procedure from \S~\ref{ss.procedures},
we obtain a $(\eps/2)$-short path of cocycles, all with the same eigenvalues,
so that at the end the angle between the continuation of $F$ and the continuation of $H$
is smaller than $\gamma$ at some point $x_0$.

Let $B$ indicate the endpoint of this path of cocycles. 
To simplify writing, the invariant bundles will still be indicated by $F$, $G$, $H$.
We can assume $F <_{\ell_0} G$ and $F/H <_{\ell_0} G/H$ with respect to $B$, 
because otherwise there is nothing to prove. 

\medskip 
\noindent \emph{Part 2: Second perturbation.}	
Recall that $H(x_0) \subset F^\gamma(x_0)$.
Let $k_0$ be the least positive integer such that 
$H(T^{-\ell_0 k_0} x_0) \not \subset F^\beta(T^{-\ell_0 k_0} x_0)$.
Notice that $k_0$ indeed exists and satisfies $k_0 \le \lfloor n/\ell_0 \rfloor$:
in the opposite case, using \eqref{e.cone} and \eqref{e.ell} 
we obtain
$$
\m(B^n(x_0) \restr H) 
< K^{\ell_0} 0.6^{\lfloor n/\ell_0 \rfloor} \|B^n(x_0) \restr G\|
< \|B^n(x_0) \restr G\|,
$$
which contradicts the assumption that the Lyapunov exponent along $H$ is bigger than those along $G$.)
Moreover, $k_0 > m_0$ due to the choice of $\gamma$.

Let $x_1 = T^{-\ell_0 (k_0-1)} x_0$.
Then $\angle (H(x_1), F(x_1)) < \beta < \theta_0/2$, 
and since $\angle(F,G)\ge \theta_0$, we have
$\angle (H(x_1) , G(x_1)) > \theta_0/2$.
On the other hand, 
by the minimality of $k_0$ we have
$\angle (H(T^{-\ell_0} x_1) , F(T^{-\ell_0} x_1))$ is at least $\beta$, 
and since the cocycle $B$ is bounded by $2K$,
we obtain
$$
\angle (H(x_1) , F(x_1)) > c_1 \beta.   
$$
Using also that $\angle (F/H, G/H) > \theta_0$,
Lemma~\ref{l.BV} gives
\begin{equation}\label{e.kingsofleon}
\angle(H(x_1), F(x_1) \oplus G(x_1)) > c_0 \cdot (c_1 \beta) \cdot (\theta_0/2) \cdot \theta_0 = c_2 \beta.
\end{equation}

Let $h$ be a unit vector on the direction of $H(x_1)$,
and let $g\in G(x_1)$ be an eigenvector for $B^n(x_1)$ with $\|g\|=\delta$.

Define linear maps $S_t : E(x_1) \to E(x_1)$, where $t\in[0,1]$, by
$$
S_t \restr F(x_1) \oplus G(x_1) = \Id , \qquad
S_t : h \mapsto  h+ tg .
$$
It follows easily from \eqref{e.kingsofleon} that
$$
\|S_t  - \Id\| < \frac{1}{\sin (c_1\beta)}\frac{\|g\|}{\|h\|} < c_3\delta/\beta \quad \text{for all $t\in [0,1]$.}
$$
Define other linear maps $U_t : E(x_1) \to E(x_1)$, where $t\in[0,1]$, by
$$
U_t \restr F(x_1) \oplus G(x_1) = \Id , \qquad
U_t : = B^n(x_1) \cdot h \mapsto  B^n(x_1) \cdot (h+tg).
$$
Since $g$ and $h$ are eigenvectors of $B^n(x_1)$, and the 
eigenvalue associated to $h$ has bigger modulus,
we have $U_t (h) = h + \rho tg$, where $|\rho|<1$.
Thus we can estimate exactly as before:
$$
\| U_t - \Id \| < c_3\delta/\beta < 1/2 \quad \text{for all $t\in [0,1]$.}
$$
In particular (using the formula $(\Id-X)^{-1} = \Id + X + X^2 + \cdots$),
we have  
$$
\|U_t^{-1} - \Id \| \le 
2\| U_t - \Id \| < 2c_3\delta/\beta.
$$
Define a path of cocycles $B_t$, $t\in [0,1]$ starting at $B_0=B$ as follows:
$$
B_t(x_1)        = B(x_1) \circ S_t , \qquad
B_t(T^{-1} x_1) = U_t^{-1} \circ B(T^{-1} x_1), 
$$
and $B_t(x) = B(x)$ for $x \not \in \{ x_1, T^{-1} x_1\}$.
By the estimates above and \eqref{e.feiomashonesto}, 
the path is $(\eps/2)$-short, that is, $\|B_t-B\|<\eps/2$. 
we see that all cocycles $B_t$ have the same spectrum as $B$.
The continuation of the $B$-invariant splitting $F \oplus G \oplus H$ 
is $F \oplus G \oplus H_t$,
where $H_t(x_1)$ is spanned by $h$
and $H_t(T^j x_1)$ is spanned by $B^j(x_1)\cdot (h+tg)$ for $0<j<n$.

\medskip 
\noindent \emph{Part 3: Finding a point that breaks dominance.}	
We will show that there is a point over which 
$\angle(F/H_1, G/H_1)$ is small.

Define $h_j = B^{\ell_0 j}(x_1) \cdot h$ and  $g_j = B^{\ell_0 j}(x_1) \cdot g$.
For $0 \le j \le m_0$,
we have $h_j \in F^\beta(T^{\ell_0 j} x_1)$.
Therefore, by \eqref{e.cone},
$$
\frac{\|h_j \|}{\| g_j\|} < 0.6^j \frac{\|h\|}{\|g\|} = \delta^{-1} 0.6^j \, .
$$
By definition of $m_0$ we have $\delta^{-1} 0.6^{m_0}<1$; 
so let $m\ge 1$ be the least number so that 
$\|h_m\|/\|g_m\|<1$.
By minimality of $m$, we have $\|h_m\|/\|g_m\| \ge K^{-2\ell_0}$.
That is, the norms of the vectors $h_m$ and $g_m$ are far from $0$ and $\infty$.
Moreover, since 
$$
\angle(h_m, g_m) \ge \angle(F,G) - \angle(h_m, F) > \theta_0 - \beta > \theta_0/2, 
$$
the vector $h_m + g_m$ cannot be much smaller than $h_m$ or $g_m$.
Thus we can say that
$$
\|h_m\|, \|g_m\|, \|h_m+g_m\| \text{ are between $c_4^{-1}$ and $c_4$.}
$$

Let $y = T^{\ell_0 m} x_1$; then $h_m + g_m$ spans $H_1(y)$.
We will show that at the point $y$, the space $H_1$ makes an small angle with $F\oplus G$,
but not with $F$ nor $G$;
see Figure~\ref{f.endproof}.
\begin{figure}[hbt]
\psfrag{F}[r][r]{{\footnotesize $F$}}
\psfrag{G}[l][l]{{\footnotesize $G$}}
\psfrag{H}[c][c]{{\footnotesize $H$}}
\psfrag{H1}[c][c]{{\footnotesize $H_1$}}
\psfrag{f}[c][c]{{\footnotesize $f$}}
\psfrag{g}[c][c]{{\footnotesize $g_m$}}
\psfrag{h}[c][c]{{\footnotesize $h_m$}}
\psfrag{w}[c][c]{{\footnotesize $w$}}
\begin{center}
\includegraphics[scale=.55]{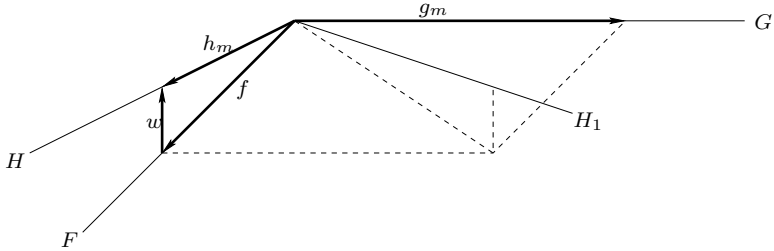}
\end{center}
\caption{\small The spaces at the point $y$.}
\label{f.endproof}
\end{figure}

The \emph{distance} between a vector $u$ and a space $V$ is
$$
d(u, V) = \inf_{v \in V} \|u-v\| = \|u\| \sin \angle(u,V).  
$$

Let $f$ be the vector in $F(y)$ that is closer to $h_m$. 
Since $h_m \in H(y) \subset F^\beta(y)$,
the vector $w = h_m - f$ has norm $\|w\| = d(h_m, F(y)) < c_4 \beta$.
The vectors $h_m + g_m$ and $f + g_m$ have comparable norms,
and their difference is $w$, we conclude that they form an angle less than $c_5\beta$.
In particular, $\angle (H_1, F \oplus G) < c_6 \beta$ at the point $y$.

On the other hand,
$$
d(g_m, F(y)) = \|g_m\| \sin \angle(F,G) > \|g_m\|\sin \theta_0;
$$
so the distance from $h_m + g_m = f + g_m + w$ to $F(y)$ is at least 
$\|g_m\|\sin \theta_0 - \|w\|$.
Thus we obtain that $\angle(H_1(y) , F(y)) > c_6$.
An entirely analogous reasoning gives $\angle(H_1(y) , G(y)) > c_6$.

It follows from Lemma~\ref{l.BV} and the previous estimates that,
at the point $y$,
$$
\angle(F/H_1, G/H_1) < c_0^{-1} \cdot 
\frac{\angle(H_1, F \oplus G)}{\angle(H_1, F) \cdot \angle(H_1, G)} <
c_7 \beta.
$$
Since $c_7\beta < \theta_0$,
we have $F/H_1 \not <_{\ell_0} G/H_1$.
So concatenating the paths from $A$ to $B$ and from $B$ to $B_1$,
we obtain a path of cocycles that falls into the second alternative in  \eqref{e.alternatives}.

This proves Lemma~\ref{l.main} and therefore Proposition~\ref{p.two exp}.
\end{proof}

\section{Mixing Lyapunov Exponents: General Results}\label{s.up}

\subsection{The First Main Result}


\begin{theo}\label{t.main}
For any $d \ge 2$, $K>1$, $\eps>0$, 
there exists $L \in \PT$
such that the following holds:
Let $(X,T,E,A)$ be a cyclic $d$-dimensional cocycle  bounded by $K$ and of period at least $L$
that admits no $L$-dominated splitting.
Let $\sigma \in \cS_d$ be a graph such that $\sigma \gtr \bsigma (A)$
and $\sigma_d = \bsigma_d(A)$.

Then there exists
an $\eps$-short path of cocycles $A_t$, $t\in [0,1]$ starting at $A$ 
such that the path of graphs $\bsigma(A_t)$ is non-decreasing
and $\bsigma(A_1) = \sigma$.
\end{theo}

Theorem~\ref{t.periodic} from the introduction follows directly 
using Franks Lemma and basic properties of dominated splittings.

\begin{rema}
A consequence of Theorem~\ref{t.main} is that with a perturbation $A_1$
we can make all Lyapunov exponents equal.
However it may be impossible to make $A_1^n(x)$ into a homothety at a given point $x$ (where $n$ is the period):
see the footnote on page~1308 of \cite{BGV}.
So it is not possible to have an statement similar to Theorem~\ref{t.main}
where instead of ``mixing'' the Lyapunov exponents,
one mixes the log's of the singular values of $A^n(x)$.
\end{rema}

The proof of Theorem~\ref{t.main} goes roughly as follows.
Starting with a cocycle without $L$-dominated splitting, where $L\gg 1$,
we first perturb it to make all eigenvalues real.
Then we want to apply Proposition~\ref{p.two exp} a certain number $N$ of times until
the spectrum gets very close to $\sigma$,
and perturb a last time to get spectrum exactly $\sigma$.
To make all this work, each time we want to apply Proposition~\ref{p.two exp}
to a cocycle, it must be non or weakly dominated.
Thus the perturbations must be sufficiently small so that no strong domination is created.
This is done with a careful choice of the quantifiers, where it is important to have an a~priori 
bound for the number $N$ of times that we will have to apply Proposition~\ref{p.two exp}.

\medskip

Now we give the detailed proof.
The first auxiliary result says how we can get rid of complex eigenvalues:

\begin{prop}[(Getting real eigenvalues)]\label{p.BGV}
Given $d$, $K$ and $\eps$, there exists $m=m(d,K,\eps)$ such that 
if $(X,T,E,A)$ is a cyclic $d$-dimensional cocycle  bounded by $K$ and of period at least $m$ 
then there is an $\eps$-short path of cocycles $A_t$, $t\in [0,1]$ starting at $A$,
all of them with the same Lyapunov spectrum,
such that $A_1$ has only real eigenvalues.
\end{prop}

This proposition is essentially contained in \cite{BGV}, but not in the exact form we need, 
so we give a proof.
The essential fact is the following lemma, where $R_\theta$ denotes the rotation of angle $\theta$ in $\R^2$.

\begin{lemm}[(Lemma 6.6 from \cite{BC})]\label{l.BC}
For every $\eps>0$ there exists $m$ with the following property:
For any finite sequence of matrices $A_1, \ldots, A_n$ in $\SL(2,\R)$ of length $n \ge m$
there exist numbers $\theta_1, \ldots, \theta_n$ in the interval $(-\eps,\eps)$ such that
the matrix $R_{\theta_n}A_n \cdots R_{\theta_1}A_1$ has real eigenvalues.
\end{lemm}

\begin{rema}
In fact, it is possible to take all $\theta_i$ equal in Lemma~\ref{l.BC};
this follows immediately from Lemma~C.2 in \cite{ABD}.
\end{rema}

\begin{proof}[of Proposition~{\rm\ref{p.BGV}}]  
The proof is by induction on the dimension $d$.
For $d=1$ the result is true.
Next consider a $2$-dimension cocycle $(X,T,E,A)$, where $T$ is a cyclic dynamical system
of large period $n$.
We can assume that the linear map $A^n(x)$ has complex eigenvalues and
in particular preserves orientation.
So with appropriate choices of bases over each point, the matrices of the cocycle have positive determinant.
Normalizing the determinant, we obtain matrices $A_1,\ldots,A_n$ in $\SL(2,\R)$.
By Lemma~\ref{l.BC}, $R_{\theta_n}A_n \cdots R_{\theta_1} A_1$ has real eigenvalues
for some choice of small numbers $\theta_i$.
Consider the path of matrices $M_t = R_{\theta_n t}A_n \cdots R_{\theta_1 t} A_1$.
Take the smallest $t_0 \in [0,1]$ such that $M_{t_0}$ has real eigenvalues.
Then by perturbing the cocycle itself by composing with rotations $R_{\theta_i t_0 t}$, $t\in [0,1]$,
we obtain the desired path of cocycles.

Now let $d>2$ be arbitrary and assume the proposition is true for any dimension less than $d$.
Any cocycle over a cyclic dynamical system has a $2$-dimensional invariant subbundle $F$.
Apply the induction assumption to the cocycles $A \restr F$ and $A / F$.
We obtain paths of cocycles on $F$ and $E/F$ with constant Lyapunov spectra, and only real eigenvalues at
the endpoints.
By the extension procedures of \S~\ref{ss.procedures} we obtain the desired path.
\end{proof}

The next result says that once all eigenvalues are real (which is the case when the spectrum is simple, for example),
we can always perturb them a little:

\begin{lemm}[(Arbitrary perturbation of the spectrum)]\label{l.triangular}
Given $d$, $K$ and $\eps$, there exists $\delta=\delta(d,K, \eps)>0$ with the following properties.
Let $(X,T,E,A)$ be a cyclic $d$-dimensional cocycle bounded by $K$ with only real eigenvalues.
Then for any path $\sigma:[0,1] \to \cS_d$ satisfying
$\sigma(0) =\bsigma(A)$ and $|\sigma(t) - \sigma(0)|<\delta$,
Then there exists
an $\eps$-short path of cocycles $A_t$, $t\in [0,1]$ starting at $A$,
such that
$\bsigma(A_t) = \sigma(t)$ for each~$t$.
\end{lemm}

\begin{proof}
Let $(X,T,E,A)$ be a $d$-dimensional cocycle.
Assuming it has only real eigenvalues, it is possible to find invariant subbundles
$F_1 \subset F_2 \subset \cdots \subset F_{d-1}$ with $\dim F_i = i$ (that is, an invariant flag).
For each $x\in X$, let $\{e_1(x), \ldots, e_d(x)\}$ be an orthonormal basis of $E(x)$
such that $e_i(x) \in F_i(x)$ for each $i$.
With respect to those basis, the cocycle is expressed by triangular matrices.
By multiplying them by appropriate diagonal matrices (close to the identity), 
we can perform any prescribed sufficiently small 
perturbation of the Lyapunov exponents, keeping the eigenvalues real.
\end{proof}

With the next lemma we control the effect of our perturbations on the non-dominance of the cocycle:

\begin{lemm}[(Stability of non-domination)]\label{l.eta}
Given $d$, $K$ and $\ell$, there exists $\eta = \eta(d,K,\ell)>0$ such that
if $A$ is a $d$-dimensional cocycle bounded by $K$ that has no $2\ell$-dominated splitting 
of index~$p$ 
then no $\eta$-perturbation of $A$ has a $\ell$-dominated splitting
of index~$p$.
\end{lemm}

\begin{proof}
This is merely the contraposition of the well-known fact that dominated splittings persist under perturbations.
(See e.g.\ Appendix B.1 in \cite{BDV}.)
\end{proof}

In the proof of Theorem~\ref{t.main}
it is necessary to have an \textit{a priori} bound on the number of times
we will have to apply Proposition~\ref{p.two exp}.
We will show that given two convex graphs 
$\sigma' \le \sigma''$ in $\cS_d$ that agree at the endpoints,
there exists a non-decreasing path of graphs inside $\cS_d$ starting at $\sigma'$ and ending close to $\sigma''$; 
this path is the concatenation of a certain number $N$ of paths of graphs,
each of these consisting of moving a single vertex of the graph.
See Figure~\ref{f.accessib}.
\begin{figure}[hbt]
\begin{center}
\psfrag{1}[l][l]{{\tiny $\sigma_1$}}
\psfrag{2}[c][c]{{\tiny $\sigma_2$}}
\includegraphics[scale=.3]{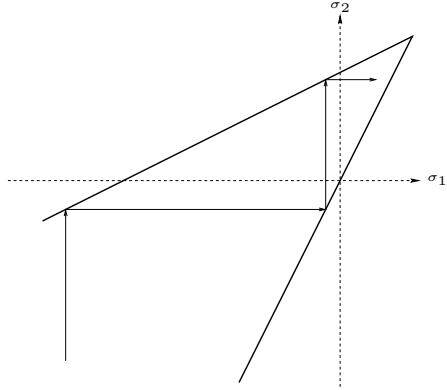}
\end{center}
\caption{\small The region between the two lines is a slice of $\cS_3$ of constant $\sigma_3>0$. A path of spectra as in Lemma~\ref{l.zigzag} is represented, with arrows joining points $\sigma^{(j)}$ and $\sigma^{(j+1)}$.}
\label{f.accessib}
\end{figure}
Moreover the number $N$ can be bounded in terms of the sizes of $\sigma'$, $\sigma''$
and the desired accuracy of the approximation to $\sigma''$.
Precisely speaking, we have the following lemma:

\begin{lemm}[(Path of spectra)]\label{l.zigzag}
Given $d \ge 2$, $c>0$ and $\delta>0$ there exists $N=N(d,c,\delta)$ with the following properties.
Let $\sigma'$, $\sigma'' \in \cS_d$ be such that
$|\sigma'_i| \le c i$ for each $i$,
$\sigma'_d = \sigma''_d$,
and $\sigma' \les \sigma''$.
Then there exists a sequence of points in $\cS_d$ 
$$
\sigma^{(0)} \les \sigma^{(1)} \les \cdots \les \sigma^{(N)}
$$
such that $\sigma^{(0)} = \sigma'$,
$\sigma^{(N)} \les \sigma''$,
$\sigma^{(N)}$ is $\delta$-close to $\sigma''$,
and there are $i_1$, \ldots, $i_N \in \{1, \ldots, d-1\}$ such that 
$$
1 \le j \le N, \  0 \le i \le d, \   i \neq i_j \quad \Rightarrow \quad
\sigma^{(j-1)}_i = \sigma^{(j)}_i \, .
$$
\end{lemm}

\begin{rema}
It is known in majorization theory that
if $\sigma'$ majorizes $\sigma''$ then one can generate $\sigma''$ from $\sigma'$ with a
countable sequence of elementary Robin Hood operations (recall Remark~\ref{r.robinhood}),
see e.g.\ \cite{MOA}, p.~82.
Lemma~\ref{l.zigzag} is a finer version of this fact:
it says that the number of operations $N$ needed to achieve a certain precision $\delta$
has a uniform bound depending on the ``size'' of the graphs (measured by $d$ and $c$).
\end{rema}

\begin{proof}[of Lemma~{\rm\ref{l.zigzag}}]  
Consider any two vectors $\sigma' \les \sigma''$ in $\R^{d+1}$ 
that are convex, that is $\Delta^2 \sigma'$, $\Delta^2 \sigma'' \gtr 0$
in the difference operator notation from \S~\ref{ss.difference operator}.
Also assume that the two graphs
agree at the endpoints, i.e.\ 
$\sigma_0'=\sigma_0''$ (it is convenient not to assume these are zero) and $\sigma_d'=\sigma_d''$.
We will  find a non-decreasing sequence $(\sigma^{(j)})_{j\ge 0}$ 
of convex vectors in $\R^{d+1}$ starting at $\sigma'$ 
and such that any two consecutive elements differ at most one coordinate, 
which is not the zeroth neither the $d$-th.
We will show that the sequence converges exponentially fast to $\sigma''$ 
with some speed that depends only on $d$.
Once this is done, it is trivial to write a formula for $N=N(d,c,\delta)$
that meets the requirements of the lemma.

The case $d=2$ is immediate: we can reach $\sigma''$ exactly with a single step.
Fix $d>2$ and assume that the procedure known for smaller dimensions.

Let $\sigma^{(0)} = \sigma'$.
Let $j\ge 0$ and assume that $\sigma^{(j)}$ was already defined.
Let $A_j$ be the area between the graphs $\sigma^{(j)}$ and $\sigma''$;
this quantity is uniformly comparable with the distance between the two graphs.
Let $D_j = \max_{0 \le i \le d-2} \Delta^2 \sigma^{(j)}_i$.
Take $i_j \in \{1, \ldots,d-1 \}$ such that $\Delta^2 \sigma^{(j)}_{i_j - 1} = D_j$.
Define $\sigma^{(j+1)}_i = \sigma^{(j)}_i$ for $i \neq i_j$
and
$$
\sigma^{(j+1)}_{i_j} = \min \left( \sigma''_{i_j}, \frac{\sigma^{(j)}_{i_j - 1} + \sigma^{(j)}_{i_j + 1}}{2} \right) \, .
$$
If $\sigma^{(j+1)}_{i_j} = \sigma''_{i_j}$ then we break the graph into two parts,
and define the rest of sequence using the induction hypothesis.
In the remaining case, we have $A_{j+1} = A_j - D_j/2$.
It follows from Lemma~\ref{l.nonlinearity} that $A_j \le (d^3/4) D_j$.
Thus $A_{j+1}/A_j \le 1-2/d^3$.
This gives the exponential convergence.
\end{proof}

\begin{proof}[of Theorem~{\rm\ref{t.main}}]  
Let $d$, $K$ and $\eps$ be given.
Reducing $\eps$ if necessary  (recall Remark~\ref{r.epsilon}), we assume that any $\eps$-perturbation of a cocycle bounded by $K$ is bounded by $2K$. 
Let $\eps_0 = \eps/2$  and let 
$\delta = \delta(d,2K,\eps_0)$ be given by Lemma~\ref{l.triangular}.
Let $N = N(d, \log K, \delta)$ be given by Lemma~\ref{l.zigzag}.

We recursively define numbers $\eps_1$, $\ell_1$, \ldots, $\eps_{N+1}$, $\ell_{N+1}$ 
(in this order) by the formulas:
\begin{align*}
\eps_j &=
\begin{cases}
\eps/4                                                   &\quad\text{if $j=1$} \\ 
\min \big( \eps/2^{j+2}, \eta(d, 2K, \ell_{j-1} ) \big) &\quad\text{if $j\ge 2$}
\end{cases} \\
\ell_j &=
\begin{cases}
\ell(d, 2K, \eps_1)                               &\quad\text{if $j=1$} \\
\max \big( 2\ell_{j-1}, \ell(d, 2K, \eps_j) \big) &\quad\text{if $2\le j \le N$} \\
\max \big( 2\ell_N, m(d, 2K, \eps_{N+1}) \big)    &\quad\text{if $j = N+1$}
\end{cases}
\end{align*}
where the functions $\eta$, $\ell$, and $m$ come respectively from 
Lemma~\ref{l.eta}, Proposition~\ref{p.two exp}, and Proposition~\ref{p.BGV}.
Let $L = \ell_{N+1}$.

Now consider any $d$-dimensional cocycle $(X,T,E,A)$ bounded by $K$ and of period at least $L$
that admits no $L$-dominated splitting.
Take $\sigma = (\sigma_0, \ldots, \sigma_d) \in \cS_d$ with $\sigma \gtr \bsigma(A)$
and $\sigma_d = \bsigma_d(A)$.

By Lemma~\ref{l.zigzag} there is 
a sequence $\bsigma(A) = \sigma^{(0)} \les \sigma^{(1)} \les \cdots \les \sigma^{(N)} \les \sigma$
such that $\sigma^{(N)}$ is $\delta$-close to $\sigma$,
and there are $i_0$, \ldots, $i_{N-1} \in \{1, \ldots, d-1\}$ such that 
$$
1 \le j \le N, \  0 \le i \le d, \   i \neq i_j \quad \Rightarrow \quad
\sigma^{(j)}_i = \sigma^{(j+1)}_i \, .
$$

Since the cardinality of $X$ is at least $L = \ell_{N+1} \ge m(d, 2K, \eps_{N+1})$,
we can use Proposition~\ref{p.BGV} to obtain a path of cocycles
$A_t$, $t \in [0,1]$ starting at $A_0 = A$, all $\eps_{N+1}$-close to $A$ and with the same Lyapunov spectrum,
so that $A_1$ has only real eigenvalues.
Since $L \ge 2\ell_N$ and $\eps_{N+1} \le \eta(d,2K,\ell_N)$,
the cocycle $A_1$ has no $\ell_N$-dominated splitting.
Moreover, $\ell_N \ge \ell(d,2K, \eps_N)$.
Therefore using Proposition~\ref{p.two exp} we can find a path of cocycles
$A_t$, $t \in [1,2]$ such that:
\begin{itemize}
\item $\|A_t - A_1\|< \eps_N$ for all $t\in [1,2]$;
\item $\bsigma_i (A_t) = \bsigma_i(A_1)$ if $t\in [1,2]$, $i \neq i_1$;
\item $\bsigma_{i_0}(A_t)$ is non-decreasing for $t$ on $[1,2]$; 
\item $\bsigma_{i_0}(A_2) = \sigma^{(1)}_{i_0}$ (we stop the path given by the proposition when we reach this point).
\end{itemize}
Since $\ell_N \ge 2\ell_{N-1}$ and $\eps_N \le \eta(d,2K,\ell_{N-1})$,
the cocycle $A_2$ has no $\ell_{N-1}$-dominated splitting.

Continuing in this way, we obtain a path of cocycles $A_t$, $t \in [1,N+1]$
such that for all $j \in \{1, \ldots, N\}$:
\begin{itemize}
\item $\|A_t - A_j \|< \eps_{N-j+1}$ for all $t\in [j,j+1]$;
\item $\bsigma_i (A_t) = \bsigma_i(A_j)$ if $t\in [j,j+1]$, $i \neq i_{j-1}$;
\item $\bsigma_{i_{j-1}}(A_t)$ is non-decreasing for $t$ on $[j,j+1]$;
\item $\bsigma_{i_{j-1}}(A_{j+1}) = \sigma^{(j)}_{i_{j-1}}$.
\end{itemize}
In particular, the path of spectra $\bsigma(A_t)$ is monotonic with respect to the partial order $\les$,
the function $\bsigma_d(A_t)$ is constant,
and $\bsigma(A_{N+1}) = \sigma^{(N)}$.

Since $\sigma^{(N)} \les \sigma$ is $\delta$-close to $\sigma$,
using Lemma~\ref{l.triangular} we can find an $\eps_0$-short path of cocycles $A_t$, $t \in [N+1,N+2]$,
such that the path of spectra $\bsigma(A_t)$ is monotonic with respect to the partial order $\les$,
the function $\bsigma_d(A_t)$ is constant,
and $\bsigma(A_{N+2}) = \sigma$.

Concatenating everything, we obtain a path $A_t$, $t\in [0,N+2]$.
We have $\|A_t-A_0 \| \le \sum_0^{N+1} \eps_j \le \eps$.
Reparameterizing to the interval $[0,1]$,
we obtain a path of cocycles with all the desired properties.
\end{proof}

While in this paper we are mainly interested in cocycles over cyclic dynamical systems,
it is natural to consider more general cocycles.
For example, we ask: 
\begin{ques}
Let $T$ be a homeomorphism of a compact space $X$, admitting a fully supported ergodic measure $\mu$. 
Does Theorem~\ref{t.main} hold in this setting (taking Lyapunov graphs with respect to $\mu$)?
\end{ques}
One can ask for similar extensions of other theorems below.

\subsection{Restriction to Subbundles}\label{ss.subbundles}

Let us see some generalizations of Theorem~\ref{t.main}.
First we consider cocycles that have an invariant subbundle with no (strong) dominated splitting:

\begin{theo}\label{t.subbundle} 
For any $d \ge 2$, $K>1$, $\eps>0$, 
there exists $L\in \PT$
such that the following holds:
Let $(X,T,E,A)$ be a cyclic $d$-dimensional cocycle  bounded by $K$ and of period at least $L$.
Let $F\subset E$ be a subbundle of positive dimension $k \le d$
such that the restricted cocycle $A \restr F$ has no $L$-dominated splitting.
Let $\sigma \in \cS_k$ be a graph such that $\sigma \gtr \bsigma (A \restr F)$
and $\sigma_k = \bsigma_k(A \restr F)$.

Then there exists
an $\eps$-short path of cocycles $A_t$, $t\in [0,1]$ starting at $A$,
all of them preserving the bundle $F$,
such that the path of graphs $\bsigma(A_t\restr F)$ is non-decreasing,
$\bsigma(A_1\restr F) = \sigma$,
and $A_t / F = A / F$ for each $t$. 
\end{theo}

\begin{proof}
This follows directly from Theorem~\ref{t.main} applied to restricted cocycles,
together with the first extension procedure from \S~\ref{ss.procedures}.
\end{proof}

\begin{theo}\label{t.finest}
For any $d \ge 2$, $K>1$, $\eps>0$, 
there exists $\ell \in \PT$
such that the following holds:
Let $(X,T,E,A)$ be a cyclic 
$d$-dimensional cocycle bounded by $K$ 
and of period at least $\ell$.
Let $i_1 < \cdots < i_{m-1}$ be the indices of its finest $\ell$-dominated splitting.
Let $\sigma \in \cS_d$ be a graph such that $\sigma \gtr \bsigma (A)$,
$\sigma_{i_j} = \bsigma_{i_j}(A)$ for each $j$, and $\sigma_{d} = \bsigma_{d}(A)$.

Then there exists
an $\eps$-short path of cocycles $A_t$, $t\in [0,1]$ starting at $A$ 
such that the path of graphs $\bsigma(A_t)$ is non-decreasing
and $\bsigma(A_1) = \sigma$.
\end{theo}

Notice that this theorem extends Proposition~\ref{p.two exp}.

Theorem~\ref{t.periodic2} from the introduction follows directly 
using Franks Lemma and basic properties of dominated splittings.

\begin{proof}[of Theorem~{\rm\ref{t.finest}}]  
Let $d$, $K$ and $\eps$ be given.
Reducing $\eps$ if necessary (recall Remark~\ref{r.epsilon}), we assume that any $\eps$-perturbation of a cocycle bounded by $K$ is bounded by $2K$. 
We recursively define numbers $\eps_1$, $\ell_1$, \ldots, $\eps_{d-1}$, $\ell_{d-1}$ 
(in this order) by the formulas:
\begin{align*}
\eps_j &=
\begin{cases}
\eps/2               &\quad\text{if $j=1$} \\ 
\min \big( \eps/2^{j}, \eta(d, 2K, \ell_{j-1} ) \big) &\quad\text{if $j\ge 2$}
\end{cases} \\
\ell_j &=
\begin{cases}
\ell(d, 2K, \eps_1)                            &\quad\text{if $j=1$} \\
\max \big( 2\ell_{j-1}, L(d, 2K, \eps_j) \big) &\quad\text{if $j \ge 2$.}
\end{cases}
\end{align*}
where the functions $\eta$ and $L$ come respectively from 
Lemma~\ref{l.eta} and Theorem~\ref{t.subbundle}.
Let $\ell = \ell_{d-1}$.

Now consider any $d$-dimensional cocycle $(X,T,E,A)$ bounded by $K$ and of period at least $L$.
Let $E = F_1 \dplus \cdots \dplus F_m$ be its finest $\ell$-dominated splitting,
and $i_1 < i_2 <\cdots<i_{m-1}$ be its indices.
Take a graph $\sigma \in \cS_d$ such that $\sigma \gtr \bsigma (A)$
and $\sigma_{i_j} = \bsigma_{i_j}(A)$ for each $j$.
We can assume that $m>1$, otherwise there is nothing to show.

Since $A \restr F_1$ has no $\ell$-dominated splitting, 
by Theorem~\ref{t.subbundle} 
there is an $\eps_{d-1}$-short path of cocycles $A_t$, $t\in [0,1]$ starting at $A$,
such that
$\bsigma_i(A_2) = \sigma_i$ for $i\le i_1$ and
$\bsigma_i(A_2) = \bsigma_i(A)$ for $i>i_1$.

Now let $E = F_1^{(1)} \dplus \cdots \dplus F_{m_1}^{(1)}$ 
be the finest $\ell_{d-2}$-dominated splitting of $A_1$.
By Lemma~\ref{l.eta}, 
its set of indices
$\{i_1^{(1)} < i_2^{(1)} < \cdots <i_{m_1-1}^{(1)}\}$ 
is contained in $\{i_1 < i_2 <\cdots<i_{m-1}\}$.
We consider two cases: if $i_1^{(1)} = i_1$ then we apply Theorem~\ref{t.subbundle}
to the subbundle $F_2^{(1)}$;
if $i_1^{(1)} > i_1$ then we apply the theorem to $F_1^{(1)}$.
In either case, we find an $\eps_{d-2}$-short path of cocycles $A_t$, $t\in [1,2]$
such that $\bsigma(A_t)$ is non-decreasing,
$\bsigma_i(A_2) = \sigma_i$ for $i\le p$ and
$\bsigma_i(A_2) = \bsigma_i(A)$ for $i> p$,
where $p=i_1^{(2)}$ in the first case and $p=i_1^{(1)}$ in the second.
Thus $p \ge i_2$ Lyapunov exponents are already adjusted.

Continuing in this way, we will adjust all exponents  in a number $k \le m$ of steps.
That is, by concatenating we
obtain a path of cocycles $A_t$, $t\in [0,k]$
such that $\bsigma(A_t)$ is non-decreasing
and $\bsigma_i(A_k)= \sigma_i$ for all $i$.
For each $t$ we have $\|A_t - A\| \le \eps_{d-1} + \cdots + \eps_{d-k}\le \eps$.
Hence a reparameterized path has all the desired properties.
\end{proof}

\subsection{Perturbing the Spectrum with Constraints} \label{ss.constraints}

As we have seen, the path of spectra in the proof of Theorem~\ref{t.main}
is, except for its final part, a zigzag like in Figure~\ref{f.accessib}.
We could ask whether it is possible to prescribe any path of graphs. 
Strictly speaking, the answer is no, because if $A$ has a pair of non-real eigenvalues then
there are two Lyapunov exponents that we cannot immediately separate.
However, we ask if the \emph{trace} of the path of graphs can be prescribed:

\begin{ques}\label{q.control}
In the assumptions of the  Theorem~\ref{t.main}, and
given any non-decreasing path of graphs $\sigma_s \in \cS_d$, $s \in [0,1]$
with $\sigma_0 = \bsigma(A)$,
is it possible to choose the path of cocycles $A_t$
so that $\big\{ \bsigma(A_t); \;  t \in [0,1] \big\} =  \big\{ \sigma_s; \;  s \in [0,1] \big\}$?
\end{ques}

At least some partial control of the path of graphs is possible:
for example, we can choose it with constant index.
This is the content of Theorem~\ref{t.index} below,
which would follow immediately from a positive answer to Question~\ref{q.control}.

\medskip

We say that a graph $\sigma \in \cS_d$ has \emph{index} $p$ 
if $\sigma_i > \sigma_p$ for all $i \in \{ 0, \ldots, d \} \setminus \{p\}$.
In terms of the Lyapunov exponents $\lambda_i = \sigma_i - \sigma_{i-1}$,
this means that $\lambda_p < 0 < \lambda_{p+1}$
(disregard the first inequality if $p=0$ and the second if $p=d$);
in particular, there are no zero Lyapunov exponents.

\begin{theo}\label{t.index} 
The number $L = L(d,K,\eps)$ in Theorem~\ref{t.main}
can be chosen with the following additional property:
if $\bsigma(A)$ and $\sigma$ have the same index then the path of cocycles $A_t$
can be chosen so that all $\bsigma(A_t)$ have the same index.
\end{theo}

Using this theorem we can repeat the arguments from \S~\ref{ss.subbundles}
and obtain index-preserving versions of Theorems~\ref{t.subbundle} and \ref{t.finest}.

The proof of Theorem~\ref{t.index}
is exactly the same as Theorem~\ref{t.main},
except that we need an index-preserving version of Lemma~\ref{l.zigzag}:

\begin{lemm}\label{l.zigzag index}
The number $N=N(d,c,\delta)$ in Lemma~\ref{l.zigzag}
can be chosen with the following additional property:
If the graphs $\sigma'$, $\sigma'' \in \cS_d$ have the same index $p$
then the graphs $\sigma^{(1)}$, \ldots, $\sigma^{(N)}$ can all be chosen
with index $p$.
\end{lemm}

\begin{proof}
The $d=2$ situation is trivial.
So assume $d>2$ and assume by induction that the lemma is already known for smaller dimensions.
Take two graphs $\sigma' \les \sigma''$ in $\cS_d$ with the same endpoints,
no zero exponents, and the same index $p$.
We don't need to consider indices $p=0$ and $p=d$ because 
these cases are covered by Lemma~\ref{l.zigzag}.

Given any graph $\sigma \in \R^{d+1}$, we split it into two graphs:
$\sigma^{(L)} = (\sigma_0,\ldots,\sigma_p)$ and
$\sigma^{(R)} = (\sigma_p,\ldots,\sigma_d)$.

Let $\sigma^{(0)}=\sigma'$.
Let $\bar{\sigma}^{(0)}$ be the maximal convex graph $\les \sigma''$
satisfying $\bar{\sigma}^{(0)}_p = \sigma^{(0)}_p$.
The two graphs $\sigma^{(0L)} \les \bar{\sigma}^{(0L)}$ have the same endpoints,
and so do  $\sigma^{(0R)} \les \bar{\sigma}^{(0R)}$.
Applying Lemma~\ref{l.zigzag} to each side, and gluing sides together,
we can find a path of graphs  
$\sigma^{(0)} \les \cdots \les \sigma^{(j)}$ such that 
at most one coordinate changes at a time,
and the graphs $\sigma^{(j)} \les \bar{\sigma}^{(0)}$
are very close. 
Notice the the indices do not change along the path.

Assume that there exists $\ell$ with $p<\ell<d$
such that $\sigma^{(j)}_\ell = \sigma''_\ell$.
Then the subgraphs $(\sigma_\ell^{(j)}, \ldots, \sigma_d^{(j)}) \le (\sigma_\ell'', \ldots, \sigma_d'')$
are very close.
We apply the induction hypotheses to the
the pair of graphs  
$(\sigma_0^{(j)}, \ldots, \sigma_\ell^{(j)}) \le (\sigma_0'', \ldots, \sigma_\ell'')$,
obtaining a zigzag path of graphs of dimension $\ell$ of index $p$. 
So we extend this to a path of graphs of dimension $d$, and we are done.
Analogously, if there exists $\ell$ with $0<\ell<p$
such that $\sigma^{(j)}_\ell = \sigma''_\ell$ then we split the graph at this point;
the left part is already ok, and using the induction hypotheses on the right part, we are done.

Therefore we can assume that $\sigma^{(j)}_\ell < \sigma''_\ell$ for
every $\ell$ with $0<\ell<p$ or $p<\ell<d$.
This implies that both $\sigma^{(jL)}$ and $\sigma^{(jR)}$
are nearly flat; more precisely,
$\Delta^2 \sigma^{(j)}_\ell$ is small for every $\ell \neq p-1$.
(Recall notation from \S~\ref{ss.difference operator}.)

Now consider the \emph{gap} $G_j = \sigma_p'' - \sigma^{(j)}_p$.
Let $E_j = \min (-\Delta \sigma^{(j)}_{p-1}, \Delta \sigma^{(j)}_p)$.
It follows from nearly flatness that $\max(\sigma_0^{(j)},\sigma_d^{(j)}) - \sigma^{(j)}_p$, 
and hence $G_j$,
cannot be much bigger than $E_j$.
More precisely, we have $G_j \le 1.1 d E_j$
provided that $\sigma^{(jL)}$ and $\sigma^{(jR)}$ are sufficiently flat.

Define a graph $\sigma^{(j+1)}$ of index $p$ by taking
$$
\sigma^{(j+1)}_p = 
\min \left( \sigma''_{p}, \sigma^{(j)}_{p} + .9 E_j   \right) 
$$
and all the other coordinates equal to those of $\sigma^{(j)}$.
If $\sigma^{(j+1)}_p = \sigma''_{p}$ then the graphs
$\sigma^{(j+1)}$ and $\sigma''$ are very close and we are done.
Otherwise, the new gap is $G_{j+1} = G_j - .9 E_j$.
Therefore $G_{j+1} / G_j$ is less than some constant less than $1$.
So we restart the procedure;
after a finite number of steps the gap will be small and we are done.
\end{proof}

\begin{proof}[of Theorem~{\rm\ref{t.index}}]  
Repeat the proof of  Theorem~\ref{t.main},
using Lemma~\ref{l.zigzag index} in the place of Lemma~\ref{l.zigzag},
and taking for example an affine path in Lemma~\ref{l.triangular} 
(so that the index is preserved).
\end{proof}

\begin{rema}
More generally, let $\cP$ be a property about Lyapunov graphs of dimension~$d$ 
(as having some prescribed index, for example),
corresponding to a set $C \subset \cG_d$.
Suppose that 
given $\sigma'$, $\sigma'' \in C$, 
the points $\sigma^{(1)}$, \ldots, $\sigma^{(N)}$ in Lemma~\ref{l.zigzag} can all be chosen in $C$,
and moreover the path with those vertices (as in Figure~\ref{f.accessib})
is wholly contained in $C$.
Then a $\cP$-preserving version of Theorem~\ref{t.main} holds:
if $\bsigma(A)$ and $\sigma$ belong to $C$, then the path of cocycles $A_t$ can be chosen so that
$\bsigma(A)$ is contained in $C$ for all $t$.
\end{rema}

\section{Separating Lyapunov Exponents}\label{s.measure}

In this section we give the (stronger) cocycle versions 
of Theorems \ref{t.measure intro} and \ref{t.measure intro 2}.

\subsection{Statements and Sketch of the Proofs}

\begin{theo}\label{t.measure}
Let $(X,T,E,A)$ be a cocycle.
Let $\mu_k$ be a sequence of $T$-invariant probabilities converging to some $\mu$ in the weak-star topology.
Assume that each $\mu_k$ is supported on a periodic orbit, 
whose period tends to infinity with $k$. 
Then there exist a sequence of cocycles $B_k \to A$
such that such that $\bsigma(B_k,\mu_k) = \bsigma(A,\mu)$ for every $k$.
\end{theo}

\begin{rema}
Let us mention a non-perturbative result 
that also concerns the approximation of Lyapunov exponents using periodic orbits:
Theorem~1.4 from \cite{Kal}
asserts that 
if a cocycle is H\"older continuous and the base dynamics 
has certain hyperbolicity properties (satisfied by basic hyperbolic sets or subshifts of finite type, for example), 
then the Lyapunov exponents of every ergodic measure can be approximated by
the Lyapunov exponents at periodic points.
\end{rema}

\begin{proof}[of Theorem~{\rm\ref{t.measure intro}}]  
Given a sequence of diffeomorphisms $f_n: M \to M$ converging to a diffeomorphism $f = f_\infty$,
we define a cocycle $(X,T,E,A)$ as follows:
Let $\bar\NN = \NN \cup \{\infty\}$ be the one-point compactification of $\NN$,
and take
$$
X = \bar \NN \times M, \quad
T(n,x) = (n,f_n(x)), \quad
E = \bar\NN \times TM, \quad
A(n,x) = Df_n(x) \, .
$$
Applying Theorem~\ref{t.measure} to this cocycle and then using Franks Lemma,
it is easy to reach the conclusions of Theorem~\ref{t.measure intro}.
\end{proof}

\begin{theo}\label{t.measure finest}
Let $(X,T,E,A)$ be a cocycle.
Let $\mu_k$ be a sequence of $T$-invariant probabilities converging to some $\mu$ in the weak-star topology.
Assume that each $\mu_k$ is supported on a periodic orbit, 
whose period tends to infinity with $k$. 
Also assume that the support of $\mu_k$ converges in the Hausdorff topology 
to an (invariant compact) set $\Lambda$. 
Let $F_1 \dplus \cdots \dplus F_m$ be the finest dominated splitting for 
the cocycle $A$ restricted to $\Lambda$.
Let $i_j = \dim F_1 \oplus \cdots \oplus F_j$.
Let $\sigma \in \cS_d$ be any convex graph such that
$\sigma \gtr \bsigma(A,\mu)$ and $\sigma_{i_j} = \bsigma_{i_j}(A,\mu)$ for each $j$.
Then there exists $B$ arbitrarily close to $A$ such that $\bsigma(B,\mu_k) = \sigma$
for some~$k$.
\end{theo}

\begin{proof}[of Theorem~{\rm\ref{t.measure intro 2}}]  
One half of the theorem is trivial:
if a graph $\sigma$ is the limit of graphs $\bsigma(g_n, \gamma_n)$ as in the statement,
then by semicontinuity $\sigma$ must be above $\bsigma(f,\mu)$,
and by basic properties of dominated splittings, $\sigma$ must touch $\bsigma(f,\mu)$
at the points corresponding to the indices of the finest dominated splitting.
That is, we necessarily have $\sigma \in \cG(\mu, \Lambda)$.

The nontrivial half of the theorem follows from Theorem~\ref{t.measure finest},
arguing analogously as in the proof of Theorem~\ref{t.measure intro}.
\end{proof}

The proofs of the Theorems~\ref{t.measure} and \ref{t.measure finest} occupy the rest of this section.
Let us first give an informal exposition of the main ideas of the proof of Theorem~\ref{t.measure}.

It is sufficient to show that for every sufficiently large $k$, there is 
a perturbation $\tilde A$ with $\bsigma(\tilde A,\mu_k)$ approximately equal to $\bsigma(A,\mu)$,
because 
it is always easy to adjust the spectrum a little.

First consider the case $d=2$.
Take a large time scale $m$ so that $\int m^{-1} \log \|A^m\| \, d\mu$ is
close to $L_1 = L_1(A,\mu)$.
(Recall the notation from \S~\ref{ss.semicontinuity}.)
Then fix any sufficiently large $k$ so that 
the period $n=n_k$ of the orbit that supports $\mu_k$ is much bigger than $m$,
and the integral above (and thus $L_1$) is approximated
by $\int m^{-1} \log \|A^m\| \, d\mu_k$.
The latter is of course
the average of
$m^{-1} \log \|A^m(T^i y)\|$ where $i$ runs from $0$ to $n-1$,
and $y$ is a point in the periodic orbit.
The enemy here are ``cancelations'' that can make the value 
$n^{-1} \log \| A^n(y)\|$ (and thus the Lyapunov exponent with respect to $\mu_k$)
significantly smaller than $L_1$.
To fix this, for $y$ in the periodic orbit, let $Z_1(y)$ indicate 
the average of
the function $m^{-1} \log \|A^m\|$  over the  points $y$,
$T^m(y)$, $T^{2m}(y)$, \dots, $T^{(\lfloor n/m \rfloor -1) m}(y)$.
We say that $y$ is \emph{good} if $Z_1(y)$ is close to $L_1$.
It's easy to see that a good $y$ exists.
Then we multiply the cocycle matrices at the
points $y$, $T^m(y)$, \dots, $T^{(\lfloor n/m \rfloor -1) m}(y)$ 
by small rotations in order to remove cancellations.
More precisely, we choose a vector that we want to make expanding,
and we choose the rotations so that the iterates of this vector at
times $m$, $2m$, \dots, $(\lfloor n/m \rfloor -1) m$ 
do not fall into any contracting cones. At the period,
another small rotation makes the eigenvalue comparable to the norm of
the iterated vector. Now the top eigenvalue for the periodic orbit is approximately $e^{nL_1}$,
as desired.
Since rotations don't change the determinants, the other
eigenvalue is also ok.\footnote{Similar techniques of \emph{avoiding cancelations} are used 
in \cite{Mane ECL} and \cite{ABC}.}

Before considering the case $d>2$, we remark that 
\emph{most} points in the periodic orbit are good.
To see this, first notice that there is no $y$ in the periodic orbit 
for which $Z_1(y)$ 
is significantly \emph{bigger} than $L_1$, 
because in that case we would perturb the cocycle and realize
this exponent, thus contradicting semicontinuity. Since the
average of the $Z_1(y)$ over $y$ is close to $L_1$, 
we get that $Z_1(y)$ is close to $L_1$ for most~$y$.

In the higher dimensional case, we have to look norms of exterior
powers to see the other exponents. To avoid cancellations, there are
$d-1$ angles that we need to keep away from zero. This is
not hard, but requires some technical lemmas. 
The key part of the argument is to find
a point $y$ that is $i$-good for all intermediate dimensions $i$.
Here a point $y$ is \emph{$i$-good} if 
the average $Z_i(y)$ of the functions
$m^{-1}\log\|\wedge^i A^m\|$ over the points $y$,
$T^m(y)$, \dots, $T^{(\lfloor n/m \rfloor -1) m}(y)$ is close to $L_i(A,\mu)$.
The remark above also applies: most points are $i$-good.
In particular, there is at least one point that is $i$-good for every $i$,
and we are able to conclude the proof of Theorem~\ref{t.measure}.

\subsection{Some Geometric Lemmas}

In this subsection we establish some lemmas
that will be used in the proof of Theorem~\ref{t.measure}.

\begin{lemm}\label{l.2 jacs}
Given $\alpha>0$  and $d \ge 2$  there exists $C_1 = C_1 (\alpha,d) > 1$ with the following properties.
If $M: E \to E'$ is a linear map between euclidean spaces of dimension~$d$ and $E = F \oplus G$ is a splitting
with $\angle (F, G) > \alpha$ then
$$
\jac M \le C_1 \left( \jac M \restr F \right) \left( \jac M \restr G \right).
$$
\end{lemm}

\begin{proof}
Let $M:E \to E'$, $F$, $G$ and $\alpha$ be as in the statement.
Let $\pi: E \to F^\perp$ and $\pi' : E' \to (MF)^\perp$  be the orthogonal projections along $F$ and $MF$ respectively.
Take sets $S_1 \subset F$ and $S_2 \subset G$ of positive volume (in the respective dimensions).
Then
$$
\jac M =
\frac{\vol M(S_1 \oplus S_2)}{\vol S_1 \oplus S_2} =
\frac{\vol M(S_1) \cdot \vol \pi'(M (S_2))}{\vol S_1 \cdot \vol \pi(S_2)} 
$$
Since $\m(\pi) \ge \sin \alpha$ and $\|\pi'\| \le 1$, the lemma holds
with $C_1 = (\sin \alpha)^{-d}$.
\end{proof}

A flag on a vector space $E$ of dimension $d$ 
is a nested sequence of vector subspaces 
$F_1\subset \cdots \subset F_{d-1}$ such that $\dim F_i = i$ for each
$i$.

\begin{lemm}\label{l.rotate} 
Given $\eps>0$ and $d \ge 2$ there exists $\alpha=\alpha(\eps,d)>0$ with the following properties.
For any pair of flags $F_1 \subset \cdots \subset F_{d-1}$ and 
$G_1 \subset \cdots \subset G_{d-1}$ of a euclidean space $E$ of dimension $d$,
there exists an orthogonal map $R : E \to E$ with $\|R - \Id \| < \eps$ 
such that $\angle(RF_i, G_{d-i}) > \alpha$.
\end{lemm}

\begin{proof}
Apply Claim~6.4 from \cite{ABC} to the spaces $F_i$ and $G_{d-i}^\perp$.
\end{proof}

\begin{lemm}\label{l.convert}
Given $\eps>0$ and $d \ge 2$ there exists $C_2=C_2(\eps,d)>0$ with the following properties.
Let $M: E \to E$ be a linear map on a euclidean space of dimension $d$.
Then there exists  an orthogonal map $R : E \to E$ with $\|R - \Id \| < \eps$ 
such that 
$$
\rad(\wed^i R M) > C_2^{-1} \|\wed^i M\| \quad \text{for each $i$.}
$$
\end{lemm}

%

The proof of Lemma~\ref{l.convert} will require a few other lemmas
(which will not be used directly in the proof of Theorem~\ref{t.measure}):

\begin{lemm}\label{l.cone}
For every $d$ and $\beta>0$ there exists $C_3 = C_3(d,\beta)>1$ with the following properties.
Let $M:E \to E$ be a linear on a space of dimension $d$ and 
let $v$ be a unit vector with $\|M v\| = \|M\|$.
Assume that $\angle (Mv, v) < \frac{\pi}{2} - \beta$.
Then 
$$
\max \big( \rad(M), \sing_2(M) \big) \ge C_3^{-1} \| M \| \, .
$$
\end{lemm}


\begin{proof}
Take the linear map $M: E \to E$ and the unit vector $v$ with $\|Mv\| = \|v\|$.
Let $\theta = \angle (Mv, v)$, and assume that $\theta < \frac{\pi}{2} - \beta$.
Write the matrix of $M$ with respect to the splitting $(\R v) \oplus v^\perp$:
$$
M = \begin{pmatrix}  \pm \|M\| \cos \theta & * \\ * & N \end{pmatrix}.
$$
Since $\sing_2(M) = \|M \restr v^\perp \|$, the norm of $N$ is less than $\sing_2(M)$.
Hence there exists $\eps_0>0$ depending only on $d$ and $\beta$
such that if $\sing_2(M) \le \eps_0 \|M\|$ then 
$$
| \trace M | \ge \frac{\sin \beta}{2} \|M \| 
\quad \text{and in particular} \quad 
\rad(M) \ge \frac{\sin \beta}{2d} \|M \| \, .
$$ 
So taking $C_3^{-1} = \min \big(\eps_0, (2d)^{-1}\sin \beta \big)$ the conclusions of the lemma are satisfied.
\end{proof}

\begin{lemm}\label{l.exterior}
For every $d$ and $\alpha>0$ there exists $\beta_1 = \beta_1(d,\alpha) > 0$ with the following properties.
Let $E$ be a space of dimension $d$, let $i \in \{1,\ldots, d-1\}$,
and let $\{ v_1, \ldots, v_i \}$ and $\{w_1, \ldots, w_i\}$ be linearly independent subsets of $E$ 
spanning subspaces
$F$ and $G$, respectively.
Consider the 
$i$-vectors $v = v_1 \wedge \cdots \wedge v_i$ and $w = w_1 \wedge \cdots \wedge w_i$ 
Then $\angle(F^\perp, G)>\alpha$ implies  $\angle(v,w) < \frac{\pi}{2}-\beta_1$.
\end{lemm}

\begin{proof}
The non-zero decomposable $i$-vectors $v$ and $w$ uniquely determine
the spaces $F$ and $G$.
Thus the quantity $\angle(F^\perp, G)$ is a function $f(v,w)$.
We can assume that $v$ and $w$ have unit norm, so the domain of $f$ becomes compact.
Notice that $f(v,w)=0$ if and only if $v \perp w$.
(This follows from the description of the inner product on $\wed^i E$ explained before.)
By continuity, if $\angle(v,w)$ is sufficiently close to $\pi/2$ then  $f(v,w)$ is small.
This gives the desired result.
\end{proof}

\begin{lemm}\label{l.eigenvalues} 
For every $\alpha>0$ and $d \ge 2$ there exists $C_4 = C_4 (\alpha,d) > 1$ with the following properties.
Let $M: E \to E$ be a linear map on a euclidean space of dimension $d$.
Let $F_1 \subset \cdots \subset F_{d-1}$ be a flag
such that for each $i$, the quantity $\jac M \restr F_i$ is as big as possible (that is, $\| \wed^i M\|$). 
Assume that  
$$
\angle (M F_i , F_i^\perp) > \alpha \quad \text{for each $i$.}
$$
Then
$$
\rad(\wed^i M) > C_4^{-1} \|\wed^i M\| \quad \text{for each $i$.}
$$
\end{lemm}

\begin{proof}
Given $d$, we define
$$
\bar \beta = \min_{0<i<d} \beta_1 \left( {\textstyle \binom{d}{i}} , \alpha \right), \quad
\bar C     = \max_{0<i<d} C_3 \left( {\textstyle \binom{d}{i}} , \bar \beta \right),
$$
where the functions $\beta$ and $C_3$ are given by Lemmas~\ref{l.exterior} and \ref{l.cone}, respectively.

Take $M:E\to E$ with the flag $F_1 \subset \cdots \subset F_{d-1}$ as in the statement,
so $\angle(MF_i, F_i^\perp) > \alpha$.
Let $e_1$, \ldots, $e_d$ be a orthonormal basis of $E$ such that
$\{e_1,\ldots, e_i\}$ spans $F_i$.
Let $v_i$ be the $i$-vector $e_1 \wedge \cdots \wedge e_i$.
By Lemma~\ref{l.exterior}, we have 
$$
\angle\big(v_i, (\wed^i M) v_i\big)< \frac{\pi}{2}-\beta.
$$
By Lemma~\ref{l.cone}, this implies that
$$
\frac{\max \big( \rad(\wed^i M), \sing_2(\wed^i M) \big)}{\| \wed^i M \|} \ge \bar{C}^{-1}  \, .
$$


Now notice that
$$ 
\frac{\sing_2(\wed^ i M)}{\| \wed^i M \|} = 
\frac{\sing_2(\wed^ i M)}{\sing_1(\wed^i M)} =
\frac{\sing_1(M) \cdots \sing_{i-1}(M)\sing_{i+1}(M)}{\sing_1(M) \cdots \sing_{i-1}(M)\sing_{i}(M)} =
\frac{\sing_{i+1}(M)}{\sing_{i}(M)} \, .
$$
Thus
\begin{equation}\label{e.2 types}
\text{for each $i$,} \quad
\rad(\wed^i M)  \ge \bar{C}^{-1} \| \wed^i M \| \quad\text{or}\quad
\sing_{i+1}(M) \ge \bar{C}^{-1} \sing_{i}(M) \, .
\end{equation}
Let's say that $i$ is of the first type if the first alternative holds,
and of second type instead.
The $i$'s of first type are already controlled.
We need a convexity argument to control the $i$'s of the second type.

Define numbers $x_i = \log \rad(\wed^i M)$ and $y_i = \log \| \wed^i M \|$ for $1 \le i \le d$,
and also $x_0 = y_0$.
Then the graphs of the functions $x_i$ and $y_i$ over $\{0,1,\ldots,d \}$,
have the same endpoints (in particular $0$ and $d$ are of first type), 
and $x_i \le y_i$.
Using the difference operator notation from \S~\ref{ss.difference operator},
concavity means that $\Delta^2 x_i$ and $\Delta^2 y_i$ are non-positive.
Notice that $\Delta[\log \sing_i(M)] = y_{i-1}$.
Letting $\gamma = \log \bar C$, \eqref{e.2 types} translates to
$$
\text{for each $i$,} \quad
x_i \ge y_i - \gamma \quad\text{or}\quad \Delta^2 y_{i-1} \ge -\gamma.
$$

Let $0 \le i_0 < i_1 \le d$ be any two consecutive indexes of first type.
Then every $i \in (i_0, i_1)$ is of first type and so 
$\Delta^2 y_{i-1} \ge -\gamma$.
It follows from Lemma~\ref{l.nonlinearity} applied to $-y_i$ that
\begin{equation}\label{e.nonlinearity}
0 \le y_i - \left(\frac{i_1-i}{i_1-i_0} y_{i_0} + \frac{i-i_0}{i_1-i_0} y_{i_1} \right) \le \frac{(i_1-i_0)^2}{4} \gamma. 
\end{equation}
So for all $i \in (i_0,i_1)$, we have
\begin{alignat*}{2}
x_i &\ge \frac{i_1-i}{i_1-i_0} x_{i_0} + \frac{i-i_0}{i_1-i_0} x_{i_1} 
&\quad&\text{(by concavity)}\\ 
    &\ge \frac{i_1-i}{i_1-i_0} (y_{i_0}-\gamma) + \frac{i-i_0}{i_1-i_0} (y_{i_1}-\gamma) 
&\quad&\text{(since $i_0$, $i_1$ are of 1st type)} \\
    &\ge y_i - \left(1 + \frac{(i_1-i_0)^2}{4}\right)\gamma 
&\quad&\text{(by \eqref{e.nonlinearity}).}
\end{alignat*}
Therefore the lemma holds with $C_4 = \bar{C}^{1+d^2/4}$.
\end{proof}

\begin{proof}[of Lemma~{\rm\ref{l.convert}}]  
Given $\eps$ and $d$, let $\alpha=\alpha(\eps, d)$ be given by Lemma~\ref{l.rotate},
and let $C_2 = C_4(\alpha,d)$ be given by Lemma~\ref{l.eigenvalues}.
Now, given a linear map $M$ of a space of dimension $d$, let   
$F_1 \subset \cdots \subset F_{d-1}$ be a flag
such that $\jac M \restr F_i$ is as big as possible.
By Lemma~\ref{l.rotate}, there is an orthogonal map $R$ with $\|R-\Id\|<\eps$ such that
$\angle (R M F_i , F_i^\perp) > \alpha$ for each $i$.
Then, by Lemma~\ref{l.eigenvalues},
\[
\rad(\wed^i R M) >  C_2^{-1} \|\wed^i R M\| =  C_2^{-1} \|\wed^i M\| . 
\]
\end{proof}

\subsection{Proofs of Theorems~\ref{t.measure} and \ref{t.measure finest}}\label{ss.measure proofs}

Let us begin with a simple extension result:

\begin{lemm}\label{l.franks} 
Given $d \ge 2$, $K>1$, $\eps>0$, there exists
$\eps' = \eps'(d,K,\eps)>0$ with the following property.
Let $(X,T,E,A)$ be a cocycle bounded by $K$.
Assume $X_0 \subset X$ is a finite subset,
and for each $x\in X_0$ it is given a linear map $B_0(x): E(x) \to E(Tx)$
with $\|B_0(x)-A(x)\| < \eps'$.
Then there exists a (continuous) linear cocycle $B$ that is $\eps$-close to $A$ such that
$B(x) = B_0(x)$ for each $x\in X_0$.
\end{lemm}

\begin{proof}
Use Tietze extension theorem.
\end{proof}

\begin{proof}[of Theorem~{\rm\ref{t.measure}}]  
Let $(X,T,E,A)$ be a cocycle.
Assume that $\mu_k$ is a sequence of
invariant probability measures converging so some $\mu$,
with each $\mu_k$  supported on a periodic orbit of period $n_k$.
We also assume that $n_k \to \infty$ as $k \to \infty$.

First, let us notice that is sufficient to prove that
\emph{there exists a sequence of cocycles $\tilde A_k$ converging to $A$ 
such that $\bsigma(A_k,\mu_k)$ converges to $\bsigma(A,\mu)$.}
Indeed if $n_k$ is large enough then by Proposition~\ref{p.BGV} (and Lemma~\ref{l.franks})
we can perturb $\tilde A_k$ without changing $\bsigma(\tilde A_k,\mu_k)$ so that its restriction to $\supp \mu_k$
has only real eigenvalues;
then using Lemma~\ref{l.triangular} we find another perturbation $B_k$ so that 
$\bsigma(B_k,\mu_k) = \bsigma(A,\mu)$.

Now, the assertion above is equivalent to the following:
\emph{for every $\delta>0$ and every sufficiently large $k$ (depending on $\delta$),
there exist a $\delta$-perturbation 
$\tilde A$ of $A$ such that the graphs $\bsigma(\tilde A,\mu_k)$ and $\bsigma(A,\mu)$ are $\delta$-close.}
So let us prove this assertion instead.

\medskip

Fix $\delta>0$.
By semicontinuity, there exists a positive $\eps < \delta$ such that
\begin{equation} \label{e.semicont}
L_i(B, \mu_k) < L_i(A,\mu) + \frac{\delta}{4d} \quad 
\text{for all $i$, all $k>\eps^{-1}$, and all $B$ $\eps$-close to $A$.}
\end{equation}
Let $K>1$ be a bound for all cocycles that are $\eps$-close to $A$.
Let $\eps' = \eps'(d,K,\eps)$ be given by Lemma~\ref{l.franks}, and $\eps'' = K^{-1} \eps'$
Let $\alpha = \alpha(\eps'',d)$ 
be given by Lemma~\ref{l.rotate}.
Let $C_1= C_1(\alpha,d)$ and $C_2 = C_2(\eps'',d)$ be given by
Lemmas~\ref{l.2 jacs} and \ref{l.convert}, respectively. 

Now let $\eta>0$ be small; the precise requirements will appear later.
Fix an integer $m$ such that 
$$
m > \eta^{-1}
\quad \text{and} \quad
L_i(A,\mu) \le \int \frac{1}{m} \log \|\wed^i A^m\| \, d\mu < L_i(A,\mu) + \eta \ \forall i.
$$
If $k$ is sufficiently large then
$$
\frac{m}{n_k}< \eta  \quad \text{and} \quad
\left| \int \frac{1}{m} \log \|\wed^i A^m\| \, d\mu_k  - \int \frac{1}{m} \log \|\wed^i A^m\| \, d\mu  \right|< \eta.
$$
Fix any $k>\eps^{-1}$ with the properties above.
Write $n=n_k$ and let $q = \lfloor n/m\rfloor$.
For each $y \in \supp \mu_k$, define
$$
Z_i(y) = \frac{1}{n} \sum_{p=0}^{q-1} \log \|\wed^i A^m (T^{pm} y) \| \, .
$$

We claim that
for each $y \in \supp \mu_k$ there is a $\eps$-perturbation 
$\tilde A_y$ of $A$ such that 
\begin{equation}\label{e.B}
L_i(\tilde A_y,\mu_k) > Z_i(y) - C_3\eta \quad \text{for each $i$,}
\end{equation}
where $C_3>0$ does not depend on $\eta$.

(At the end, the perturbation $\tilde A$ that we are looking for 
will be $\tilde A_y$ for an appropriate choice of $y$, but we cannot say a priori which $y$ works.)

\begin{proof}[of the claim]
Fix $y \in \supp \mu_k$.
For each $p=0,1,\ldots,q-1$, we will define two flags of $E(T^{pm} y)$,
$$
F_1^{(p)} \subset \cdots \subset F_{d-1}^{(p)} \quad \text{and} \quad
G_1^{(p)} \subset \cdots \subset G_{d-1}^{(p)} \, .
$$  
The first family of flags is chosen so that 
$\jac A^m(T^{pm} y) \restr F_i^{(p)}$ is as small as possible,
that is, $\m(\wed^i A^m(T^{pm} y))$.
The second family of flags is defined recursively.
Let $G_i^{(0)} = \big[ F_{d-i}^{(0)} \big]^\perp$.
Once the $(p-1)$-th flag is defined ($p>0$), 
we apply Lemma~\ref{l.rotate}
to choose an orthogonal map $R^{(p)}$ of $E(T^{pm} y)$ with 
$\|R^{(p)} - \Id \| < \eps''$ such that
$$
\text{defining} \quad 
G_i^{(p)} = R^{(p)} A^m(T^{(p-1)m} y) G_i^{(p-1)},
\quad\text{we have}\quad
\angle \big( G_i^{(p)} , F_{d-i}^{(p)} \big) > \alpha.
$$

By Lemma~\ref{l.2 jacs}, the lower bounds on the angles between the flags
imply that, for each $p=0,\ldots,q-1$,
$$
\jac A^m(T^{pm} y) \restr G_i^{(p)} \ge
C_1^{-1} \frac{\jac A^m(T^{pm} y)}{\jac A^m(T^{pm} y) \restr F_{d-i}^{(p)}} =
C_1^{-1} \| \wed^i A^m(T^{pm} y) \| \, .
$$
Consider the linear map $M:E(y) \to E(y)$ given by
$$
M = A^{n-mq}(T^{qm} y) R^{(q)} A^m(T^{(q-1)m} y) \cdots R^{(2)} A^m(T^m y) R^{(1)} A^m(y) \, .
$$
Then
\begin{align*}
\| \wed^i M\| 
\ge \jac M \restr G_i^{(0)} 
&\ge K^{-i(n-mq)} C_1^{-q} \prod_{p=0}^{q-1} \|\wed^i A^m(T^{pm} y)\| \\
&\ge K^{-dm} C_1^{-2n/m} e^{n Z_i(y)} \, .
\end{align*}
By Lemma~\ref{l.convert}, there exists an orthogonal map $R^{(0)}$ of $E(y)$ such that 
$\| R^{(0)} - \Id \| < \eps''$ and 
$\rad(\wed^i R^{(0)} M )  \ge C_2^{-1} \|\wed^i M\|$ for each $i$.
Let $\tilde A_y$ be a continuous cocycle $\eps$-close to $A$
that equals $A$ along the orbit of $y$ except at the points specified below:
$$
\tilde A_y(y) = A(y) R^{(0)}, \quad
\tilde A_y(T^{pm-1} y) = R^{(p)} A^m(T^{pm-1} y) \text{ for } p \in \{1,\ldots,q\}.
$$
Then
$$
L_i(\tilde A_y, \mu_k) = \frac{1}{n}\log \rad \big(\wed^i \tilde A_y^n(y)\big) 
= \frac{1}{n} \log \rad\big( \wed^i M R^{(0)} \big)
\ge Z_i(y) - C_3\eta,
$$
where $C_3 = \log(K^d C_1^2 C_2)$.
That is, $\tilde A_y$ has property~\eqref{e.B}.
\end{proof}

It follows from \eqref{e.B} and \eqref{e.semicont} that
\begin{equation} \label{e.Z upper}
Z_i(y) < L_i(A,\mu) + C_3 \eta + (4d)^{-1} \delta \quad \text{for all $y\in \supp \mu_k$.}
\end{equation}
Fix any $y_0 \in \supp \mu_k$ and denote $y_\ell = T^\ell y_0$ for $0 \le \ell < m$.
Then
\begin{align}
\frac{1}{m}\sum_{\ell=0}^{m-1} Z_i(y_\ell) 
&=   \frac{1}{mn} \sum_{j=0}^{qm-1} \log\|\wed^i A^m (T^j y_0) \| \notag \\
&\ge \int \frac{1}{m}  \log\|\wed^i A^m \|\, d\mu_k - \frac{\log K^{im}}{n} \notag \\
&\ge L_i(A,\mu) - C_4 \eta, \label{e.Z lower}
\end{align}
where $C_4 = 2+d\log K$.
Using \eqref{e.Z upper} and \eqref{e.Z lower} we will show that
that for each $i$,
the number $Z_i(y_\ell)$ is close
to $L_i(A,\mu)$ for ``most'' $\ell \in \{0,\ldots,m-1\}$, and in particular we can find some $\ell$
that works for every $i$.

More precisely, let
$$
\rho_i = \frac{1}{m} \# \big\{ \ell \in \{0,\ldots,m-1\} ; \; Z_i(y_\ell) < L_i(A,\mu)- \delta/2 \big\} \, .
$$
Then it follows from~\eqref{e.Z lower} and \eqref{e.Z upper} that 
\begin{align*}
L_i(A,\mu) 
&\le C_4\eta + \frac{1}{m}\sum_{\ell=0}^{m-1} Z_i(y_\ell) \\ 
&\le C_4\eta +  \rho_i \left(L_i(A,\mu) - \frac{\delta}{2}\right) 
+ (1-\rho_i)  \left(L_i(A,\mu) + C_3\eta + \frac{\delta}{4d}  \right) \\
&\le L_i(A,\mu) + \left( \frac{1}{4d} -  \frac{\rho_i}{2} \right)\delta + (C_3 + C_4)\eta \\
&<   L_i(A,\mu) + \left( \frac{1}{2d} -  \frac{\rho_i}{2} \right)\delta,
\quad\text{provided we choose $\eta$ small enough.}
\end{align*}
So $\rho_i < 1/d$
and in particular $\sum_{i=1}^d \rho_i < 1$.
Thus there is some $\ell \in \{0,\ldots,m-1\}$ such that
$Z_i(y_\ell) \ge L_i(A,\mu) + \delta/2$ for all $i\in \{1,\ldots,d\}$.
Now let $\tilde A = \tilde A_{y_\ell}$;
this is a $\delta$-perturbation of $A$ such that for each $i$ we have
$$
L_i (\tilde A, \mu_k) > Z_i(y_\ell) - C_3 \eta >  L_i(A,\mu) - \delta 
$$
(again because $\eta$ is small).
By \eqref{e.semicont}, $|L_i (\tilde A, \mu_k) - L_i(A,\mu)| < \delta$,
as we wanted.
\end{proof}

\begin{proof}[of Theorem~{\rm\ref{t.measure finest}}]  
Let $(X,T,E,A)$, $\mu$, $\mu_k$, $\Lambda$, $F_i$, $i_j$, and $\sigma$
be as in the statement of the theorem.
Let $n_k$ be the period of the orbit that supports $\mu_k$.
By assumption, $n_k \to \infty$.

Given a small $\eps>0$,
let $K>1$ be a bound for all $\eps$-perturbations of $A$.
Let $\eps' = \eps'(d,K,\eps/2)$ be given by Lemma~\ref{l.franks}.
Let $L = L(d,K,\eps')$ be given by Theorem~\ref{t.finest}.
Let $\eta = \eta(d,K,L)$ be given by Lemma~\ref{l.eta}. 
We can assume $\eta < \eps/2$.

There is $k_0$ such that if $k > k_0$ then $n_k > 2L$ and
the restriction of $A$ to $\supp \mu_k$
has no $2L$-dominated splittings of indices $i_j$.
Using Theorem~\ref{t.measure},
find some $k>k_0$ such that 
such that $\bsigma (\tilde A, \mu_k) = \bsigma (A , \mu)$
for some $\eta$-perturbation $\tilde A$ of $A$.
The restriction of $\tilde A$ to $\supp \mu_k$
has no $L$-dominated splitting of index $i_j$.
By Theorem~\ref{t.finest}, there is an $\eps'$-perturbation
of $\tilde A$ along the orbit that supports $\mu_k$
such that the Lyapunov graph becomes exactly $\sigma$.
Using Lemma~\ref{l.franks} we extend this to a global $\eps/2$-perturbation of $\tilde A$,
which is the desired $\eps$-perturbation of $A$.
\end{proof}

\section{Immediate Applications}\label{s.corol}

Here we show Corollaries~\ref{c.index} and~\ref{c.generic}.

\subsection{Changing the Index of a Periodic Point} \label{ss.index}

\begin{proof}[of Corollary~{\rm\ref{c.index}}]  
Consider a sequence $\gamma_n=\orb(p_n)$  of periodic orbits whose period tends to infinity
and let $\mu_n$ be the corresponding measures.
Assume that $(\mu_n,\gamma_n)$ converges (in the weak-star times Hausdorff topology) to a pair $(\mu,\La)$. 
Let $i_j = \dim(E_1\oplus\cdots\oplus E_j)$ for $j\in\{0,\dots,m\}$, where $E_1\dplus E_2 \dplus \cdots \dplus E_m$ 
is the finest dominated splitting over $\La$. 
Domination implies strict convexity of the graph $\bsigma(\mu)$
at the points $i_j$;
more precisely, for each $j \in \{1 ,\ldots, m-1\}$ we have
$$
\lambda_{i_j}(\mu) < \lambda_{i_j + 1}(\mu), \quad \text{that is,} \quad
\bsigma_{i_j}(\mu) < \frac{\bsigma_{i_j-1}(\mu) + \bsigma_{i_j+1}(\mu)}{2} \, .
$$
Let 
$$
s =  \min_{j \in \{0,\dots,m\}}  \bsigma_{i_j}(\mu) \quad \text{and} \quad
K = \big\{ k \in \{0,\ldots,d\} ; \; \bsigma_k(\mu) \le s \big\}.
$$
By the convexity properties, $K$ is an interval, and there is an unique $j \in \{1,\dots,m\}$ 
such that $K \subset \{i_{j-1}, \dots, i_j\}$.

Now fix $k \in  K$ and define a convex graph $\sigma = (\sigma_0, \ldots, \sigma_d)$ as follows:
$$
\sigma_i = 
\begin{cases}
\bsigma_i(\mu)      
&\quad\text{if $i\le i_{j-1}$ or $i \ge i_j$,}\\
\dfrac{k-i}{k-i_{j-1}} \cdot \bsigma_{i_{j-1}}(\mu) + \dfrac{i-i_{j-1}}{k-i_{j-1}} \cdot \bsigma_{k}(\mu)
&\quad\text{if $i_{j-1} < i \le k$,} \\ 
\dfrac{i_j-i}{i_j-k} \cdot \bsigma_{k}(\mu) + \dfrac{i-k}{i_j-k} \cdot \bsigma_{i_j}(\mu)
&\quad\text{if $k \le  i < i_j$.} 
\end{cases}
$$
Then $\sigma$ is above $\bsigma(\mu)$
and is compatible with the dominated splitting on $\La$.
That is, $\sigma$ belongs to the set $\cG(\mu,\La)$ (defined in \S~\ref{sss.answer}). 
By Theorem~\ref{t.measure intro 2},
if $n$ is sufficiently large 
then there is a perturbation $g$ of $f$ 
preserving the orbit $\gamma_n$ and such that the Lyapunov graph $\bsigma(\gamma_n,g)$ is  equal to~$\sigma$. 
If $\bsigma_k(\mu,f)$ is strictly smaller than $s$,
then the index of $\gamma_n$ for $g$ is precisely $k$, concluding the proof in that case. 
In the remaining case,
some of the Lyapunov exponents of $\gamma_n$ for $g$ are zero.
Then an arbitrarily small perturbation of $g$ along the orbit of $\gamma_n$ 
(using Lemma~\ref{l.triangular} and Proposition~\ref{p.BGV})
allows us to change these vanishing exponents in order to get any prescribed signs, 
thus concluding the proof.  
\end{proof}

\subsection{Lyapunov Spectra of Periodic Orbits for Generic Diffeomorphisms} \label{ss.generic}

Let us rephrase Corollary~\ref{c.generic}.
If $f$ is a diffeomorphism of the compact $d$-dimensional manifold $M$, 
let $\cZ(f)$ indicate the closure in $\cP(M)\times\cK(M)\times \RR^{d+1}$ 
of the set of triples $(\mu_\gamma,\gamma,\bsigma(f,\gamma))$,
where $\gamma$ runs on all hyperbolic periodic orbits of $f$.
(Here we write $\bsigma(f,\gamma) = \bsigma(f,\mu_\gamma)$ for simplicity.) 
Then Corollary~\ref{c.generic} states that for generic $f$ we have
\begin{equation}\label{e.corollary}
\cZ(f) = \bigcup_{(\mu,\Lambda) \in \cX(f)} \{(\mu,\Lambda)\} \times \cG(\mu,\Lambda) \, .
\end{equation}
Now we prove it:

\begin{proof}
The ``$\subset$'' inclusion is the easy one, and works for every $f$:
By semicontinuity of the Lyapunov graph (see \S~\ref{ss.semicontinuity}),
$$
(\mu_{\gamma_n} , \gamma_n, \bsigma(f,\gamma_n) )  \to  (\mu, \Lambda, \sigma)
\quad \Rightarrow \quad \sigma \ge \bsigma(f,\mu).
$$
Moreover, by persistence and continuity of dominated splittings, 
for each $i\in\{1,\dots, d-1\}$ such that there is a dominated splitting 
$T_\Lambda M = E \dplus F$ with $\dim E= i$, we must have $\sigma_i = \bsigma_i(f,\mu)$.
That is, we have $\sigma \in \cG(\mu,\Lambda)$,
proving one inclusion in \eqref{e.corollary}.

\begin{clai}
	The map $f \mapsto \cZ(f)$ is lower semicontinuous. 
\end{clai}
\begin{proof}
The set $\cZ(f)$ may be approached from inside by a finite set of triples $(\mu_\gamma,\gamma,\bsigma(f,\gamma))$,
where $\gamma$'s are hyperbolic periodic orbits.
Each hyperbolic periodic orbit of $f$ persists and varies continuosly 
in a small neighborhood of $f$, and its Lyapunov graph varies continuously on this neighborhood.
Thus the finite set of triples varies continuously on that neighborhood, giving the lower semicontinuity.
\end{proof}

It follows from a well-known result from general topology that
the points of continuity of the map $f\mapsto\cZ(f)$ form 
a residual subset $\cR$ of $\Diff^1(M)$. 
Fix any $f \in \cR$; we now claim that \eqref{e.corollary} holds for $f$. 
Take $(\mu, \Lambda) \in \cX(f)$ and $\sigma \in \cG(\mu,\Lambda)$.
Let us show that $(\mu, \Lambda, \sigma) \in \cZ(f)$.
By definition of $\cX(f)$, there exists a sequence of hyperbolic periodic orbits $\gamma_n$
such that $(\mu_{\gamma_n}, \gamma_n) \to (\mu, \Lambda)$.
By Theorem~\ref{t.measure intro 2},
there is are diffeomorphisms $g_n$ preserving respectively the orbits $\gamma_n$
such that $g_n \to f$ and $\bsigma( g_n, \gamma_n) \to \sigma$.
In addition, we can assume that $\gamma_n$ is hyperbolic with respect to $g_n$.
Since $f\in\cR$,
the sequence of sets $\cZ(g_n)$ converges to $\cZ(f)$ in the Hausdorff topology.
Each element of the sequence  $(\mu_{\gamma_n}, \gamma_n, \bsigma(g_n, \gamma_n))$
belongs to the respective $\cZ(g_n)$ and therefore
the limit of the sequence, which is  $(\mu,\Lambda,\sigma)$, belongs to $\cZ(f)$.
Thus \eqref{e.corollary} is true for any $f\in \cR$,
concluding the proof of Corollary~\ref{c.generic}.
\end{proof}

\section{Consequences to Universal Dynamics}  \label{s.universal}

Here we will give the applications to universal dynamics explained in \S~\ref{sss.intro universal}.

\subsection{Proof of Theorem~{\rm\ref{t.criterion}}}

We will need a number of lemmas. 

\medskip

The lemma below is useful when we want to apply Theorem~\ref{t.measure intro 2} to homoclinic classes.
Recall that $\mu_\gamma$ indicates the unique invariant probability measure supported on a periodic orbit $\gamma$.

\begin{lemm}\label{l.class measure}
Let $H$ be the homoclinic class of a periodic orbit $\gamma$.
Then there is a sequence of periodic orbits $\gamma_n$ 
homoclinically related to $\gamma$ 
such that:
\begin{itemize}
\item the sets $\gamma_n$ converge to $H$ in the Hausdorff topology;
\item the measures $\mu_{\gamma_n}$ converge to the measure $\mu_\gamma$ in the weak-star topology.
\end{itemize}
\end{lemm}

In the notation of \S~\ref{sss.intro generic}, 
the lemma says that $(\mu_\gamma, H) \in \cX(f)$.

\begin{proof}
Using Markov partitions, we see that the lemma holds true
in the case that $H$ is a horseshoe (that is, a locally maximal hyperbolic set).
In the general case, we can take a sequence of horseshoes $H_n$
contained in the homoclinic class $H$ and containing $\gamma$
such that $H_n \to H$ in the Hausdorff topology.
Then the lemma follows immediately from the previous case.
\end{proof}

\begin{lemm}[(Creating zero exponents)]\label{l.vanishing}   
Let $f \in \Diff^1(M)$. 
Let  $\gamma_n$ be a sequence of periodic orbits whose periods tend to infinity. 
Suppose that the invariant probabilities $\mu_{\gamma_n}$
converge in the weak star topology to some $\mu$, 
and that the sets $\gamma_n$ converge in the Hausdorff topology to an $f$-invariant compact set $\La$. 
Let $E_1 \dplus \cdots \dplus E_m$ be the finest dominated splitting over $\La$.
Assume that 
$$
\bsigma_k(\mu) \le 0 \le \bsigma_K(\mu), 
\quad \text{where} \quad
k < K = \dim E_1.
$$
Let also $r_n$ be a sequence of positive numbers. 
Then there is a sequence of diffeomorphisms $g_n$ converging to $f$
such that for each~$n$:
\begin{itemize} 
\item $\gamma_n$ is a periodic orbit of $g_n$, and $g_n$ equals $f$ outside the $r_n$-neighborhood of~$\gamma_n$;
\item under $g_n$, the orbit $\gamma_n$ has exactly $k$ vanishing Lyapunov exponents and $d-k$ positive Lyapunov exponents. 
\end{itemize}
\end{lemm}

\begin{figure}[hbt]
\begin{center}
\includegraphics[scale=.4]{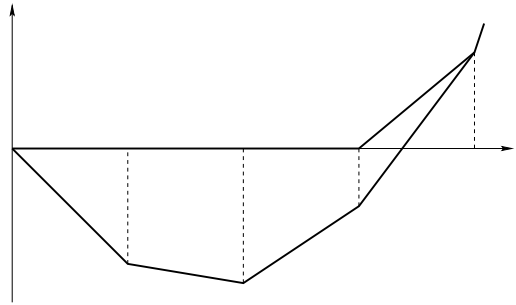}
\end{center}
\caption{\small An example in the situation of Lemma~\ref{l.vanishing} with $K = 4$ and $k = 3$.
The lower graph is $\bsigma(f,\mu)$ and the upper one is $\bsigma(g_n,\gamma_n)$.}
\label{f.universal}
\end{figure}

\begin{proof}
By Franks Lemma, it is sufficient to find out how to perturb the derivatives along the periodic orbits.
Define a graph $\sigma = (\sigma_0, \ldots, \sigma_d)$ by 
$$
\sigma_i = 
\begin{cases}
0                                &\text{if $i\le k$,} \\
\dfrac{i-k}{K-k}\bsigma_K(f,\mu) &\text{if $k \le i \le K$,} \\
\bsigma_i(f,\mu)                 &\text{if $i \ge K$.} 
\end{cases}
$$
Then $\sigma$ belongs to the set $\cG(\mu,\La)$ (defined in \S~\ref{sss.answer}).
Applying Theorem~\ref{t.measure intro 2} we find
a perturbation $g$ of $f$ preserving some $\gamma=\gamma_n$
so that $\bsigma(g,\gamma) = \sigma$.
If $\bsigma_K(f,\mu) > 0$ then we are done:
$\gamma$ has exactly $k$ vanishing Lyapunov exponents and $d-k$ positive Lyapunov exponents under $g$. 
If $\bsigma_K(f,\mu)=0$ then 
$\gamma$ has exactly $K$ vanishing Lyapunov exponents and $d-K$ positive Lyapunov exponents under $g$.
(because, by domination, $\bsigma_{K+1}(f,\mu)>0$).
Then (using Lemma~\ref{l.triangular} and Proposition~\ref{p.BGV})
we make another perturbation to make $K-k$ of these exponents positive.
\end{proof}

\begin{lemm}[(From zero exponents to identity map)]\label{l.identity} 
Let $f$ be a diffeomorphism with a periodic orbit $\gamma$ of period $\pi(\gamma)$.
Assume $\gamma$ has exactly $k$ vanishing Lyapunov exponents.
Then, for any $C^1$-neighborhood $\cU$ of $f$ 
and any neighborhood $V$ of $\gamma$, 
there exists a diffeomorphism $g\in\cU$ that equals $f$ outside $V$,
and there exists a $g$-periodic orbit $\tilde \gamma$ 
of period $\pi(\tilde \gamma) \ge \pi(\gamma)$
contained in $V$,
such that the following properties hold:
\begin{itemize}
\item For any $q \in \tilde \gamma$
there is a subspace $E\subset T_q M$ of dimension $k$ such that 
the restriction of $Dg^{\pi(\tilde \gamma)}$ to $E$ is the identity map;
\item The orbits $\gamma$ and $\tilde \gamma$ have the same Lyapunov spectra,
and thus the same Lyapunov graphs: $\bsigma(f,\gamma) = \bsigma(g, \tilde\gamma)$.
\end{itemize}
\end{lemm}

\begin{proof} 
Take local coordinates around each point in $\gamma$.
Up to performing an arbitrarily $C^1$ small perturbation of $f$, one may assume 
that it is linear in those coordinates.
Fix some $p\in\gamma$ and
let $E$ denote the subspace, in the coordinates around $p$, 
corresponding to the vanishing Lyapunov exponents. 
So $E$ is a locally $f$-invariant submanifold, 
and the restriction of $f^{\pi(p)}$ to $E$ is a linear map whose Lyapunov exponents are all zero. 
Now consider the following:

\begin{clai}
Let $L: \R^k \to \R^k$ be a linear map all whose eigenvalues lie in the unit circle.
Then there is an arbitrarily small perturbation $\tilde L$ of $L$  
that is diagonalizable over $\C$ and has
all its eigenvalues in the unit circle.
\end{clai}

Assume the Claim for a while and take the perturbation $\tilde L:E\to E$.
Notice that a power $\tilde L^n$ of it is the identity;
fix the least such $n \ge 1$.
By a procedure from \S~\ref{ss.procedures}, 
we can extend $\tilde L$ to a linear map $\hat L$ on the whole space 
whose Lyapunov spectrum is the same as  $Df^{\pi(p)}(p)$.
We can find a locally linear $C^1$-perturbation $g$ of $f$ 
such that the maps $g^{\pi(p)}$ and $\hat L$ coincide in a neighborhood of $p$.
Now we take any point $q$ on $E$ close to $p$ that has period $n$ under $\tilde L$,
and take $\tilde \gamma$ as $g$-orbit of $q$.
This shows the lemma, modulo proving the Claim.

The proof of the Claim is by induction on the dimension $k$.
The case $k=1$ is trivial: there is no need to perturb $L$.
Next consider the case $k=2$.
If $L$ is diagonalizable, but its eigenvalues are not roots of unity,
then we can perturb $L$ to make them so. 
On the other hand, if $L$ is not diagonalizable then 
either $L$ or $-L$ has Jordan form
$$
\begin{pmatrix}
1 & 1 \\ 0 & 1
\end{pmatrix}.
$$
Perturb this to
$$
\begin{pmatrix}
1-\eps & 1 \\ -\eps & 1
\end{pmatrix};
$$
for a suitably chosen small $\eps>0$, the eigenvalues are non-real roots of unity, and we are done.

Now take $k \ge 3$ and assume that the Claim has been proved for every dimension between $1$ and $k-1$.
Let $L: \R^k \to \R^k$ have all its eigenvalues in the unit circle.
First assume that $L$ has an invariant splitting $\R^k = F \oplus G$ into non-zero bundles.
Then we apply the induction hypothesis to 
the restrictions of $L$ to each subspace $F$ and $G$, 
and take $\tilde L$ as the product perturbation. 
In the remaining case where no such invariant splitting exists, 
the real Jordan normal form of $L$ has a single block.
Take the invariant subspace $F$ of dimension $2$
corresponding to the upper left corner of the Jordan matrix.
Then (repeating a previous reasoning)
we can perturb $L \restr F$ to make its eigenvalues non-real of modulus $1$,
and different from the eigenvalues of $L/F$.
By a procedure from \S~\ref{ss.procedures}, we can
extend the perturbation on $F$ to a perturbation on the whole space,
without altering eigenvalues.
The new linear map has an invariant splitting into non-zero bundles,
and we are reduced to the previous case.
This concludes the proof of the Claim and hence of the lemma.
\end{proof}


\begin{lemm}[(From identity to any map)]\label{l.anything} 
Let $f$ be a diffeomorphism with a periodic point $p$ of period $\pi(p)$.
Assume that there is a subspace $E\subset T_p M$ of dimension $k$ such that 
the restriction of $Df^{\pi(p)}$ to $E$ is the identity map.
Then, for any $C^1$-neighborhood $\cU$ of $f$ 
any neighborhood $V$ of the orbit of $p$, 
and any  $\phi \in \DDiff{k}$,
there exists 
a diffeomorphism $g\in\cU$ that equals $f$ outside $V$,
an embedded $k$-disk $D$, 
and $\pi \ge \pi(p)$
such that:
\begin{itemize}
\item $D$, $g(D)$, \ldots, $g^{\pi-1}(D)$ are pairwise disjoint, and $g^{\pi}(D)$ is contained in the (relative) interior of~$D$.
\item $D \cup g(D) \cup \cdots \cup g^{\pi-1}(D) \subset V$.
\item $D$ is normally hyperbolic for $g^{\pi}$.
\item The restriction of $g^{\pi}$ to $D$ is differentiably conjugate to $\phi$.
\end{itemize}
\end{lemm}

\begin{proof}
This is basically a reformulation of Proposition~3.1 from \cite{BD IHES}.
\end{proof}

\begin{lemm}[(Stably finest dominated splittings)] \label{l.stably finest}   
Let $\cV$ be a $C^1$-open set of diffeomorphisms $f$ 
having a hyperbolic periodic point $p_f$ varying continuously with $f$.
Let $E_{1,f}\dplus\cdots\dplus E_{m_f,f}$ 
indicate the finest dominated splitting over the chain recurrence class $C(p_f)$,
for each $f \in \cV$.
Then there is an open and dense subset $\cU \subset \cV$
where the functions $f \mapsto m_f$ and  $f\mapsto \dim E_{i,f}$ are locally constant.
\end{lemm}


\begin{proof}
As a consequence of Conley theorem, 
the map $f \in \cV \mapsto C(p_f)$ is upper semicontinuous.
So, given any dominated splitting on $C(p_f)$,
for every $g$ sufficiently close to $f$ the class $C(p_g)$
has a dominated splitting with the same number of bundles.
Hence the number of bundles $m_f$ in the finest dominated splitting on $C(p_f)$ is 
a lower semicontinuous function of $f$.
So this number is locally constant on a dense open subset $\cU$ of $\cV$.
In this set the dimensions of the bundles are also locally constant.
\end{proof}

The following lemma is a version of Theorem~\ref{t.criterion} for individual periodic orbits:

\begin{lemm}\label{l.criterion} 
Let $\cU$ be a $C^1$-open set of diffeomorphisms $f$ 
having a hyperbolic periodic point $p_f$ of index $k$, varying continuously with $f$,
so that the dimensions of the bundles of the finest dominated splitting
$E_{1}\dplus\cdots\dplus E_{m}$ over the chain recurrence class $C(p_f)$ do not depend on $f\in \cU$.

Assume that:
\begin{equation}\label{e.det}
|\det Df^{\pi(p_f)}\restr E_1(p_f)|>1,  \quad \text{for each $f \in \cU$.}
\end{equation}
Then there is a residual subset $\cR$ of $\cU$ such that every $f\in\cR$ has 
normally expanding $k$-universal dynamics. 
\end{lemm}

\begin{proof}
If $n \in \NN$ and $\cO$ is an open nonempty subset of $\DDiff{k}$,
let $\cV(n,\cO)$ indicate the set of $f \in \cU$ 
such that there is an embedded closed $k$-disk $D$ and $\pi \in \NN$
such that:
\begin{itemize}
\item[i.] 
$D$, $f(D)$, \ldots, $f^{\pi-1}(D)$ are pairwise disjoint, and 
$f^{\pi}(D)$ is contained in the (relative) interior of~$D$.
\item[ii.] 
$D$ is normally expanding for $f^{\pi}$.
\item[iii.] 
The restriction of $f^{\pi}$ to $D$ is differentiably conjugate to a map in $\cO$.
\item[iv.]
$D$ is contained in the $1/n$-neighborhood of $p_f$.
\item[v.]
$p_f \not \in \bigcup_{j=0}^{\pi-1} f^j(D)$.
\end{itemize}
We claim that $\cV(n,\cO)$ is open and dense in $\cU$.
Openness is obvious.
To show denseness, take any $f \in \cU$, and successively perturb it as follows:
\begin{enumerate}
\item\label{i.first step}
For every diffeomorphism in a residual subset $\cR_0 \subset \Diff^1(M)$,
the homoclinic classes are chain recurrence classes.
Perturb $f$ so that $f\in \cR_0$ and thus $H(p_f) = C(p_f)$.
Let $\mu$ be the invariant probability measure on the orbit of $p_f$.
By Lemma~\ref{l.class measure}, $(\mu, H(p_f)) \in \cX(f)$.
Since $p_f$ has index $k$, $\bsigma_k(\mu) < 0$,
and by assumption~\eqref{e.det}, $\bsigma_K(\mu) \ge 0$, where $K = \dim E_1$.

\item 
Having the measure $\mu$ at our disposal, 
we use Lemma~\ref{l.vanishing} to perturb $f$ again so that
there is a periodic orbit (obviously different from that of $p_f$)
with $k$ vanishing exponents and $d-k$ positive exponents.
Moreover, we can take this orbit $1/(2n)$-close to $H(p_f)$ in the Hausdorff distance.

\item 
Using Lemma~\ref{l.identity} and then Lemma~\ref{l.anything},
we perturb $f$ again and create the disk $D$ with properties (i)--(v) 
so that $f \in \cV(n,\cO)$. 
\end{enumerate}

Now consider a countable base of (nonempty) open sets $\cO_n$ for $\DDiff{k}$.

A first attempt to conclude the proof would be to define $\cR$ as $\bigcap \cV(n,\cO_n)$.
Then for any diffeomorphism in this set, we would be able to find a countable family of discs
satisfying all the requirements of normally expanding $k$-universal dynamics, 
\emph{except} for the disjointness between the disks.
To fix that problem, we proceed as follows.

\smallskip

Let $\hat{\cO}_n$ be the subset of $\DDiff{k}$ formed by the maps $\phi$
such that for each $i = 1, \ldots, n$ 
there exist a subdisk $D_i \subset \Int \DD_k$ and an integer $\pi_i>0$
with the following properties:
\begin{itemize}
\item $D_i$, $\phi(D_i)$, \ldots, $\phi^{\pi_i-1}(D_i)$ are pairwise disjoint, and 
$\phi^{\pi_i}(D_i) \subset \Int D_i$.
\item The restriction of $\phi^{\pi_i}$ to $D_i$ is differentiably conjugate to a map in $\cO_i$.
\end{itemize}
Obviously, $\hat{\cO}_n$ is nonempty and open.
Define the following residual subset of $\cU$:
$$
\cR = \bigcap_{n \in \NN} \cV(n, \hat{\cO}_n) \, .
$$

Take any $f \in \cR$. 
Let us show that $f$ has normally expanding $k$-universal dynamics.
For each $n$, since $f \in \cV(n, \hat{\cO}_n)$, 
there is a disk $\hat{D}_n$
so that properties (i)-(v) hold with $D=\hat{D}_n$, $\pi$ equal to some $\hat{\pi}_n$, 
and $\cO = \hat{\cO}_n$.
Let $\delta_n$ be the
distance between $p_f$ and the orbit of $\hat{D}_n$,
which is positive by condition~(v).
Define a subsequence $\{n_i\}$ recursively by
taking $n_1 = 1$ and  choosing $n_{i+1} > n_i$ so that
$$
\frac{1}{n_{i+1}} < \min \big( \delta_{n_1} , \delta_{n_2} , \ldots, \delta_{n_i} \big) \, .
$$
This guarantees that the orbit of $\hat{D}_{n_{i+1}}$ is disjoint from the orbits of 
$\hat{D}_{n_{1}}$, \ldots, $\hat{D}_{n_{i}}$.

For each $i$, the restriction of  $f^{\hat{\pi}_n}$ to $\hat{D}_{n_i}$ is differentiably conjugate to 
a map in $\hat{\cO}_{n_i}$.
Since $i \le n_i$, we can find a periodic subdisk $D_i \subset \hat{D}_{n_i}$ whose first return
is differentiably conjugate to an element of $\cO_i$.
Thus the family of disks $\{D_i\}$ has all the properties required 
for normally expanding $k$-universal dynamics,
concluding the proof.
\end{proof}

Let us make a remark that will be useful later (in \S~\ref{ss.dichotomy}):
Assumption~\eqref{e.det} in Lemma~\ref{l.criterion}
can be replaced by the following weaker condition:
\begin{equation}\label{e.residual}
\begin{gathered}
\text{For each $f$ in a residual subset $\cS \subset \cU$, there is a measure $\mu$ such that}\\
\text{$(\mu, H(p_f)) \in \cX(f)$ and $\bsigma_k(\mu) \le 0 \le \bsigma_K(\mu)$, where $K = \dim E_1$.}
\end{gathered}
\end{equation}
Indeed, the only part of the proof that requires modification is
step~\ref{i.first step} in the proof of denseness of $\cV(n,\cO)$:
Here we perturb $f$ so that $f\in \cR_0 \cap \cS$,
and now the measure $\mu$ is given a priori. 

\medskip

We need the following lemma from point-set topology:

\begin{lemm}\label{l.cover}
Let $B$ be a Baire space.
Let $B = \bigcup_n V_n$ be a countable pointwise finite\footnote{A cover of a set is called \emph{pointwise finite}
if each point belongs to only finitely many sets in the covering.}
cover of $B$ by open sets.
Suppose that $R_n$ is a residual subset of $V_n$, for each $n$.
Then $\bigcup_n R_n$ is a residual subset of $B$.
\end{lemm}

\begin{proof}
Write $R_n = \bigcap_{i\in \NN} U_{n,i}$,
where each $U_{n,i}$ is open and dense in $V_n$.
Pointwise finiteness implies that
$$
\bigcup_n R_n = \bigcap_i \bigcup_n \bigcap_{j=1}^i U_{n,j} \, , 
$$
which is clearly a residual subset of $B$.
\end{proof}

\begin{proof}[of Theorem~{\rm\ref{t.criterion}}]  
Fix $k$ throughout the proof.
If $f$ is a diffeomorphism and $p$ is a hyperbolic periodic point, let us say that 
the pair $(f,p)$ has \emph{property~$X$} if at least one of the following properties hold:
\begin{enumerate}
\item $f$ has normally expanding $k$-universal dynamics.
\item $|\det Df^{\pi(p)}(p) \restr E_1(p)| \le 1$, where $E_1$ is the first bundle on the finest dominated splitting 
on $C(p)$.
\end{enumerate}

For each $n$, consider the set $\cV_n$ of diffeomorphisms such that all periodic points of period $n$
are hyperbolic.
This is open and dense set.

For each $f \in \cV_n$, we can find an open set $\cV_n^f$, an integer $r_f$,
and continuous maps $p_1$, \ldots, $p_{r_f} : \cV_{n}^f \to M$
such that the periodic points of period $n$ of each $g \in \cV_{n}^f$
are precisely $p_1(g)$, \ldots, $p_{r_f}(g)$.
Consider the cover of $\cV_n$ by the sets $\cV_n^f$.
Since $\Diff^1(M)$ 
is paracompact (as every metric space) 
and Lindel\"of,
we can take a countable locally finite subcover,
say $\cV_n = \bigcup_i \cV_{n,i}$.

Apply Lemma~\ref{l.stably finest} to each $\cV_{n,i}$ and each periodic point $p_j$ 
obtaining an open dense subset $\cU_{n,i,j} \subset \cV_{n,i}$ where 
the dimensions of the bundles on the finest dominated splitting 
on the chain recurrence class of $p_j$ are locally constant.
Obviously, $\cU_{n,i,j}$ is the (disjoint) union of a finite family of sets $\cU_{n,i,j,\ell}$,
where in each of these sets the dimensions are constant.

It follows from Lemma~\ref{l.criterion} 
that for every $f$ in a residual subset $\cR_{n,i,j,\ell}$ of $\cU_{n,i,j,\ell}$,
the pair $(f, p_j(f))$ has property $X$.

Now define
$$
\cR = \bigcap_n \bigcup_i \bigcap_j \bigcup_\ell \cR_{n,i,j,\ell} \, .
$$
Using Lemma~\ref{l.cover}, we see that $\cR$ is a residual subset of $\Diff^1(M)$.
If $f \in \cR$ then every periodic point $p$ is hyperbolic and $(f,p)$ has property~$X$.
The theorem follows.
\end{proof}

\subsection{Proof of Theorem~{\rm\ref{t.dichotomy}}}\label{ss.dichotomy}

Let us begin with a lemma:

\begin{lemm}\label{l.stupid}
Let $\cU$ be a $C^1$-open set of diffeomorphisms $f$ 
having a hyperbolic periodic point $p_f$ of index $k$, varying continuously with $f$,
so that the dimensions of the bundles of the finest dominated splitting
$E_{1}\dplus\cdots\dplus E_{m}$ over the chain recurrence class $C(p_f)$ do not depend on $f\in \cU$.

Then there is a residual subset $\cS$ of $\cU$ such that every $f \in \cS$,
has (at least) one of the following properties:
\begin{enumerate}
\item \label{i.stupid a}
There are periodic points $q_n$ homoclinically related to $p_f$ 
such that
$$
\liminf_{n \to \infty} 
\frac{1}{\pi(q_n)} \log \big| \det Df^{\pi(q_n)} \restr E_1(q_n) \big| \ge 0 .
$$

\item \label{i.stupid b}
There is $\alpha>0$ and there is a neighborhood $\cV$ of $f$ contained in $\cU$ 
such that
for every $g \in \cV$ and every periodic point $q$ homoclinically related to $p_g$, 
we have 
$$
\frac{1}{\pi(q)} \log \big| \det Dg^{\pi(q)} \restr E_1(q) \big| \le -\alpha .
$$
\end{enumerate}
\end{lemm}

\begin{proof}
For each $n \in \NN$, let $\cA_n$ be the set of $f \in \cU$ such that 
there exists a periodic point $q$ homoclinically related to $p_f$
such that 
$$
\frac{1}{\pi(q)} \log \big| \det Df^{\pi(q)} \restr E_1(q) \big| > - \frac{1}{n} \, .
$$
This is evidently an open set.
Let $\cB_n = \Int (\cU \setminus \cA_n)$.
Then $\cA_n \cup \cB_n$ is open and dense in $\cU$.
Taking the intersection over $n$, we obtain a residual subset $\cS$ of $\cU$.

Now take $f \in \cS$. 
If $f \in \cB_n$ for some $n$ then $f$ has property~\ref{i.stupid b} with $\alpha = 1/n$.
If, on the contrary, $f \not\in \bigcup_n \cB_n$ then $f \in \bigcap_n \cA_n$
and so $f$ has property~\ref{i.stupid a}.
\end{proof}

In order to prove Theorem~\ref{t.dichotomy},
we first obtain a version of it for individual periodic orbits:

\begin{lemm}\label{l.dichotomy}
Let $\cU$ be a $C^1$-open set of diffeomorphisms $f$ 
having a hyperbolic periodic point $p_f$ of index $k$, varying continuously with $f$,
so that the dimensions of the bundles of the finest dominated splitting
$E_{1}\dplus\cdots\dplus E_{m}$ over the chain recurrence class $C(p_f)$ do not depend on $f\in \cU$.

Then there is a residual subset $\cR$ of $\cU$ such that every $f \in \cR$
has (at least) one of the following properties:
\begin{enumerate}
\item \label{i.dich a}
$f$ is normally expanding $k$-universal; or:
\item \label{i.dich b}
There is $\alpha > 0$ such that 
for any $q$ homoclinically related with $p$, 
$$
\frac{1}{\pi(q)} \log \big| \det Df^{\pi(q)} \restr E_1(q) \big| < -\alpha .
$$
\end{enumerate}
\end{lemm}

\begin{proof}
Take a set $\cU$ as in the statement.
Let $\cA$, resp.\ $\cB$, be the set of $f\in \cU$ 
that have property \ref{i.stupid a}, resp.\ \ref{i.stupid b}, from Lemma~\ref{l.stupid}.
Then $\cB$ is open, and  $\cA \cup \cB$ contains a residual subset $\cS$ of $\cU$.

\begin{clai}
If $f \in \cA$ then there is a measure $\mu$ 
such that
$(\mu, H(p_f)) \in \cX(f)$ and 
$\bsigma_k(\mu) \le 0 \le \bsigma_K (\mu)$, where $K = \dim E_1$.
\end{clai}

\begin{proof}
We know that there are periodic points $q_n$ homoclinically related to $p_f$  
such that $\liminf \bsigma_K (q_n) \ge 0$.
For each $n$, we use Lemma~\ref{l.class measure} and find a periodic point $\hat{q}_n$ 
homoclinically related to $q_n$ and hence to $p_f$
such that $\bsigma_K(\hat{q}_n) > \bsigma_K (q_n) - 1/n$
and the Hausdorff distance between the orbit of $\hat{q}_n$ and $H(p_f)$ is less than $1/n$.
Let $\mu_n$ be the $g$-invariant probability measure supported on the orbit of $\hat{q}_n$,
and let $\mu$ be a accumulation point of this sequence of measures.
Since $\bsigma_K$ is continuous, we have $\bsigma_K(\mu) \ge 0$.
Each  $\hat{q}_n$ has index $k$, and $\bsigma_k$ is lower semicontinuous, 
therefore $\bsigma_k(\mu) \le 0$,
proving the claim.
\end{proof}

Let $\cA^* = \cU \setminus \bar{\cB}$.
Then $\cA$ is residual in the open set $\cA^*$.
Applying Lemma~\ref{l.criterion} (with assumption~\eqref{e.det} replaced by \eqref{e.residual}) to $\cA^*$,
we conclude that there is a residual subset $\cR^*$ of $\cA^*$ 
formed by normally expanding $k$-universal diffeomorphisms.
Therefore $\cR = \cR^* \cup \cB$ is the residual set we were looking for.
\end{proof}

\begin{proof}[of Theorem~{\rm\ref{t.dichotomy}}]  
Fix $k$.
If $f$ is a diffeomorphism and $p$ is a hyperbolic periodic point, let us say that 
the pair $(f,p)$ has \emph{property $Y$} if at least one of properties 
\ref{i.dich a} or \ref{i.dich b} from Lemma~\ref{l.dichotomy} holds.

Then we follow word for word  the proof of Theorem~\ref{t.criterion},
just replacing property~$X$ by property~$Y$, and using Lemma~\ref{l.dichotomy} instead of Lemma~\ref{l.criterion}.
\end{proof}

\subsection{Criterion for \texorpdfstring{$k$}{k}-Universality}    

Let us give a criterion for $k$-universal dynamics somewhat similar to Theorem~\ref{t.criterion}:

\begin{theo}\label{t.non-free}
Let $f$ be a $C^1$ generic diffeomorphism.
Let $p$ be a periodic point, and let 
$E_1\dplus\cdots \dplus E_m$ be the finest dominated splitting on the homoclinic class $H(p)$.
Denote $i_j = \dim(E_1\oplus\cdots\oplus E_j)$ for $j \in \{1,\dots,m\}$, and $i_0=0$. 
Let
\begin{equation}\label{e.k}
k = \# \big\{ i \in \{1,\ldots,d\} ; \; \bsigma_i(p) \le s \big\},  \quad \text{where} \quad
s = \min_{j \in \{0,\dots,m\}}  \bsigma_{i_j}(p) \, .
\end{equation}
Then generic diffeomorphisms in a neighborhood of $f$ have $k$-universal dynamics. 
\end{theo}

\begin{proof}[(sketch)]
First notice that if $k$ is given by \eqref{e.k} then
it is possible to perturb $f$ so to create periodic orbits 
with exactly $k$ vanishing Lyapunov exponents.

As in the proof of Theorem~\ref{t.criterion}, it is sufficient to prove a version of the theorem
for individual orbits.
This is done making minor adaptations in the proof of Lemma~\ref{l.criterion}.
\end{proof}

\begin{ques}
Can one find similar criteria for free, but neither normally expanding nor normally contracting, $k$-universality?
\end{ques}

\subsection{Proof of Theorem~{\rm \ref{t.BD}}} \label{ss.BD}

For completeness, we now explain how Theorem~\ref{t.BD} follows from \cite{BD IHES}.
Since we haven't used this theorem, this part is independent from anything else in this paper.

\medskip

The main result of \cite{BD IHES} says that 
if a diffeomorphism $g$ has a homoclinic class $H$ that is robustly without dominated splitting,
and $H$ contains two homoclinically related points, 
one with jacobian bigger than $1$ and the other with jacobian less than $1$,
then $g$ is in the closure of a locally generic set 
formed by diffeomorphisms with universal dynamics.
Examples of such diffeomorphisms $g$ can be constructed 
in any manifold of dimension $k\ge 3$;
in fact they can be constructed in a $k$-disc
and be taken close to the identity.

\begin{proof}[of Theorem~{\rm\ref{t.BD}}]  
Let $p$ be a periodic point for $f$
such that $Df^{\pi(p)}(p)$ is the identity on a 
subspace $E \subset T_p M$ of dimension $k \ge 3$,
and the other eigenvalues have modulus bigger than $1$.
With a $C^1$ perturbation supported on a small neighborhood of the orbit of $p$, 
we can create a normally expanding periodic $k$-disc $D$
such that $g = f^{\pi(p)} \restr D$ is the identity.
With a new perturbation, $g$ satisfies the conditions from\cite{BD IHES} explained above.
It follows that the perturbed $f$ belongs to the closure of a locally generic set 
of diffeomorphisms with normally expanding $k$-universal dynamics.
\end{proof}


\begin{acknowledgements}
We thank Flavio Abdenur, Lorenzo D\'iaz, Nicolas Gourmelon, Rafael Potrie and Jiagang Yang 
for helpful discussions.
We also thank Oliver Jenkinson for telling us about majorization
and the referee for several suggestions that helped to improve the writing.
\end{acknowledgements}


\affiliationone{
	Jairo Bochi\\
	Departamento de Matem\'atica, PUC--Rio\\ 
	Rua Mq.\ S.\ Vicente, 225\\
	22453-900 Rio de Janeiro, RJ\\
	Brazil
	\email{jairo@mat.puc-rio.br}
}
\affiliationtwo{
	Christian Bonatti\\
	IMB, Universit\'e de Bourgogne\\
	B.P. 47 870\\
	21078 Dijon Cedex\\
	France
	\email{bonatti@u-bourgogne.fr}
}

\end{document}